\documentclass[12pt,british]{amsart}
\usepackage{charter}
\usepackage[T1]{fontenc}
\usepackage[latin9]{inputenc}
\usepackage{geometry}
\usepackage{babel}
\usepackage{mathrsfs}
\usepackage{enumitem}
\usepackage{amstext}
\usepackage{amsthm}
\usepackage{amssymb}
\usepackage{setspace}
\usepackage[unicode=true,pdfusetitle,
 bookmarks=true,bookmarksnumbered=false,bookmarksopen=false,
 breaklinks=false,pdfborder={0 0 1},backref=false,colorlinks=false]
 {hyperref}

\makeatletter
\numberwithin{equation}{section}
\numberwithin{figure}{section}
\newlength{\lyxlabelwidth}      
\theoremstyle{plain}
\newtheorem{thmx}{\protect\theoremname}

\theoremstyle{plain}
\newtheorem{corx}[thmx]{Corollary}
\theoremstyle{plain}
\newtheorem{thm}{\protect\theoremname}[section]
\theoremstyle{plain}
\newtheorem{lem}[thm]{\protect\lemmaname}
\theoremstyle{definition}
\newtheorem{observation}[thm]{\protect\observationname}
\theoremstyle{definition}
\newtheorem{standinghypothesis}[thm]{\protect\standinghypothesisname}
\theoremstyle{plain}
\newtheorem{prop}[thm]{\protect\propositionname}
	\newenvironment{elabeling}[2][]%
	{\settowidth{\lyxlabelwidth}{#2}
		\begin{description}[font=\normalfont,style=sameline,
			leftmargin=\lyxlabelwidth,#1]}
	{\end{description}}
\theoremstyle{remark}
\newtheorem{rem}[thm]{\protect\remarkname}
\theoremstyle{remark}
\newtheorem*{rem*}{\protect\remarkname}
\theoremstyle{definition}
\newtheorem{defn}[thm]{\protect\definitionname}
\theoremstyle{plain}
\newtheorem{question}[thm]{Question}
\theoremstyle{plain}
\newtheorem{cor}[thm]{\protect\corollaryname}
\theoremstyle{definition}
\newtheorem{example}[thm]{\protect\examplename}

\usepackage{eucal}
\let\mathcal=\CMcal
\usepackage{dsfont}
\usepackage{graphicx}
\usepackage{mathtools} 

\let\ldash=\l

\def\cst{C$^{*}$}

\setlist[enumerate,1]{label={(\arabic*)},ref={(\arabic*)}}
\setlist[enumerate,2]{label={\alph*.},ref={\alph*}}

\makeatother

\providecommand{\corollaryname}{Corollary}
\providecommand{\definitionname}{Definition}
\providecommand{\examplename}{Example}
\providecommand{\lemmaname}{Lemma}
\providecommand{\observationname}{Observation}
\providecommand{\propositionname}{Proposition}
\providecommand{\remarkname}{Remark}
\providecommand{\standinghypothesisname}{Standing Hypothesis}
\providecommand{\theoremname}{Theorem}

\begin{document}
\global\long\def\e{\varepsilon}%
\global\long\def\N{\mathbb{N}}%
\global\long\def\Z{\mathbb{Z}}%
\global\long\def\Q{\mathbb{Q}}%
\global\long\def\R{\mathbb{R}}%
\global\long\def\C{\mathbb{C}}%
\global\long\def\H{\mathcal{H}}%
\global\long\def\K{\mathcal{K}}%

\global\long\def\linspan{\operatorname{span}}%
\global\long\def\clinspan{\operatorname{\overline{span}}}%
\global\long\def\supp{\operatorname{supp}}%
\global\long\def\op{\mathrm{op}}%

\global\long\def\Rant{\text{\scalebox{1.15}{\ensuremath{\mathtt{R}}}}}%
\global\long\def\Sant{\text{\scalebox{1.15}{\ensuremath{\mathtt{S}}}}}%
\global\long\def\QG{\mathbb{G}}%
\global\long\def\G{\QG}%
\global\long\def\Cz{\mathrm{C}_{0}}%
\global\long\def\CzU{\mathrm{C}_{0}^{\mathrm{u}}}%
\global\long\def\CStar{\mathrm{C}^{*}}%
\global\long\def\VN{\operatorname{VN}}%
\global\long\def\Mor{\operatorname{Mor}}%
\global\long\def\M{\mathrm{M}}%
\global\long\def\Aut{\operatorname{Aut}}%
\global\long\def\CStarR{\mathrm{C}_{\mathrm{r}}^{*}}%
\global\long\def\tensormin{\otimes_{\mathrm{min}}}%
\global\long\def\tensorn{\mathbin{\overline{\otimes}}}%
\global\long\def\id{\mathrm{id}}%
\global\long\def\i{\id}%
\global\long\def\Cb{\mathrm{C}_{b}}%
\global\long\def\Cc{\mathrm{C}_{c}}%
\global\long\def\Ad{\operatorname{Ad}}%
\global\long\def\ev{\mathrm{ev}}%
\global\long\def\a{\alpha}%
\global\long\def\one{\mathds{1}}%
\global\long\def\d{\,\mathrm{d}}%
\global\long\def\dd{\mathrm{d}}%
\global\long\def\tensor{\otimes}%
\global\long\def\ot{\tensor}%
\global\long\def\Rieffel#1#2{#1^{#2}}%
\global\long\def\GPsi{\Rieffel{\G}{\Psi}}%
\global\long\def\GammaPsi{\Rieffel{\Gamma}{\Psi}}%
\global\long\def\GammaPsi{\Gamma^{\Psi}}%
\global\long\def\Linfty{L^{\infty}}%
\global\long\def\Lone{L^{1}}%
\global\long\def\Ltwo{L^{2}}%
\global\long\def\LoneS{L_{\sharp}^{1}}%
\global\long\def\Ww{\mathds{W}}%
\global\long\def\wW{\text{\reflectbox{\ensuremath{\Ww}}}\:\!}%

\global\long\def\pres#1#2#3{\prescript{#1}{#2}{#3}}%
\global\long\def\netdirset#1{\EuScript{#1}}%

\title[Convolution semigroups on Rieffel deformations]{Convolution semigroups on Rieffel deformations of locally compact
quantum groups}
\author{Adam Skalski}
\address{Institute of Mathematics of the Polish Academy of Sciences, ul.~\'Sniadeckich
8, 00--656 Warszawa, Poland}
\email{a.skalski@impan.pl}
\author{Ami Viselter}
\address{Department of Mathematics, University of Haifa, 3103301 Haifa, Israel}
\email{aviselter@univ.haifa.ac.il}


\begin{abstract}
Consider a locally compact quantum group $\mathbb{G}$ with a closed
classical abelian subgroup $\Gamma$ equipped with a $2$-cocycle
$\Psi:\hat{\Gamma}\times\hat{\Gamma}\to\mathbb{C}$. We study in detail
the associated Rieffel deformation $\mathbb{G}^{\Psi}$ and establish
a canonical correspondence between $\Gamma$-invariant convolution
semigroups of states on $\mathbb{G}$ and on $\mathbb{G}^{\Psi}$.
\end{abstract}

\maketitle

\section{Introduction and general preliminaries}

The study of the Rieffel deformation in the sense of \cite{Rieffel_deform}
in the context of compact quantum groups goes back at least to the
article \cite{Wang__deformation}. At the same time an analogous deformation
was also studied for Kac algebras in the sense of Enock and Schwartz
in \cite{Enock_Vainerman__deform_Kac_alg_abel_grp}, where it was
noted that the Rieffel deformation, at least formally, can be dually
realised as a special case of the Hopf algebraic cocycle twist. This
perspective was later investigated in depth in the context of general
locally compact quantum groups in \cite{Kasprzak__Rieffel_deform_crossed_prod}
(available as a preprint since June '06), which first recognised the
connection between the Rieffel deformation of (quantum) groups and
Landstad's theory for crossed products \cite{Landstad__dual_thy_cov_sys}.
A similar approach was adopted in \cite{Fima_Vainerman__twist_Rieffel_deform}
(available as a preprint since January '08). It is worth noting that
\cite{Kasprzak__Rieffel_deform_crossed_prod} focuses on the \cst-algebraic
picture and takes as ingredients of the construction a locally compact
group $G$ equipped with a closed abelian subgroup $\Gamma$ and a
continuous $2$-cocycle $\Psi$ on $\hat{\Gamma}$, satisfying mild
additional assumptions, whereas \cite{Fima_Vainerman__twist_Rieffel_deform}
is concerned with the von Neumann algebraic formulation and its initial
data is a general locally compact quantum group $\G$ with a closed
abelian subgroup $\Gamma$ and a bicharacter $\Psi$ on $\hat{\Gamma}$,
again with some additional technical conditions satisfied. In each
case the authors produce a locally compact quantum group $\GPsi$
in the sense of \cite{Kustermans_Vaes__LCQG_C_star}. The need for
the extra assumptions was related to the fact that the cocycle twisting
of locally compact quantum groups was at the time not yet fully understood.
This changed with the groundbreaking work \cite{DeCommer__Galois_obj_cocycle_twist_LCQG},
where it was shown that every cocycle twist of a locally compact quantum
group yields another locally compact quantum group.

In this paper we address the question of constructing quantum convolution
semigroups of states on Rieffel deformations. Quantum convolution
semigroups of states on locally compact quantum groups correspond
(in the probabilistic sense) to quantum L\'evy processes, and we
refer to \cite{Skalski_Viselter__convolution_semigroups,Skalski_Viselter__generating_functionals}
for an introduction to the concept and its various incarnations and
applications. It should be noted that cocycle twisting was exploited
as a means of constructing quantum convolution semigroups already
in \cite[Subsection~5.2]{Skalski_Viselter__convolution_semigroups},
but there -- in comparison to this work -- we were considering directly
the objects associated to the cocycle twist, and not to its dual.
This means in particular that the examples in \cite[Subsection~5.2]{Skalski_Viselter__convolution_semigroups}
should be conceptually rather viewed as deformations of semigroups
of Herz--Schur multipliers, and not of classical convolution semigroups
of measures.

To achieve our aim, we formally put the Rieffel deformation of \cite{Kasprzak__Rieffel_deform_crossed_prod,Fima_Vainerman__twist_Rieffel_deform}
under the umbrella of cocycle twisting in full generality, exploiting
De~Commer's results from \cite{DeCommer__Galois_obj_cocycle_twist_LCQG}.
We choose to work with the following setup, later called the \emph{Standing
Hypothesis}: we start with a locally compact quantum group $\QG$
equipped with a closed abelian subgroup $\Gamma$ and a continuous
$2$-cocycle $\Psi$ on $\hat{\Gamma}$ (with no extra assumptions).
This context has one clear advantage: the deformation procedure becomes
a fully symmetric operation, namely the new locally compact quantum
group $\GPsi$ again admits $\Gamma$ as a subgroup, and $\QG$ is
itself a Rieffel deformation of $\GPsi$. 

Thus we can formulate the main result of our paper is the following
theorem.
\begin{thmx}
\label{thmx}Let $\G$ be a locally compact quantum group with a closed
classical abelian subgroup $\Gamma$. Let $\Psi:\hat{\Gamma}\times\hat{\Gamma}\to\C$
be a $2$-cocycle on $\hat{\Gamma}$. Then there is a canonical one-to-one
correspondence between $w^{*}$-continuous convolution semigroups
of states invariant under the adjoint action of $\Gamma$ on $\G$
and on its Rieffel deformation $\GPsi$.
\end{thmx}

If $\G$ is in addition co-amenable, we can show that the correspondence
above preserves the \emph{symmetry} (i.e.~unitary antipode invariance)
of the relevant convolution semigroups.

As mentioned before the formulation of the theorem, the main motivation
behind developing such results comes from the desire to use a classical
locally compact group $G$ with a convolution semigroup of probability
measures to produce a genuinely quantum stochastic evolution on the
deformed locally compact \emph{quantum} group $\Rieffel G{\Psi}$.
And indeed, we obtain as a very special case of the above result the
following corollary. 
\begin{corx}
\label{corx}Let $G$ be a locally compact group with a closed abelian
subgroup $\Gamma$ and let $\Psi:\hat{\Gamma}\times\hat{\Gamma}\to\C$
be a $2$-cocycle on $\hat{\Gamma}$. Every (symmetric) $w^{*}$-continuous
convolution semigroup of probability measures ${(\mu_{t})}_{t\geq0}$
on $G$ that is invariant under the adjoint action of $\Gamma$ (so
for example is supported on $Z_{G}(\Gamma)$, the centraliser of $\Gamma$
inside $G$) determines in a canonical way a (symmetric) $w^{*}$-continuous
convolution semigroup ${(\mu_{t}^{\Psi})}_{t\geq0}$ of states on
the Rieffel deformation $\Rieffel G{\Psi}$.
\end{corx}

We also observe that given the assumption of the above corollary one
also has a natural realisation of $Z_{G}(\Gamma)$ as a closed subgroup
of $\Rieffel G{\Psi}$. This means that if we start from a convolution
semigroup of measures supported on $Z_{G}(\Gamma)$ there is an alternative,
easier way of producing the semigroup ${(\mu_{t}^{\Psi})}_{t\geq0}$.
We would like to stress however that even in this case we obtain a
non-trivial quantum stochastic evolution on the deformed quantum group. 

The detailed plan of the paper is as follows. After preparing basic
notation and terminology and making a few remarks on non-degeneracy
in the last part of the introduction, in Section~\ref{sec:LCQGCT}
we present the locally compact quantum group preliminaries and describe
the cocycle twisting in this analytic context, focusing on cocycles
coming from classical abelian subgroups and presenting in this case
an explicit formula for the multiplicative unitary of the twist. Section~\ref{sec:Rieffel_deform},
which is potentially of independent value, formalises the dual relationship
between the \cst-algebraic Rieffel deformation and the cocycle twist,
all in the framework of Standing~Hypothesis~\ref{stand_hyp}. We
largely follow the approach of \cite{Kasprzak__Rieffel_deform_crossed_prod},
but the greater generality forces us to improve several arguments.
Here we also study a few quantum group aspects of the Rieffel deformation,
and show that the setup with which we choose to work allows us to
view the Rieffel deformation as an `invertible' transformation.
Indeed, a simple computation shows that the original Hopf--von Neumann
algebraic object is itself a cocycle twist of the twisted one, and
we discuss the corresponding fact on the level of \cst-algebraic
Rieffel deformations. Finally in Section~\ref{sec:convsemig} we
introduce convolution semigroups of states, provide a general construction
of Rieffel deformations of well-behaved completely positive maps,
and apply it together with the results of Section~\ref{sec:Rieffel_deform}
to obtain the main theorems of the paper. In particular Theorem~\ref{thmx}
above follows from Theorem~\ref{thm:mainconv}, and Corollary~\ref{corx}
follows from Theorems~\ref{thm:mainconv} and \ref{thm:symmetry}.
Here we also discuss in Example~\ref{ex:E(2)} a concrete application
of Corollary~\ref{corx}. Each subsection is based on its own individual
assumptions stated clearly in its beginning. 

Lastly, we remark that the ultimate generalisation of the results
of this paper would concern the Rieffel deformation with respect to
an \emph{arbitrary} $2$-cocycle $\Omega$ on $\hat{\G}$, not necessarily
`living on' the dual of an abelian subgroup. This would involve
considering the dual $\Rieffel{\G}{\Omega}:=\widehat{\hat{\G}_{\Omega}}$
of the cocycle twist $\hat{\G}_{\Omega}$ and proving that either
$\Cz(\Rieffel{\G}{\Omega})$ or $\Linfty(\Rieffel{\G}{\Omega})$ can
be seen as a crossed product-based deformation of $\Cz(\G)$ or $\Linfty(\G)$,
respectively, in a way that allows writing a convenient formula for
the co-multiplication of $\Rieffel{\G}{\Omega}$ based on that of
$\G$ (analogously to Theorems~\ref{thm:C0_G_Psi_crossed_prod} and
\ref{thm:Psi_comult_crossed_prod}). The starting point of such a
research would probably be \cite{Neshveyev_Tuset__deform_C_alg_cocycle_LCQG},
which has already found a presentation of $\Cz(\Rieffel{\G}{\Omega})$
in that spirit. The specific construction presented in our paper relies
heavily on the $2$-cocycle coming from an abelian subgroup, and probably
cannot be pushed further.

\subsection{Notation and conventions}

The unit of a unital \cst-algebra is denoted by $\one$. For a \cst-algebra
$A$ we write $\M(A)$ for the multiplier algebra of $A$ and denote
the unitary group of $\M(A)$ by $\mathcal{U}(A)$. Further, given
$u\in\mathcal{U}(A)$ we write $\Ad(u)$ for the inner automorphism
of $A$ given by $x\mapsto uxu^{*}$. For \cst-algebras $A,B$, a
non-degenerate $*$-homomorphism from $A$ to $\M(B)$ is called a
\emph{morphism} from $A$ to $B$ and the space of these is denoted
by $\Mor(A,B)$. More generally, every \emph{strict}, and in particular
every \emph{non-degenerate}, completely positive map from $A$ to
$\M(B)$ \cite[p.~49]{Lance} has a (unique) extension to a completely
positive map from $\M(A)$ to $\M(B)$ that is strictly continuous
on the unit ball \cite[Corollary~5.7]{Lance}, and we usually do not
distinguish between these two maps.

Let $\Gamma$ be a locally compact group. All $L^{p}$-spaces of $\Gamma$
are with respect to the left Haar measure. We write $\lambda^{\Gamma}=\left(\lambda_{\gamma}^{\Gamma}\right)_{\gamma\in\Gamma}$
for the left regular representation of $\Gamma$ on $\Ltwo(\Gamma)$,
namely $(\lambda_{\gamma}^{\Gamma}f)(\gamma')=f(\gamma^{-1}\gamma')$
for $f\in\Ltwo(\Gamma)$ and $\gamma,\gamma'\in\Gamma$. By a (unitary)
representation of $\Gamma$ inside a \cst-algebra $A$ we mean a
strictly continuous homomorphism $\lambda=\left(\lambda_{\gamma}\right)_{\gamma\in\Gamma}$
from $\Gamma$ into the unitary group $\mathcal{U}(A)$. To every
such $\lambda$ there is an associated morphism from the full group
\cst-algebra $\CStar(\Gamma)$ to $A$ mapping the representative
of $f\in\Cc(\Gamma)$ in $\CStar(\Gamma)$ to $\int_{\Gamma}f(\gamma)\lambda_{\gamma}\d\gamma$
(strict integration). Slightly abusing the notation we will also denote
the latter morphism by $\lambda$. By an action of $\Gamma$ on a
\cst-algebra $A$ we mean a continuous homomorphism $\Gamma\to\Aut(A)$,
where $\Aut(A)$ is endowed with the point--norm topology as customary. 

A (measurable) \emph{$2$-cocycle} on $\Gamma$ is a measurable function
$\Psi$ from $\Gamma\times\Gamma$ to the complex unit circle $\mathbb{T}$
satisfying the $2$-cocycle condition 
\begin{equation}
\Psi(\gamma_{1},\gamma_{2})\Psi(\gamma_{1}\gamma_{2},\gamma_{3})=\Psi(\gamma_{2},\gamma_{3})\Psi(\gamma_{1},\gamma_{2}\gamma_{3})\qquad(\forall_{\gamma_{1},\gamma_{2},\gamma_{3}\in\Gamma}).\label{eq:func_2_cocycle}
\end{equation}
We mostly require $\Psi$ to be continuous, as well as normalised
in the sense that $\Psi(e,\cdot)\equiv1\equiv\Psi(\cdot,e)$. Under
the normalisation condition one has 
\begin{equation}
\Psi(\gamma_{1}^{-1},\gamma_{1}\gamma_{2})=\overline{\Psi(\gamma_{1},\gamma_{2})}\Psi(\gamma_{1}^{-1},\gamma_{1})\qquad(\forall_{\gamma_{1},\gamma_{2}\in\Gamma}).\label{eq:normal_2_cocycle_U}
\end{equation}
If $\Psi'$ is another $2$-cocycle on $\Gamma$, then the function
$\Psi\boxtimes\Psi':\Gamma^{2}\times\Gamma^{2}\to\mathbb{T}$ given
by 
\begin{equation}
(\Psi\boxtimes\Psi')\big((\gamma_{1},\gamma_{1}'),(\gamma_{2},\gamma_{2}')\big):=\Psi(\gamma_{1},\gamma_{2})\Psi'(\gamma_{1}',\gamma_{2}'),\qquad(\gamma_{1},\gamma_{1}'),(\gamma_{2},\gamma_{2}')\in\Gamma^{2},\label{eq:cocycle_tensor_prod}
\end{equation}
is a $2$-cocycle on $\Gamma^{2}$. 

Assume from now on that $\Gamma$ is abelian. We follow the convention
of \cite{Pedersen__book,Kasprzak__Rieffel_deform_crossed_prod} by
using the \emph{inverse} Fourier transform on $\Gamma$, that is,
letting $\hat{f}(\hat{\gamma}):=\int_{\Gamma}\left\langle \hat{\gamma},\gamma\right\rangle f(\gamma)\d\gamma$
for $f\in\Cc(\Gamma)$ and $\hat{\gamma}\in\hat{\Gamma}$ (so $\left\langle \hat{\gamma},\gamma\right\rangle $
instead of $\overline{\left\langle \hat{\gamma},\gamma\right\rangle }$).
The map that sends the representative of $f\in\Cc(\Gamma)$ in $\CStarR(\Gamma)$
to $\hat{f}\in\Cz(\hat{\Gamma})$ extends to a canonical \cst-algebra
isomorphism $\CStarR(\Gamma)\cong\Cz(\hat{\Gamma})$; the canonical
strict extension of its inverse, for each $\gamma\in\Gamma$, sends
$\gamma$ treated as a function on $\hat{\Gamma}$ that belongs to
$\Cb(\hat{\Gamma})\cong\M(\Cz(\hat{\Gamma}))$ to $\lambda_{\gamma}^{\Gamma}$
in $\M(\CStarR(\Gamma))$. The isomorphism $\Cz(\hat{\Gamma})\cong\CStarR(\Gamma)$
is a restriction of the spatial isomorphism $\Linfty(\hat{\Gamma})\cong\VN(\Gamma)$
obtained by conjugation with the $\Ltwo(\Gamma)\to\Ltwo(\hat{\Gamma})$
Fourier transform.

If $\lambda$ is a representation of (the abelian) $\Gamma$ inside
a \cst-algebra $A$, we treat the associated map $\lambda\in\Mor(\CStarR(\Gamma)\cong\CStar(\Gamma),A)$
also as a map $\lambda\in\Mor(\Cz(\hat{\Gamma}),A)$ using the above
isomorphism $\Cz(\hat{\Gamma})\cong\CStarR(\Gamma)$. In other words,
$\lambda(\hat{f})=\int_{\Gamma}f(\gamma)\lambda_{\gamma}\d\gamma$
for all $f\in\Cc(\Gamma)$.

\subsection{A few remarks on non-degeneracy of inclusions and completely positive
maps}

Non-degeneracy of $*$-homomorphisms and, more generally, of completely
positive maps plays an important role in this paper. We finish the
introductory section with a few remarks on this subject. Recall that
if $A,B$ are \cst-algebras and $S:A\to\M(B)$ is a completely positive
map, then $S$ is said to be \emph{non-degenerate} if for some/every
approximate identity $\left(e_{i}\right)_{i\in\netdirset I}$ of $A$
the net $\left(Se_{i}\right)_{i\in\netdirset I}$ converges strictly
to $\one_{\M(B)}$ in $\M(B)$. 
\begin{lem}
\label{lem:nondeg_CP_Banach_alg_approx_id}Let $A,B$ be \cst-algebras
and $S:A\to\M(B)$ be a contractive completely positive map. If there
exists a contractive Banach algebraic approximate identity $\left(a_{i}\right)_{i\in\netdirset I}$
of $A$ such that $\left(Sa_{i}\right)_{i\in\netdirset I}$ converges
strictly to $\one_{\M(B)}$ in $\M(B)$, then $S$ is non-degenerate.
\end{lem}

\begin{proof}
We should show that $\left(a_{i}\right)_{i\in\netdirset I}$ can be
replaced by a \cst-algebraic approximate identity.

Since $\left(Sa_{i}\right)_{i\in\netdirset I}$ converges strictly
to $\one_{\M(B)}$, for every state $\omega$ of $B$ we have $\omega(Sa_{i})\xrightarrow[i\in\netdirset I]{}1$
by a standard argument involving Cohen's factorisation theorem, hence
also $\omega(S\left|a_{i}\right|)\xrightarrow[i\in\netdirset I]{}1$
by the Cauchy--Schwarz inequality as $\left(a_{i}\right)_{i\in\netdirset I}$
is contractive and $\omega\circ S$ is also a state for $S$ is contractive.

We may and will assume that $\left\Vert a_{i}\right\Vert <1$ for
each $i\in\netdirset I$. This allows us to replace $\left(\left|a_{i}\right|\right)_{i\in\netdirset I}$
by an increasing net as follows. Consider the order-preserving isomorphism
$\Phi:\left\{ a\in A_{+}:\left\Vert a\right\Vert <1\right\} \to A_{+}$
given by $a\mapsto a(\one-a)^{-1}$. Let $\netdirset J$ be the directed
set of all finite subsets of $\netdirset I$ ordered by inclusion.
Since ${\big(e_{j}':=\sum_{i\in j}\Phi\left(\left|a_{i}\right|\right)\big)}_{j\in\netdirset J}$
is an increasing net in $A_{+}$ satisfying $\Phi\left(\left|a_{i}\right|\right)\le e_{j}'$
for all $i\in j\in\netdirset J$, the family ${\big(e_{j}:=\Phi^{-1}(e_{j}')\big)}_{j\in\netdirset J}$
is an increasing net in $\left\{ a\in A_{+}:\left\Vert a\right\Vert <1\right\} $
satisfying $\left|a_{i}\right|\le e_{j}$ for all $i\in j\in\netdirset J$. 

Consequently, for every state $\omega$ of $B$, from the fact that
$\omega(S\left|a_{i}\right|)\xrightarrow[i\in\netdirset I]{}1$ we
infer that $\omega(Se_{j})\xrightarrow[j\in\netdirset J]{}1$ because
$\omega\circ S$ is a state. Thus, by \cite[Lemma~A.4]{Kustermans_Vaes__LCQG_C_star},
the increasing net $\left(Se_{j}\right)_{j\in\netdirset J}$ converges
strictly to $\one_{\M(B)}$. Finally, repeating the last argument
with $S$ being the identity map on $A$ shows that $\left(e_{j}\right)_{j\in\netdirset J}$
is an approximate identity for $A$.
\end{proof}
The following is a general observation, or rather warning, that is
pertinent to the proofs of Theorem~\ref{thm:Psi_comult_crossed_prod}
and Proposition~\ref{prop:single map}.
\begin{observation}
\label{obs:two_strict_top}Let $A,B$ be \cst-algebras such that
$A$ is a non-degenerate \cst-subalgebra of $\M(B)$, i.e.~the inclusion
map is non-degenerate. We can, and throughout the paper tacitly will,
view $\M(A)$ as a \cst-subalgebra of $\M(B)$, namely the idealiser
of $A$ inside $\M(B)$ \cite[Proposition~2.3]{Lance}. Assume now
that $A,B$ are non-unital. Then the closed unit ball of $\M(A)$
has two (non-trivial) strict topologies, namely the one coming from
$\M(A)$ and the one coming from $\M(B)$. The former is finer than
the latter by one of the equivalent conditions of non-degeneracy \cite[Proposition~2.5]{Lance}.
However, these topologies \emph{do not} coincide in general when $A\nsubseteq B$.
Below we give two examples of this.

The first example, which is in the commutative case, was kindly communicated
to us by Mariusz Tobolski and one of the referees. Take $A:=\Cz(\R)$
and $B:=\Cz(\R\backslash\left\{ 0\right\} )$. Then every bounded
sequence $\left(f_{n}\right)_{n=1}^{\infty}$ in $\Cz(\R)$ satisfying
$f_{n}(x)=1$ for each $x\in\left[-n,-\frac{1}{n}\right]\cup\left[\frac{1}{n},n\right]$
and $f_{n}(0)=0$ for all $n\in\N$ converges to $\one$ in the strict
topology of $\Cb(\R\backslash\left\{ 0\right\} )=\M(\Cz(\R\backslash\left\{ 0\right\} ))$
but not in the strict topology of $\Cb(\R)=\M(\Cz(\R))$. The topological
reason behind this is that the (embedding) map $\R\backslash\left\{ 0\right\} \to\R$
is \emph{not compact covering}, i.e., there exists a compact subset
of the co-domain (here, $\left\{ 0\right\} $) that is not contained
in the image of any compact subset of the domain.

The second example was kindly communicated to us by Jacek Krajczok.
Take $A:=\Cz(\R)$ and $B:=\Cz(\R)\rtimes\R$, where the crossed product
is with respect to the (say, left) translation action. By the Stone--von
Neumann theorem (see, e.g., \cite[Theorem~4.24]{Williams__cros_prod_book})
there is a natural isomorphism $\Cz(\R)\rtimes\R\cong K(\Ltwo(\R))$.
Now, every bounded sequence $\left(f_{n}\right)_{n=1}^{\infty}$ in
$\Cz(\R)$ such that $f_{n}(0)=1$ and $\supp f_{n}\subseteq\left[-\frac{1}{n},\frac{1}{n}\right]$
for all $n\in\N$ does not converge to $0$ in the strict topology
of $\Cb(\R)=\M(\Cz(\R))$, but converges to $0$ in the strict topology
of $B(\Ltwo(\R))=\M(K(\Ltwo(\R)))$.
\end{observation}

\begin{observation}
\label{obs:slice_maps_nondeg_incl}Let $B,D$ be \cst-algebras. Then
we can and shall view $\M(B)\tensormin\M(D)$ as embedded (unitally)
inside $\M(B\tensormin D)$. Let also $A,C$ be \cst-algebras such
that $A\subseteq\M(B)$ and $C\subseteq\M(D)$ non-degenerately. Then
by virtue of the embedding of $A\tensormin C$ inside $\M(B)\tensormin\M(D)$
we also view $A\tensormin C$ as embedded non-degenerately inside
$\M(B\tensormin D)$ \cite[pp.~37--38]{Lance}.

For $\omega\in B^{*}$ we have the slice map ${\it \omega\tensor\i_{D}:\M(B\tensormin D)\to\M(D)}$,
and considering the restriction $\omega|_{A}\in A^{*}$ (where, of
course, we are restricting to $A$ the strict extension of $\omega$
to $\M(B)$) we also have the slice map $(\omega|_{A})\tensor\i_{C}:\M(A\tensormin C)\to\M(C)$.
Each of these maps is continuous on the closed unit ball in the respective
strict topologies \cite[p.~79]{Lance}. The assertion of this observation
is that, under the embedding $\M(A\tensormin C)\hookrightarrow\M(B\tensormin D)$,
the second slice map is the restriction of the first one. We leave
the details to the reader.
\end{observation}

\section{Locally compact quantum groups and cocycle twisting\label{sec:LCQGCT}}

In this section we introduce locally compact quantum group preliminaries
and discuss the cocycle twisting in this context, focusing on cocycles
coming from closed abelian subgroups. 

\subsection{Locally compact quantum groups}

A \emph{locally compact quantum group} in its von Neumann algebraic
form as introduced by Kustermans and Vaes is a pair $\G=\left(\Linfty(\G),\Delta_{\G}\right)$
such that $\Linfty(\G)$ is a von Neumann algebra and $\Delta_{\G}:\Linfty(\G)\to\Linfty(\G)\tensorn\Linfty(\G)$
is a co-multiplication, i.e.~a normal faithful unital $*$-homomorphism
that is co-associative: $(\Delta_{\G}\tensor\i)\circ\Delta_{\G}=(\i\tensor\Delta_{\G})\circ\Delta_{\G}$,
such that there exist left and right Haar weights for $\Delta_{\G}$.
The Haar weights are unique up to scaling. We refer to \cite{Kustermans__LCQG_universal,Kustermans_Vaes__LCQG_C_star,Kustermans_Vaes__LCQG_von_Neumann,Van_Daele__LCQGs}
for the full details. 

When the context is clear, we drop the $\G$ in symbols associated
to $\G$ and adorn with hats symbols associated to the dual locally
compact quantum group $\hat{\G}$, e.g.~$\Delta:=\Delta_{\G}$ and
$\hat{\Delta}:=\Delta_{\hat{\G}}$.

In addition to the von Neumann algebra $\Linfty(\G)$, two \cst-algebras
providing different perspectives on $\G$ are the reduced \cst-algebra
$\Cz(\G)$ and the universal \cst-algebra $\CzU(\G)$. Each carries
its own version of the co-multiplication.

We work with the left Haar weight $\varphi$ and left regular representation
$W$, which implements $\Delta$ by $\Delta(x)=W^{*}(\one\tensor x)W$
for all $x\in\Linfty(\G)$. Identifying the GNS Hilbert spaces $\Ltwo(\Linfty(\G),\varphi)$
and $\Ltwo(\Linfty(\hat{\G}),\hat{\varphi})$ as customary and denoting
the common Hilbert space by $\Ltwo(\G)$, we view $\Linfty(\G)$ and
$\Linfty(\hat{\G})$ as acting on it. The GNS map $\mathcal{N}_{\varphi}\to\Ltwo(\G)$
of $\varphi$ is denoted by $\Lambda$. We write $\nabla,J$ for the
modular operator and modular conjugation of $\varphi$ (on $\Ltwo(\G)$)
and $\left(\sigma_{t}\right)_{t\in\R}$ for the modular automorphism
group of $\varphi$. We will occasionally need the commutant quantum
group of $\G$, denoted $\G'$ (so that $\Linfty(\G'):=\Linfty(\G)'$
inside $B(\Ltwo(\G))$).

We put $\Lone(\G):=\Linfty(\G)_{*}$ and note the density condition
\begin{equation}
\clinspan^{\sigma\text{-weak}}\left\{ (\i\tensor\omega)\left(\Delta(x)\right):\omega\in\Lone(\G),x\in\Linfty(\G)\right\} =\Linfty(\G).\label{eq:comult_density}
\end{equation}
The antipode, the unitary antipode, and the scaling group of $\G$
will be denoted $\Sant$, $\Rant$, and $\left(\tau_{t}\right)_{t\in\R}$,
respectively. We shall use, sometimes tacitly, the formula $\Rant(x)=\hat{J}x^{*}\hat{J}$
for $x\in\Linfty(\G)$, and also require the formulas 
\begin{gather}
(\Rant\tensor\hat{\Rant})(W)=W,\text{ i.e. }(\Ad(\hat{J}\tensor J))(W)=W^{*},\label{eq:W_R_R_hat}\\
\Delta\circ\Rant=\sigma\circ(\Rant\tensor\Rant)\circ\Delta,\label{eq:Delta_R}
\end{gather}
where $\sigma$ is the algebra-level flip operator, and 
\begin{equation}
\left(\tau_{t}\tensor\sigma_{t}\right)\circ\Delta=\Delta\circ\sigma_{t}\qquad(\forall_{t\in\R}).\label{eq:comult_tau_sigma}
\end{equation}

We end this subsection with two notations. First, define $\mathcal{I}$
as the subspace of $\Lone(\G)$ consisting of all $\omega\in\Lone(\G)$
such that the conjugate-linear map $\Lambda(x)\mapsto\omega(x^{*})$,
$x\in\mathcal{N}_{\varphi}$, is bounded; equivalently, such that
there exists (uniquely) $\xi\in\Ltwo(\G)$ with $\omega(x^{*})=\left\langle \xi,\Lambda(x)\right\rangle $
for all $x\in\mathcal{N}_{\varphi}$. Second, define $\LoneS(\G)$
to be the subspace of $\Lone(\G)$ consisting of all $\omega\in\Lone(\G)$
such that some (unique) element of $\Lone(\G)$, denoted by $\omega^{\sharp}$,
satisfies $\omega^{\sharp}(x)=\overline{\omega(\Sant(x)^{*})}$ for
all $x\in D(\Sant)$.

\subsection{Cocycle twisting of locally compact quantum groups\label{subsec:cocycle_twisting}}

We take the general approach to the Rieffel deformations of locally
compact quantum groups based on the work of De~Commer on Galois objects
and cocycle twisting \cite{DeCommer__Galois_obj_cocycle_twist_LCQG,DeCommer__PhD},
which is introduced in this subsection. Let $\QG$ be a locally compact
quantum group and $\Omega$ be a (measurable, unitary) \emph{$2$-cocycle}
on the dual $\hat{\G}$, that is, a unitary in $\Linfty(\hat{\G})\tensorn\Linfty(\hat{\G})$
such that 
\begin{equation}
(\Omega\tensor\one)\cdot(\hat{\Delta}\tensor\i)(\Omega)=(\one\tensor\Omega)\cdot(\i\tensor\hat{\Delta})(\Omega).\label{eq:2_cocycle}
\end{equation}
The map $\hat{\Delta}_{\Omega}:\Linfty(\hat{\G})\to\Linfty(\hat{\G})\tensorn\Linfty(\hat{\G})$
given by 
\[
\hat{\Delta}_{\Omega}:=(\Ad\Omega)\circ\hat{\Delta}
\]
is a co-multiplication on $\Linfty(\hat{\G})$ by (\ref{eq:2_cocycle}).
It is a deep result that $\hat{\Delta}_{\Omega}$ admits left and
right invariant weights, namely that $\hat{\G}_{\Omega}:=(\Linfty(\hat{\G}),\hat{\Delta}_{\Omega})$
is a (von Neumann algebraic) locally compact quantum group \cite[Theorem~6.2]{DeCommer__Galois_obj_cocycle_twist_LCQG},
called the \emph{cocycle twisting} of $\hat{\G}$ by $\Omega$. Note
that $\Omega^{*}$ may be viewed as a (measurable, unitary) $2$-cocycle
on $\hat{\G}_{\Omega}$ and we naturally have that $\hat{\G}$ is
the cocycle twisting of $\hat{\G}_{\Omega}$ by $\Omega^{*}$.

The analysis of $\hat{\G}_{\Omega}$, and in particular the formula
for its left regular representation, relies on the following construction
of \cite[p.~97]{DeCommer__Galois_obj_cocycle_twist_LCQG} based on
\cite{Vaes_Vainerman__extensions_LCQGs}, see also \cite[p.~464]{Neshveyev_Tuset__deform_C_alg_cocycle_LCQG}.
Consider the trivial left action of $\hat{\G}$ on $\C$, which is
just the unital homomorphism $\C\to\Linfty(\hat{\G})\tensor\C\cong\Linfty(\hat{\G})$.
Together with $\Omega$ it forms a cocycle action, and the associated
cocycle (or twisted) crossed\emph{ }product is the von Neumann algebra
\[
\hat{\G}\pres{}{\Omega}{\ltimes}\C=\overline{\big\{(\omega\tensor\i)(\hat{W}\Omega^{*}):\omega\in\Lone(\hat{\G})\big\}}^{\sigma\text{-weak}}\subseteq B(\Ltwo(\G)).
\]
The dual of this trivial action is the right action $\a:\hat{\G}\pres{}{\Omega}{\ltimes}\C\to(\hat{\G}\pres{}{\Omega}{\ltimes}\C)\tensorn\Linfty(\G)$
of $\G$ on $\hat{\G}\pres{}{\Omega}{\ltimes}\C$ determined by $(\i\tensor\a)(\hat{W}\Omega^{*})=\hat{W}_{13}{(\hat{W}\Omega^{*})}_{12}$
(here the dual action is in the sense of \cite[Proposition~1.4]{Vaes_Vainerman__extensions_LCQGs},
and in particular it is a right action; note however that \cite{Vaes_Vainerman__extensions_LCQGs}
uses in this context the equivalent concept of a left action of the
opposite locally compact quantum group). Denoting by $\theta$ the
(unique) state on $\C$, its dual weight in the sense of \cite[Definition~1.13]{Vaes_Vainerman__extensions_LCQGs}
is the n.s.f.~weight $\tilde{\theta}$ on $\hat{\G}\pres{}{\Omega}{\ltimes}\C$
given by $(\i\tensor\varphi)\circ\a$; the image of this operator-valued
weight is $\C\one$, which we identify with $\C$. With respect to
the natural representation of $\hat{\G}\pres{}{\Omega}{\ltimes}\C$
on $\Ltwo(\G)$ this dual weight $\tilde{\theta}$ has a GNS construction
on $\Ltwo(\G)$ given by 
\begin{equation}
\tilde{\Lambda}((\omega\tensor\i)(\hat{W}\Omega^{*}))=\Lambda((\omega\tensor\i)(\hat{W}))\qquad(\forall_{\omega\in\hat{\mathcal{I}}}).\label{eq:Lambda_tilde}
\end{equation}
Denote by $\tilde{\nabla},\tilde{J}$ the modular operator and modular
conjugation of $\tilde{\theta}$ pertaining to this GNS construction.
The canonical unitary implementation of the cocycle action formed
by the trivial action of $\hat{\G}$ on $\C$ and $\Omega$ is precisely
\begin{equation}
X:=\tilde{J}J\in\mathcal{U}(\Ltwo(\G)),\label{eq:unitary_X}
\end{equation}
which actually belongs to $\Linfty(\hat{\G})$ and is invariant under
the dual unitary antipode, i.e.~$JXJ=X^{*}$ \cite[Definition~1.18 and Proposition~1.19]{Vaes_Vainerman__extensions_LCQGs}
(that $X\in\Linfty(\hat{\G})$ is also proved in \cite[p.~99, l.~3,4]{DeCommer__Galois_obj_cocycle_twist_LCQG}).

The left regular representation of $\hat{\G}_{\Omega}$, when acting
on $\Ltwo(\G)\tensor\Ltwo(\G)$, is 
\begin{equation}
\hat{W}_{\Omega}=(\tilde{J}\tensor\hat{J})\Omega\hat{W}^{*}(J\tensor\hat{J})\Omega^{*}\label{eq:W_hat_Omega}
\end{equation}
\cite[Proposition~6.5]{DeCommer__Galois_obj_cocycle_twist_LCQG}.
Employing the unitary $X$ above, we obtain 
\begin{equation}
\begin{split}\hat{W}_{\Omega} & =(JX^{*}\tensor\hat{J})\Omega\hat{W}^{*}(J\tensor\hat{J})\Omega^{*}\\
 & =(J\tensor\hat{J})(X^{*}\tensor\one)\Omega(J\tensor\hat{J})(J\tensor\hat{J})\hat{W}^{*}(J\tensor\hat{J})\Omega^{*}\\
 & =(J\tensor\hat{J})(X^{*}\tensor\one)\Omega(J\tensor\hat{J})\hat{W}\Omega^{*}.
\end{split}
\label{eq:W_hat_Omega_clearer}
\end{equation}

The von Neumann algebra $\Linfty(\hat{\G})=\Linfty(\hat{\G}_{\Omega})$
is equipped with the left Haar weights $\hat{\varphi},\hat{\varphi}_{\Omega}$
of $\hat{\G},\hat{\G}_{\Omega}$, respectively. The modular conjugation
$\hat{J}_{\Omega}$ of $\hat{\varphi}_{\Omega}$, when viewed canonically
as acting on $\Ltwo(\G)=\Ltwo(\Linfty(\hat{\G}),\hat{\varphi})$,
equals the modular conjugation $\hat{J}$ of $\hat{\varphi}$, because
the canonical unitary operator that intertwines the GNS representations
of $\Linfty(\hat{\G})$ on $\Ltwo(\Linfty(\hat{\G}),\hat{\varphi})$
and $\Ltwo(\Linfty(\hat{\G}),\hat{\varphi}_{\Omega})$ is the same
unitary that intertwines $\hat{J}$ and $\hat{J}_{\Omega}$ (by a
general fact on von Neumann algebras \cite[Lemma~IX.1.5 and its proof]{Takesaki__book_vol_2}).
Note that this unitary was used in the proof of \cite[Proposition~6.5]{DeCommer__Galois_obj_cocycle_twist_LCQG}
to have the right leg of $\hat{W}_{\Omega}$ in (\ref{eq:W_hat_Omega})
act on $\Ltwo(\G)=\Ltwo(\Linfty(\hat{\G}),\hat{\varphi})$ (in addition
to the left leg). See also \cite[text succeeding eq.~(2.4)]{Neshveyev_Tuset__deform_C_alg_cocycle_LCQG}.

In this paper we are interested in the dual locally compact quantum
group $\G^{\Omega}:=\widehat{\hat{\G}_{\Omega}}$, whose $\Linfty$-algebra
we also view as acting on $\Ltwo(\G)=\Ltwo(\Linfty(\hat{\G}),\hat{\varphi})$
canonically. By the previous paragraph and a general formula stated
earlier, the unitary antipode $\Rant_{\G^{\Omega}}$ is implemented
by $\hat{J}$: 
\begin{equation}
\Rant_{\G^{\Omega}}(x)=\hat{J}x^{*}\hat{J}\qquad(\forall_{x\in\G^{\Omega}}).\label{eq:R_G_sup_Omega}
\end{equation}
Under certain assumptions, $\G^{\Omega}$ can be seen as the Rieffel
deformation of $\G$ with respect to $\Omega$. This is explained
in the next section.

We finish this subsection with a few more facts. The operators $\big(v_{t}:=\tilde{\nabla}^{it}\nabla^{-it}\big)_{t\in\R}$
belong to $\Linfty(\hat{\G})$ \cite[p.~99, l.~2,3]{DeCommer__Galois_obj_cocycle_twist_LCQG}
and form a $1$-cocycle for $\hat{\tau}$, i.e.~$v_{s+t}=v_{s}\hat{\tau}_{s}(v_{t})$
for all $s,t\in\R$ \cite[Proposition~6.3, (i)]{DeCommer__Galois_obj_cocycle_twist_LCQG}.
Then the scaling group $\hat{\tau}_{\Omega}=\left(\hat{\tau}_{\Omega,t}\right)_{t\in\R}$
of $\hat{\G}_{\Omega}$ equals $\left(\Ad(v_{t})\circ\hat{\tau}_{t}\right)_{t\in\R}$.
Indeed, this is spelled out in \cite[l.~19,20 in proof of Proposition~9.1.6]{DeCommer__PhD}.
Since the last reference is unpublished, we note the following alternative
simple proof. Consider the Connes cocycle derivative $\left(u_{t}:=\left[D\hat{\varphi}_{\Omega}:D\hat{\varphi}\right]_{t}\right)_{t\in\R}$
in $\Linfty(\hat{\G})$. Using the fact that the modular automorphism
group $\hat{\sigma}_{\Omega}=\left(\hat{\sigma}_{\Omega,t}\right)_{t\in\R}$
of $\hat{\varphi}_{\Omega}$ equals $\left(\Ad(u_{t})\circ\hat{\sigma}_{t}\right)_{t\in\R}$
as well as the formula 
\[
\hat{\Delta}\left(u_{t}\right)\cdot\left(\hat{\tau}_{t}\tensor\hat{\sigma}_{t}\right)\left(\Omega^{*}\right)=\Omega^{*}\cdot(v_{t}\tensor u_{t})\qquad(\forall_{t\in\R})
\]
(see \cite[proof of Proposition~6.3]{DeCommer__Galois_obj_cocycle_twist_LCQG})
and the analogue of (\ref{eq:comult_tau_sigma}) for $\hat{\G}$,
one calculates that 
\[
\left(\left(\Ad(v_{t})\circ\hat{\tau}_{t}\right)\tensor\hat{\sigma}_{\Omega,t}\right)\circ\hat{\Delta}_{\Omega}=\hat{\Delta}_{\Omega}\circ\hat{\sigma}_{\Omega,t}\qquad(\forall_{t\in\R}).
\]
This implies that $\hat{\tau}_{\Omega}=\left(\Ad(v_{t})\circ\hat{\tau}_{t}\right)_{t\in\R}$
because of the analogues of (\ref{eq:comult_density}) and (\ref{eq:comult_tau_sigma})
for $\hat{\G}_{\Omega}$.
\begin{lem}
\label{lem:Omega_scal_grp_inv}If $\left(\hat{\tau}_{t}\tensor\hat{\tau}_{t}\right)\left(\Omega\right)=\Omega$
for all $t\in\R$, then the positive selfadjoint operators $\tilde{\nabla},\nabla$
strongly commute. 
\end{lem}

\begin{proof}
Recall the operators $v_{t}$ introduced in the beginning of the last
paragraph. From \cite[Proposition~6.3, (ii)]{DeCommer__Galois_obj_cocycle_twist_LCQG}
and the invariance assumption we get that 
\[
\Omega^{*}\left(v_{t}\tensor v_{t}\right)=\hat{\Delta}\left(v_{t}\right)\left(\hat{\tau}_{t}\tensor\hat{\tau}_{t}\right)\left(\Omega^{*}\right)=\hat{\Delta}\left(v_{t}\right)\Omega^{*},
\]
that is, $\hat{\Delta}_{\Omega}\left(v_{t}\right)=v_{t}\tensor v_{t}$,
for all $t\in\R$. So each $v_{t}$ is a one-dimensional representation
of $\hat{\G}_{\Omega}$, and is thus invariant under the scaling group
$\hat{\tau}_{\Omega}$. Using the facts that $\hat{\tau}_{\Omega}=\left(\Ad(v_{t})\circ\hat{\tau}_{t}\right)_{t\in\R}$
and that $\left(v_{t}\right)_{t\in\R}$ is a $1$-cocycle for $\hat{\tau}$,
we obtain 
\[
v_{t}=\hat{\tau}_{\Omega,s}(v_{t})=v_{s}\hat{\tau}_{s}\left(v_{t}\right)v_{s}^{*}=v_{s+t}v_{s}^{*}\qquad(\forall_{s,t\in\R}),
\]
namely $\left(v_{t}\right)_{t\in\R}$ is a one-parameter unitary group.
By the definition of $\left(v_{t}\right)_{t\in\R}$ and a classical
result \cite[Exercise~E.9.24]{Stratila_Zsido__lectures_vN}, this
means that the positive selfadjoint operators $\tilde{\nabla},\nabla$
strongly commute. 
\end{proof}

\subsection{Cocycle twisting via locally compact abelian subgroups\label{subsec:cocyc_twist_lcag}}

We lay out the hypothesis of this subsection and of paper's main results.
Note that for a locally compact group, the two notions of being a
closed quantum subgroup of a locally compact quantum group, due to
Vaes and to Woronowicz \cite{Daws_Kasprzak_Skalski_Soltan__closed_q_subgroups_LCQGs},
coincide by \cite[Corollary~5.7]{Daws__Categorical_asp_QG_mult_intr_grp}.
\begin{standinghypothesis}
\label{stand_hyp}Let a locally compact abelian group $\Gamma$ be
a closed quantum subgroup of a locally compact quantum group $\G$.
Let $\Psi$ be a normalised continuous $2$-cocycle on $\hat{\Gamma}$. 
\end{standinghypothesis}

The (Vaes) closed quantum subgroup assumption means that there exists
an injective normal $*$-homomorphism $L:\Linfty(\hat{\Gamma})\to\Linfty(\hat{\G})$
intertwining the co-multiplications (and thus also automatically the
unitary antipodes, scaling groups, etc.). Viewing $\Psi$ as an element
of $\mathcal{U}(\Linfty(\hat{\Gamma})\tensorn\Linfty(\hat{\Gamma}))$,
the operator $\Omega:=(L\tensor L)(\Psi)$ is a $2$-cocycle on the
locally compact quantum group $\hat{\G}$, and we can form the locally
compact quantum groups $\hat{\G}_{\Psi}:=\hat{\G}_{\Omega}$ and $\GPsi:=\widehat{\hat{\G}_{\Psi}}$
as in the previous subsection. 

Define $R:\Linfty(\hat{\Gamma})\to\Linfty(\hat{\G}')$ by $R:=\Ad(\hat{J}J)\circ L$.
Since $L$ intertwines the co-multiplications, 
\begin{equation}
\hat{W}^{*}(\one\tensor L(g))\hat{W}=(L\tensor L)(\Delta_{\hat{\Gamma}}(g))\text{ and }\hat{W}(\one\tensor R(g))\hat{W}^{*}=(L\tensor R)(\Delta_{\hat{\Gamma}}(g))\qquad(\forall_{g\in\Linfty(\hat{\Gamma})})\label{eq:L_R_Delta}
\end{equation}
(the second formula follows from (\ref{eq:W_R_R_hat}), (\ref{eq:Delta_R}),
and $\hat{\Gamma}$ being abelian; the last fact is used to deduce
that the formula above, first checked for $g$ being a character,
is valid for arbitrary $g$). Write $\left(L_{\gamma}\right)_{\gamma\in\Gamma}$
and $\left(R_{\gamma}\right)_{\gamma\in\Gamma}$ for the representation
of $\Gamma$ on $\Ltwo(\G)$ associated with the morphisms $L|_{\Cz(\hat{\Gamma})}\in\Mor(\Cz(\hat{\Gamma}),\Cz(\hat{\G}))$
and $R|_{\Cz(\hat{\Gamma})}\in\Mor(\Cz(\hat{\Gamma}),\Cz(\hat{\G}'))$.
Equivalently, for all $\gamma\in\Gamma$,
\[
L_{\gamma}:=L(\gamma)\text{ and }R_{\gamma}:=\hat{J}L_{\gamma}\hat{J},
\]
where as usual we treat $\gamma$ as a function in $\Cb(\hat{\Gamma})\subseteq\Linfty(\hat{\Gamma})$.
Since $\gamma$ is group-like in $\Linfty(\hat{\Gamma})$, so is $L_{\gamma}$
in $\Linfty(\hat{\G})$: 
\begin{gather}
\hat{\Delta}(L_{\gamma})=L_{\gamma}\tensor L_{\gamma},\text{ equivalently: }\hat{W}^{*}(\one\tensor L_{\gamma})\hat{W}=L_{\gamma}\tensor L_{\gamma},\label{eq:L_gamma_group_like}\\
\text{and thus }\hat{W}(\one\tensor R_{\gamma})\hat{W}^{*}=L_{\gamma}\tensor R_{\gamma}.\label{eq:R_gamma_group_like}
\end{gather}
(The map $\Gamma\ni\gamma\mapsto L_{\gamma}$, as an element of $\Linfty(\Gamma)\tensorn\Linfty(\hat{\G})$,
is the bicharacter associated with $\Gamma$ being a closed quantum
subgroup of $\G$ \cite{Meyer_Roy_Woronowicz__hom_quant_grps,Daws_Kasprzak_Skalski_Soltan__closed_q_subgroups_LCQGs}
up to left/right and Fourier transform conventions).

Before we continue, we give a concrete formula for the unitary $X$
of (\ref{eq:unitary_X}) in the current setting.
\begin{prop}
\label{prop:X_equals_U}The unitary $X=\tilde{J}J\in\Linfty(\hat{\G})$
equals $L(u)$ where $u\in\Linfty(\hat{\Gamma})$ is given by 
\begin{equation}
u(\hat{\gamma}):=\Psi(-\hat{\gamma},\hat{\gamma}),\qquad\hat{\gamma}\in\hat{\Gamma}.\label{eq:u}
\end{equation}
\end{prop}

\begin{proof}
Let $U:=L(u)$. We claim that for every $\omega\in\LoneS(\hat{\G})$
we have 
\[
((\omega\tensor\i)(\hat{W}\Omega^{*}))^{*}=((U\omega^{\sharp})\tensor\i)(\hat{W}\Omega^{*}).
\]
Assuming that the claim is proved, denote by $T$ and $\tilde{T}$
the closures of the conjugate-linear operators $\Lambda(x)\mapsto\Lambda(x^{*})$,
$x\in\mathcal{N_{\varphi}\cap N_{\varphi}^{*}}$, and $\tilde{\Lambda}(y)\mapsto\tilde{\Lambda}(y^{*})$,
$y\in\mathcal{N_{\tilde{\theta}}\cap N_{\tilde{\theta}}^{*}}$ (using
the notation of the previous section), both on $\Ltwo(\G)$. Let $\hat{\mathcal{I}}_{\sharp}:=\{\omega\in\LoneS(\hat{\G})\cap\hat{\mathcal{I}}:\omega^{\sharp}\in\hat{\mathcal{I}}\}$.
If $\omega\in\hat{\mathcal{I}}_{\sharp}$, then by the von Neumann
algebraic and dual analogues of \cite[Result~8.6 and Proposition~8.13]{Kustermans_Vaes__LCQG_C_star}
also $U\omega^{\sharp}\in\hat{\mathcal{I}}$ and $\Lambda\big(((U\omega^{\sharp})\tensor\i)(\hat{W})\big)=U\Lambda((\omega^{\sharp}\tensor\i)(\hat{W}))$
as $U\in\Linfty(\hat{\G})$, and we get from (\ref{eq:Lambda_tilde})
and the claim above that $((\omega\tensor\i)(\hat{W}\Omega^{*}))^{*}\in\mathcal{N}_{\tilde{\theta}}$
and
\[
\tilde{T}\tilde{\Lambda}((\omega\tensor\i)(\hat{W}\Omega^{*}))=\tilde{\Lambda}\big(((\omega\tensor\i)(\hat{W}\Omega^{*}))^{*}\big)=\tilde{\Lambda}\big(((U\omega^{\sharp})\tensor\i)(\hat{W}\Omega^{*})\big),
\]
so that 
\[
\begin{split}\tilde{T}\Lambda((\omega\tensor\i)(\hat{W})) & =\Lambda\big(((U\omega^{\sharp})\tensor\i)(\hat{W})\big)=U\Lambda((\omega^{\sharp}\tensor\i)(\hat{W}))\\
 & =U\Lambda\big(((\omega\tensor\i)(\hat{W}))^{*}\big)=UT\Lambda((\omega\tensor\i)(\hat{W})).
\end{split}
\]
The set $\{\Lambda((\omega\tensor\i)(\hat{W})):\omega\in\hat{\mathcal{I}}_{\sharp}\}$
is a core for $T$ by \cite[Lemma~2.7]{Kustermans_Vaes__LCQG_von_Neumann},
hence $UT\subseteq\tilde{T}$, i.e.~$UJ\nabla^{1/2}\subseteq\tilde{J}\tilde{\nabla}^{1/2}$.
Since $\Omega:=(L\tensor L)(\Psi)$ and the (quantum) group $\hat{\Gamma}$
has trivial scaling group, $\Omega$ is invariant under $\left(\hat{\tau}_{t}\tensor\hat{\tau}_{t}\right)_{t\in\R}$.
By Lemma~\ref{lem:Omega_scal_grp_inv}, the positive selfadjoint
operators $\tilde{\nabla},\nabla$ strongly commute, and hence they
have a common core, and in particular a core for $\tilde{\nabla}^{1/2}$
contained in the domain of $\nabla^{1/2}$. By the closedness of $UJ\nabla^{1/2}$
this entails the equality $UJ\nabla^{1/2}=\tilde{J}\tilde{\nabla}^{1/2}$.
The uniqueness of the polar decomposition yields that ($\tilde{\nabla}=\nabla$
and) $UJ=\tilde{J}$, as desired.

We finish by proving the claim. Let $\omega\in\LoneS(\hat{\G})$.
For every $\gamma_{1},\gamma_{2}\in\Gamma$ we have 
\[
\begin{split}\big((\omega\tensor\id)(\hat{W}(L_{\gamma_{1}}\tensor L_{\gamma_{2}})^{*})\big)^{*} & =L_{\gamma_{2}}\cdot\big((L_{-\gamma_{1}}\omega\tensor\id)(\hat{W})\big)^{*}\\
 & \overset{(\star)}{=}L_{\gamma_{2}}\cdot(L_{-\gamma_{1}}\omega^{\sharp}\tensor\id)(\hat{W})\\
 & \overset{(\star\star)}{=}(\omega^{\sharp}\tensor\id)(\hat{W}(L_{\gamma_{2}-\gamma_{1}}\tensor L_{\gamma_{2}}))
\end{split}
\]
(in $(\star)$ we used that $(\hat{\Sant}(L_{-\gamma_{1}}))^{*}=L_{\gamma_{1}}^{*}=L_{-\gamma_{1}}$
and thus $L_{-\gamma_{1}}\omega\in\LoneS(\hat{\G})$ and ${(L_{-\gamma_{1}}\omega)}^{\sharp}=L_{-\gamma_{1}}\omega^{\sharp}$,
and in $(\star\star)$ we used (\ref{eq:L_gamma_group_like})), that
is, 
\begin{equation}
\big((\omega\tensor\id)(\hat{W}\cdot(L\tensor L)(\gamma_{1}\tensor\gamma_{2})^{*})\big)^{*}=(\omega^{\sharp}\tensor\id)\big(\hat{W}\cdot(L\tensor L)((\gamma_{2}-\gamma_{1})\tensor\gamma_{2})\big).\label{eq:X_equals_U__1}
\end{equation}
Write $\Theta$ for the automorphism of $\Linfty(\hat{\Gamma}\times\hat{\Gamma})$
that maps $F\in\Linfty(\hat{\Gamma}\times\hat{\Gamma})$ to $(\hat{\gamma}_{1},\hat{\gamma}_{2})\mapsto F(-\hat{\gamma}_{1},\hat{\gamma}_{1}+\hat{\gamma}_{2})$.
For every $\gamma_{1},\gamma_{2}\in\Gamma$ we have $\Theta(\gamma_{1}\tensor\gamma_{2})=(\gamma_{2}-\gamma_{1})\tensor\gamma_{2}$.
By linearity and continuity we infer from (\ref{eq:X_equals_U__1})
that 
\[
\big((\omega\tensor\id)(\hat{W}\cdot(L\tensor L)(F)^{*})\big)^{*}=(\omega^{\sharp}\tensor\id)\big(\hat{W}\cdot(L\tensor L)(\Theta(F))\big)\qquad(\forall_{F\in\Linfty(\hat{\Gamma}\times\hat{\Gamma})}).
\]
Taking in particular $F:=\Psi$ and applying the equality (\ref{eq:normal_2_cocycle_U})
which in the language of this proposition reads $\Theta(\Psi)=\overline{\Psi}\cdot(u\tensor\one)$,
we obtain 
\[
\big((\omega\tensor\id)(\hat{W}\Omega^{*})\big)^{*}=(\omega^{\sharp}\tensor\id)\big(\hat{W}\Omega^{*}(U\tensor\one)\big)=((U\omega^{\sharp})\tensor\id)\big(\hat{W}\Omega^{*}\big),
\]
as claimed.
\end{proof}
Proposition~\ref{prop:X_equals_U} yields a more transparent formula
for the left regular representation $\hat{W}_{\Psi}:=\hat{W}_{\Omega}$
(acting on $\Ltwo(\G)\tensor\Ltwo(\G)$) of the locally compact quantum
group $\hat{\G}_{\Psi}:=\hat{\G}_{\Omega}$. Define functions $\Upsilon,\mathring{\Psi}:\hat{\Gamma}\times\hat{\Gamma}\to\mathbb{T}$
by 
\begin{equation}
\begin{aligned}\Upsilon(\hat{\gamma}_{1},\hat{\gamma}_{2}) & :=\overline{\Psi(-\hat{\gamma}_{1},\hat{\gamma}_{1}+\hat{\gamma}_{2})},\\
\mathring{\Psi}(\hat{\gamma}_{1},\hat{\gamma}_{2}) & :=\Psi(\hat{\gamma}_{1},-\hat{\gamma}_{1}-\hat{\gamma}_{2}),
\end{aligned}
\quad\hat{\gamma}_{1},\hat{\gamma}_{2}\in\hat{\Gamma}.\label{eq:Upsilon_Psi_ring}
\end{equation}
By (\ref{eq:normal_2_cocycle_U}) we have $\Upsilon=(u^{*}\tensor\one)\Psi$
with $u\in\Linfty(\hat{\Gamma})$ defined in (\ref{eq:u}), hence
$(\Rant_{\hat{\Gamma}}\tensor\Rant_{\hat{\Gamma}})((u^{*}\tensor\one)\Psi)^{*}=\mathring{\Psi}$.
As a result, 
\[
\begin{split}(J\tensor\hat{J})(L(u)^{*}\tensor\one)\Omega(J\tensor\hat{J}) & =(\one\tensor\hat{J}J)(J\tensor J)(L(u)^{*}\tensor\one)\Omega(J\tensor J)(\one\tensor J\hat{J})\\
 & =(\one\tensor\hat{J}J)\cdot(L\tensor L)\big((\Rant_{\hat{\Gamma}}\tensor\Rant_{\hat{\Gamma}})((u^{*}\tensor\one)\Psi)\big)^{*}\cdot(\one\tensor J\hat{J})\\
 & =(\one\tensor\hat{J}J)\cdot(L\tensor L)(\mathring{\Psi})\cdot(\one\tensor J\hat{J})=(L\tensor R)(\mathring{\Psi}).
\end{split}
\]
Combining (\ref{eq:W_hat_Omega_clearer}) with Proposition~\ref{prop:X_equals_U}
we get $\hat{W}_{\Psi}=(J\tensor\hat{J})(L(u)^{*}\tensor\one)\Omega(J\tensor\hat{J})\hat{W}\Omega^{*}$.
Thus ultimately
\begin{equation}
\hat{W}_{\Psi}=(L\tensor R)(\mathring{\Psi})\cdot\hat{W}\cdot(L\tensor L)(\overline{\Psi}).\label{eq:W_hat_Psi}
\end{equation}

\section{The Rieffel deformation\label{sec:Rieffel_deform}}

This section is devoted to the study of \cst-algebraic Rieffel deformations
in the context of locally compact quantum groups.

\subsection{The Rieffel deformation of \cst- and von Neumann algebras\label{subsec:Rieffel_deform_algebras}}

In this subsection we follow \cite[Sections~2,3]{Kasprzak__Rieffel_deform_crossed_prod},
which is based on \cite{Landstad__dual_thy_cov_sys}, see also \cite[Section~7.8]{Pedersen__book}.
Let $\Gamma\overset{\rho}{\curvearrowright}A$ be an action of a locally
compact abelian group $\Gamma$ on a \cst-algebra $A$ and let $\Psi$
be a continuous $2$-cocycle on the dual $\hat{\Gamma}$. We construct
a deformed \cst-subalgebra $A^{\Psi}$ of the multiplier algebra
of the crossed product $A\rtimes_{\rho}\Gamma$ as follows. 

Write $\iota:A\to\M(A\rtimes_{\rho}\Gamma)$ for the canonical embedding,
$\lambda=\left(\lambda_{\gamma}\right)_{\gamma\in\Gamma}$ for the
canonical representation of $\Gamma$ inside $A\rtimes_{\rho}\Gamma$,
and $\hat{\rho}$ for the dual action $\hat{\Gamma}\curvearrowright A\rtimes_{\rho}\Gamma$
characterised by $\hat{\rho}$ fixing $A$ pointwise and satisfying
\begin{equation}
\hat{\rho}_{\hat{\gamma}}(\lambda_{\gamma})=\left\langle \hat{\gamma},\gamma\right\rangle \lambda_{\gamma}\qquad(\forall_{\gamma\in\Gamma,\hat{\gamma}\in\hat{\Gamma}});\label{eq:dual_action}
\end{equation}
this again follows the convention of \cite{Pedersen__book,Kasprzak__Rieffel_deform_crossed_prod}
in contrast to the common convention \cite{Takesaki__book_vol_2,Fima_Vainerman__twist_Rieffel_deform}
of having $\hat{\rho}_{\hat{\gamma}}(\lambda_{\gamma})=\overline{\left\langle \hat{\gamma},\gamma\right\rangle }\lambda_{\gamma}$.
The map in $\Mor(\Cz(\hat{\Gamma})\cong\CStarR(\Gamma),A\rtimes_{\rho}\Gamma)$
associated with $\lambda$ (also denoted by $\lambda$) is injective.
From (\ref{eq:dual_action}) we get 
\begin{equation}
\hat{\rho}_{\hat{\gamma}}(\lambda(g))=\lambda(g(\cdot+\hat{\gamma}))\qquad(\forall_{g\in\Cz(\hat{\Gamma}),\hat{\gamma}\in\hat{\Gamma}}).\label{eq:dual_action_via_Fourier}
\end{equation}

For each $\hat{\gamma}\in\hat{\Gamma}$ we define a unitary element
$\Psi_{\hat{\gamma}}\in\Cb(\hat{\Gamma})\cong\M(\CStarR(\Gamma))$
by $\Psi_{\hat{\gamma}}(\cdot)=\Psi(\cdot,\hat{\gamma})$. Then we
obtain a strictly continuous family of unitaries $\left(U_{\hat{\gamma}}:=\lambda(\Psi_{\hat{\gamma}})\right)_{\hat{\gamma}\in\hat{\Gamma}}$
in $\M(A\rtimes_{\rho}\Gamma)$, which by (\ref{eq:dual_action_via_Fourier})
and (\ref{eq:func_2_cocycle}) satisfies
\[
U_{\hat{\gamma}_{1}+\hat{\gamma}_{2}}=\overline{\Psi(\hat{\gamma}_{1},\hat{\gamma}_{2})}U_{\hat{\gamma}_{1}}\hat{\rho}_{\hat{\gamma}_{1}}(U_{\hat{\gamma}_{2}})\qquad(\forall_{\hat{\gamma}_{1},\hat{\gamma}_{2}\in\hat{\Gamma}})
\]
(and the scalars belong to $\mathbb{T}$). So $\left(U_{\hat{\gamma}}\right)_{\hat{\gamma}\in\hat{\Gamma}}$
forms a weak $1$-cocycle, in the sense of the \cst-algebraic analogue
of \cite[p.~281]{Popa__some_rig_Bernoulli}, for the dual action $\hat{\Gamma}\overset{\hat{\rho}}{\curvearrowright}A\rtimes_{\rho}\Gamma$.
Since $\left(U_{\hat{\gamma}}\right)_{\hat{\gamma}\in\hat{\Gamma}}$
is commuting, ${(U_{\hat{\gamma}}^{*})}_{\hat{\gamma}\in\hat{\Gamma}}$
is also such a weak $1$-cocycle. Hence, the formula
\begin{equation}
\hat{\rho}_{\hat{\gamma}}^{\Psi}:=\Ad(U_{\hat{\gamma}}^{*})\circ\hat{\rho}_{\hat{\gamma}}\qquad(\hat{\gamma}\in\hat{\Gamma})\label{eq:deformed_dual_action}
\end{equation}
defines a new (`\emph{deformed dual}') action $\hat{\Gamma}\overset{\hat{\rho}^{\Psi}}{\curvearrowright}A\rtimes_{\rho}\Gamma$,
which acts on $\left(\lambda_{\gamma}\right)_{\gamma\in\Gamma}$ like
$\hat{\rho}$ does in (\ref{eq:dual_action}). 

The \emph{Landstad algebra} $A^{\Psi}$ is defined as the set of all
these $x\in\M(A\rtimes_{\rho}\Gamma)$ for which:
\begin{elabeling}{00.00.0000}
\item [{(L1)}] $x$ is a fixed point of the action $\hat{\rho}^{\Psi}$; 
\item [{(L2)}] the map $\Gamma\ni\gamma\mapsto\lambda_{\gamma}x\lambda_{\gamma}^{*}$
is norm continuous; 
\item [{(L3)}] and for all $f,g\in\Cz(\hat{\Gamma})$ we have $\lambda(f)x\lambda(g)\in A\rtimes_{\rho}\Gamma$. 
\end{elabeling}
Then $A^{\Psi}$ is a \cst-subalgebra of $\M(A\rtimes_{\rho}\Gamma)$,
and by \emph{Landstad's theorem} (\cite{Landstad__dual_thy_cov_sys},
\cite[Theorem~7.8.8]{Pedersen__book}) the formula 
\begin{equation}
\rho_{\gamma}^{\Psi}(x):=\lambda_{\gamma}x\lambda_{\gamma}^{*}\qquad(\gamma\in\Gamma,x\in A^{\Psi})\label{eq:deformed_action}
\end{equation}
defines an (`deformed') action $\Gamma\overset{\rho^{\Psi}}{\curvearrowright}A^{\Psi}$
and there is an isomorphism $A\rtimes_{\rho}\Gamma\cong A^{\Psi}\rtimes_{\rho^{\Psi}}\Gamma$.
This isomorphism is natural: after extending to the multiplier algebras,
it maps $A^{\Psi}\subseteq\M(A\rtimes_{\rho}\Gamma)$ given by (L1)--(L3)
above to its canonical copy in $\M(A^{\Psi}\rtimes_{\rho^{\Psi}}\Gamma)$
and the canonical copy of $\Gamma$ in $\M(A\rtimes_{\rho}\Gamma)$
to the corresponding canonical copy in $\M(A^{\Psi}\rtimes_{\rho^{\Psi}}\Gamma)$.
In particular, the inclusion $A^{\Psi}\subseteq\M(A\rtimes_{\rho}\Gamma)$
is non-degenerate.

A very similar deformation procedure is available in the von Neumann
algebraic setting: in this case, the Landstad algebra is just the
fixed-point algebra under the deformed dual action.
\begin{rem}
Landstad's work \cite{Landstad__dual_thy_cov_sys} covers actions
of non-abelian locally compact groups. It was extended to actions
of locally compact quantum groups on von Neumann algebras \cite[Proposition~1.22]{Vaes_Vainerman__extensions_LCQGs}
and on \cst-algebras \cite[Theorem~6.7]{Vaes__new_appr_induct_impriv}.
\end{rem}

\subsection{Pictures of the Rieffel deformation of locally compact quantum groups\label{subsec:Rieffel_deform_LCQGs}}

We continue to use Standing~Hypothesis~\ref{stand_hyp} and the
notation of Subsection~\ref{subsec:cocyc_twist_lcag}, involving
a locally compact quantum group $\G$, a locally compact abelian group
$\Gamma$ that is a closed quantum subgroup of $\G$, and a normalised
continuous $2$-cocycle $\Psi$ on $\hat{\Gamma}$. We again write
$\Omega:=(L\tensor L)(\Psi)$, $\hat{\G}_{\Psi}:=\hat{\G}_{\Omega}$,
and $\GPsi:=\widehat{\hat{\G}_{\Psi}}$.

The \emph{Rieffel deformation} of locally compact (quantum) groups
was introduced and studied in \cite{Kasprzak__Rieffel_deform_crossed_prod}
and then in \cite{Fima_Vainerman__twist_Rieffel_deform} following
the works of \cite{Rieffel_deform,Enock_Vainerman__deform_Kac_alg_abel_grp}.
It is a new locally compact quantum group built from $\G,\Gamma,\Psi$
as above. The two papers \cite{Kasprzak__Rieffel_deform_crossed_prod,Fima_Vainerman__twist_Rieffel_deform}
provide essentially the same construction, but they work under different
assumptions (both are a little less general than ours) and under different
conventions, as follows:
\begin{itemize}
\item Kasprzak \cite{Kasprzak__Rieffel_deform_crossed_prod} assumes that
$\G$ is a locally compact \emph{group} $G$. He uses the \emph{right}
Haar measure/weight and regular representation, and works in the \cst\emph{-algebraic}
setting. In Subsection~4.3 therein he makes an additional assumption
about the modular function of $G$.
\item Fima and Vainerman \cite{Fima_Vainerman__twist_Rieffel_deform} ask
that $\Psi$ be a \emph{bicharacter} (instead of a general $2$-cocycle)
and that $\Gamma$ satisfy a condition that they call \emph{stability}.
They use the \emph{left} Haar weight and regular representation, and
work in the \emph{von Neumann algebraic} setting.
\end{itemize}
In what follows we briefly describe the Rieffel deformation construction
of \cite{Kasprzak__Rieffel_deform_crossed_prod,Fima_Vainerman__twist_Rieffel_deform}
and explain why it `equals' our $\GPsi$. While the purpose of the
extra assumptions in \cite{Kasprzak__Rieffel_deform_crossed_prod,Fima_Vainerman__twist_Rieffel_deform}
was to prove that the Rieffel deformation is indeed a locally compact
quantum group, we get this `for free' from the comparison to the
locally compact quantum group $\GPsi$ constructed using the work
of De~Commer as explained in the previous chapter. We follow closely
the construction of Kasprzak \cite{Kasprzak__Rieffel_deform_crossed_prod}
as we require the \cst-algebraic setting in some stages of our work.
Nevertheless, we continue to use the \emph{left} Haar weight and regular
representation to make the text in line with \cite{DeCommer__Galois_obj_cocycle_twist_LCQG}.
One should note that the latter paper appeared later than both \cite{Kasprzak__Rieffel_deform_crossed_prod,Fima_Vainerman__twist_Rieffel_deform}.

Recall from Subsection~\ref{subsec:cocyc_twist_lcag} the embeddings
$L:\Linfty(\hat{\Gamma})\to\Linfty(\hat{\G})$ and $R:\Linfty(\hat{\Gamma})\to\Linfty(\hat{\G}')$
pertaining to $\Gamma\le\G$ and the associated representations $\left(L_{\gamma}\right)_{\gamma\in\Gamma}$
and $\left(R_{\gamma}\right)_{\gamma\in\Gamma}$ of $\Gamma$. Since
$L_{\gamma},R_{\gamma}$ are group-like unitaries of $\hat{\G},\hat{\G}'$,
respectively, the automorphisms $\Ad(L_{\gamma}),\Ad(R_{\gamma})$
of $B(\Ltwo(\G))$ map both $\Linfty(\G)$ and $\Cz(\G)$ (acting
on $\Ltwo(\G)$) onto themselves; this is easily seen from (\ref{eq:L_gamma_group_like})
as explained in \cite[proof of Theorem~3.7]{KalantarNeufang_groups},
and is also a very particular case of \cite[Theorem~5.3 and Lemma~5.8]{Meyer_Roy_Woronowicz__hom_quant_grps}.
We call the resulting (commuting) actions $\Gamma\overset{\rho^{L},\rho^{R}}{\curvearrowright}\Cz(\G)$
given by $\gamma\mapsto\Ad(L_{\gamma})|_{\Cz(\G)},\Ad(R_{\gamma})|_{\Cz(\G)}\in\Aut(\Cz(\G))$
the left and right, respectively, translation actions. Note that the
left, resp.~right, quantum homomorphism associated with $\Gamma$
being a closed quantum subgroup of $\G$ is just the left, resp.~right,
coaction associated with the (left) action $\rho^{L}$, resp.~$\rho^{R}$;
these coactions are given by $\Cz(\G)\to\M(\Cz(\Gamma)\tensormin\Cz(\G))$,
$x\mapsto\left(\gamma\to\rho_{-\gamma}^{L}(x)\right)$, resp.~$\Cz(\G)\to\M(\Cz(\G)\tensormin\Cz(\Gamma))$,
$x\mapsto\left(\gamma\to\rho_{\gamma}^{R}(x)\right)$ \cite{Meyer_Roy_Woronowicz__hom_quant_grps,Daws_Kasprzak_Skalski_Soltan__closed_q_subgroups_LCQGs}
(see also \cite[pp.~2557\textendash 2558]{Brannan_Chirvasitu_Viselter__act_quo_lat}).
We are particularly interested here in the \emph{left-right translation
action} 
\[
\Gamma^{2}\overset{\rho}{\curvearrowright}\Cz(\G)\text{ given by }(\gamma_{1},\gamma_{2})\mapsto\rho_{\gamma_{1}}^{L}\circ\rho_{\gamma_{2}}^{R}=\Ad(L_{\gamma_{1}}R_{\gamma_{2}})|_{\Cz(\G)}.
\]
Define a $2$-cocycle $\dot{\Psi}$ on $\hat{\Gamma}$ by 
\[
\dot{\Psi}(\hat{\gamma}_{1},\hat{\gamma}_{2}):=\Psi(-\hat{\gamma}_{1},-\hat{\gamma}_{2}),\quad\hat{\gamma}_{1},\hat{\gamma}_{2}\in\hat{\Gamma}.
\]
We form the $2$-cocycle $\overline{\Psi}\boxtimes\dot{\Psi}$ on
$\hat{\Gamma}^{2}$ as in (\ref{eq:cocycle_tensor_prod}). So 
\[
(\overline{\Psi}\boxtimes\dot{\Psi})\big((\hat{\gamma}_{1},\hat{\gamma}_{1}'),(\hat{\gamma}_{2},\hat{\gamma}_{2}')\big):=\overline{\Psi}(\hat{\gamma}_{1},\hat{\gamma}_{2})\dot{\Psi}(\hat{\gamma}_{1}',\hat{\gamma}_{2}')=\overline{\Psi(\hat{\gamma}_{1},\hat{\gamma}_{2})}\Psi(-\hat{\gamma}_{1}',-\hat{\gamma}_{2}').
\]
We now put the \cst-algebra $\Cz(\G)$, the left-right translation
action $\Gamma^{2}\overset{\rho}{\curvearrowright}\Cz(\G)$, and the
continuous $2$-cocycle $\overline{\Psi}\boxtimes\dot{\Psi}$ on $\widehat{\Gamma^{2}}\cong\hat{\Gamma}^{2}$
as data into the Rieffel deformation of \cst-algebras construction
described in the previous subsection. It yields the `deformed dual'
action $\hat{\Gamma}^{2}\overset{\hat{\rho}^{\overline{\Psi}\boxtimes\dot{\Psi}}}{\curvearrowright}\Cz(\G)\rtimes_{\rho}\Gamma^{2}$
and the Landstad algebra $\Cz(\G)^{\overline{\Psi}\boxtimes\dot{\Psi}}\subseteq\M(\Cz(\G)\rtimes_{\rho}\Gamma^{2})$.

Staying in line with the previous subsection, we write $\iota:\Cz(\G)\to\M(\Cz(\G)\rtimes_{\rho}\Gamma^{2})$
for the canonical embedding and $\lambda$ for the canonical representation
of $\Gamma^{2}$ inside $\Cz(\G)\rtimes_{\rho}\Gamma^{2}$. Let $\lambda^{L},\lambda^{R}$
be the representations of $\Gamma$ inside $\Cz(\G)\rtimes_{\rho}\Gamma^{2}$
given by 
\[
\lambda_{\gamma}^{L}:=\lambda_{\gamma,0},\quad\lambda_{\gamma}^{R}:=\lambda_{0,\gamma}\qquad(\gamma\in\Gamma).
\]
For all $g,g_{1},g_{2}\in\Cz(\hat{\Gamma})$ we have 
\begin{gather}
\lambda(g_{1}\tensor g_{2})=\lambda^{L}(g_{1})\lambda^{R}(g_{2})\text{ (viewing }g_{1}\tensor g_{2}\text{ in }\Cz(\hat{\Gamma}^{2})\text{)},\nonumber \\
(\i\tensor\iota)(\hat{W}^{*})\cdot(\one\tensor\lambda^{L}(g))\cdot(\i\tensor\iota)(\hat{W})=(L\tensor\lambda^{L})\big(\Delta_{\hat{\Gamma}}(g)\big)\text{ by }\eqref{eq:L_gamma_group_like},\label{eq:lambda_L_lambda_R}\\
(\i\tensor\iota)(\hat{W})\cdot(\one\tensor\lambda^{R}(g))\cdot(\i\tensor\iota)(\hat{W}^{*})=(L\tensor\lambda^{R})\big(\Delta_{\hat{\Gamma}}(g)\big)\text{ by }\eqref{eq:R_gamma_group_like},\nonumber 
\end{gather}
where the two lower formulas follow by verifying them for $g$ in
the strictly total subset $\{\gamma:\hat{\Gamma}\to\C\mid\gamma\in\Gamma\}\subseteq\Cb(\hat{\Gamma})$
using the fact that the representations $\lambda^{L},\lambda^{R}$
implement the actions $\rho^{L},\rho^{R}$ inside $\Cz(\G)\rtimes_{\rho}\Gamma^{2}$.

Recall from (\ref{eq:Upsilon_Psi_ring}) the function $\mathring{\Psi}$
and from (\ref{eq:W_hat_Psi}) the formula $\hat{W}_{\Psi}=(L\tensor R)(\mathring{\Psi})\cdot\hat{W}\cdot(L\tensor L)(\overline{\Psi})$
for the left regular representation $\hat{W}_{\Psi}:=\hat{W}_{\Omega}$
of $\hat{\G}_{\Psi}:=\hat{\G}_{\Omega}$. Consider the unitary 
\begin{equation}
\hat{\mathcal{W}}_{\Psi}:=(L\tensor\lambda^{R})(\mathring{\Psi})\cdot(\i\tensor\iota)(\hat{W})\cdot(L\tensor\lambda^{L})(\overline{\Psi})\in\M\big(\Cz(\hat{\G})\tensormin(\Cz(\G)\rtimes_{\rho}\Gamma^{2})\big).\label{eq:cal_W_hat_Psi}
\end{equation}
Just as in \cite[Proposition~4.3]{Kasprzak__Rieffel_deform_crossed_prod}
one verifies that $\hat{\mathcal{W}}_{\Psi}$ is invariant under the
action $\i\tensor\hat{\rho}^{\overline{\Psi}\boxtimes\dot{\Psi}}$
of $\hat{\Gamma}^{2}$ on $\Cz(\hat{\G})\tensormin(\Cz(\G)\rtimes_{\rho}\Gamma^{2})$.
Indeed, similarly to (\ref{eq:dual_action_via_Fourier}), for all
$\hat{\gamma}_{1},\hat{\gamma}_{2}\in\hat{\Gamma}$ we have 
\[
\hat{\rho}_{\hat{\gamma}_{1},\hat{\gamma}_{2}}(\lambda^{L}(g))=\lambda^{L}(g(\cdot+\hat{\gamma}_{1}))\text{ and }\hat{\rho}_{\hat{\gamma}_{1},\hat{\gamma}_{2}}(\lambda^{R}(g))=\lambda^{R}(g(\cdot+\hat{\gamma}_{2}))\qquad(\forall_{g\in\Cz(\hat{\Gamma})}),
\]
hence 
\[
(\i\tensor\hat{\rho}_{\hat{\gamma}_{1},\hat{\gamma}_{2}})(\hat{\mathcal{W}}_{\Psi})=(L\tensor\lambda^{R})(\mathring{\Psi}(\cdot,\cdot+\hat{\gamma}_{2}))\cdot(\i\tensor\iota)(\hat{W})\cdot(L\tensor\lambda^{L})(\overline{\Psi}(\cdot,\cdot+\hat{\gamma}_{1})),
\]
and by the definition (\ref{eq:deformed_dual_action}) of the `deformed
dual' action and the equations (\ref{eq:lambda_L_lambda_R}) we deduce
that $(\i\tensor\hat{\rho}_{\hat{\gamma}_{1},\hat{\gamma}_{2}}^{\overline{\Psi}\boxtimes\dot{\Psi}})(\hat{\mathcal{W}}_{\Psi})$
equals
\[
(L\tensor\lambda^{R})\big(\mathring{\Psi}(\cdot,\cdot+\hat{\gamma}_{2})\Delta_{\hat{\Gamma}}(\dot{\Psi}_{\hat{\gamma}_{2}})(\one\tensor\overline{\dot{\Psi}}_{\hat{\gamma}_{2}})\big)\cdot(\i\tensor\iota)(\hat{W})\cdot(L\tensor\lambda^{L})\big((\one\tensor\overline{\Psi}_{\hat{\gamma}_{1}})\Delta_{\hat{\Gamma}}(\Psi_{\hat{\gamma}_{1}})\overline{\Psi}(\cdot,\cdot+\hat{\gamma}_{1})\big),
\]
which equals $\hat{\mathcal{W}}_{\Psi}$ by virtue of the $2$-cocycle
condition.

Next, precisely as in \cite[Theorem~4.4]{Kasprzak__Rieffel_deform_crossed_prod},
we conclude that $\{(\omega\tensor\i)(\hat{\mathcal{W}}_{\Psi}):\omega\in\Lone(\hat{\G})\}$
is a (norm-) dense subspace of the Landstad algebra $\Cz(\G)^{\overline{\Psi}\boxtimes\dot{\Psi}}$.
Indeed, the Landstad condition (L1) has just been established, and
(L2) follows because $(\Ad(\one\tensor\lambda_{\gamma_{1},\gamma_{2}}))\big((\i\tensor\iota)(\hat{W})\big)=(L_{-\gamma_{2}}\tensor\one)\cdot(\i\tensor\iota)(\hat{W})\cdot(L_{\gamma_{1}}\tensor\one)$
for all $\gamma_{1},\gamma_{2}\in\Gamma$ by (\ref{eq:L_gamma_group_like})
and (\ref{eq:R_gamma_group_like}). One then shows like in the proof
of \cite[Theorem~4.4]{Kasprzak__Rieffel_deform_crossed_prod} that
\begin{multline*}
\clinspan\{\lambda(g_{1})\cdot(\omega\tensor\i)(\hat{\mathcal{W}}_{\Psi})\cdot\lambda(g_{2}):\omega\in\Lone(\hat{\G}),g_{1},g_{2}\in\Cz(\hat{\Gamma}^{2})\}\\
=\clinspan\{\lambda(g_{1})\cdot(\omega\tensor\iota)(\hat{W})\cdot\lambda(g_{2}):\omega\in\Lone(\hat{\G}),g_{1},g_{2}\in\Cz(\hat{\Gamma}^{2})\},
\end{multline*}
which equals $\Cz(\G)\rtimes_{\rho}\Gamma^{2}$. This proves both
(L3) and the desired density by \cite[Lemma~2.6]{Kasprzak__Rieffel_deform_crossed_prod}.

By definition $\rho$ is spatially implemented on $\Ltwo(\G)$, thus
there exists a representation of $\Cz(\G)\rtimes_{\rho}\Gamma^{2}$
on $\Ltwo(\G)$, which we denote by $\pi^{\mathrm{can}}$ as in \cite{Kasprzak__Rieffel_deform_crossed_prod},
that restricts to the natural representation of $\Cz(\G)$ on $\Ltwo(\G)$
and maps $\lambda_{\gamma_{1},\gamma_{2}}$ to $L_{\gamma_{1}}R_{\gamma_{2}}$
for $(\gamma_{1},\gamma_{2})\in\Gamma^{2}$. Since $\pi^{\mathrm{can}}$
is plainly faithful on $\Cz(\G)$, it is also faithful on $\Cz(\G)^{\overline{\Psi}\boxtimes\dot{\Psi}}$
by \cite[Theorem~3.6]{Kasprzak__Rieffel_deform_crossed_prod}. Notice
that $\pi^{\mathrm{can}}(\lambda^{L}(g))=L(g)$ and $\pi^{\mathrm{can}}(\lambda^{R}(g))=R(g)$
for $g\in\Cz(\hat{\Gamma})$. By the foregoing and the definitions
(\ref{eq:W_hat_Psi}) and (\ref{eq:cal_W_hat_Psi}) of $\hat{W}_{\Psi}$
and $\hat{\mathcal{W}}_{\Psi}$, we have the following statement,
demonstrating the bifold nature of the Rieffel deformation. Recall
that $\GPsi:=\widehat{\hat{\G}_{\Psi}}$.
\begin{thm}
\label{thm:C0_G_Psi_crossed_prod}The left regular representation
$\hat{W}_{\Psi}\in\mathcal{U}(\Ltwo(\G)\tensor\Ltwo(\G))$ of $\hat{\G}_{\Psi}$
and the unitary $\hat{\mathcal{W}}_{\Psi}\in\M\big(\Cz(\hat{\G})\tensormin(\Cz(\G)\rtimes_{\rho}\Gamma^{2})\big)$
satisfy 
\[
\overline{\{(\omega\tensor\i)(\hat{\mathcal{W}}_{\Psi}):\omega\in\Lone(\hat{\G})\}}^{\left\Vert \cdot\right\Vert }=\Cz(\G)^{\overline{\Psi}\boxtimes\dot{\Psi}}\text{ and }(\i\tensor\pi^{\mathrm{can}})(\hat{\mathcal{W}}_{\Psi})=\hat{W}_{\Psi}
\]
(in the right equation we view both sides as acting on $\Ltwo(\G)\tensor\Ltwo(\G)$).
In particular, $\pi^{\mathrm{can}}$ maps $\Cz(\G)^{\overline{\Psi}\boxtimes\dot{\Psi}}$
injectively onto $\Cz(\GPsi)\subseteq B(\Ltwo(\G))$.
\end{thm}

Denote by $\pi^{\Psi}$ the isomorphism $\pi^{\mathrm{can}}|_{\Cz(\G)^{\overline{\Psi}\boxtimes\dot{\Psi}}}:\Cz(\G)^{\overline{\Psi}\boxtimes\dot{\Psi}}\to\Cz(\GPsi)$. 
\begin{rem*}
Note the following subtlety. We have the unitaries $\hat{W}_{\Psi}\in\M(\Cz(\hat{\G}_{\Psi})\tensormin\Cz(\GPsi=\widehat{\hat{\G}_{\Psi}}))$
and $\hat{\mathcal{W}}_{\Psi}\in\M\big(\Cz(\hat{\G})\tensormin(\Cz(\G)\rtimes_{\rho}\Gamma^{2})\big)$.
Theorem~\ref{thm:C0_G_Psi_crossed_prod} relates their slices by
$\pi^{\Psi}((\omega\tensor\i)(\hat{\mathcal{W}}_{\Psi}))=(\omega\tensor\i)(\hat{W}_{\Psi})$
for all $\omega\in\Lone(\hat{\G})$, a formula that we use several
times below. Nevertheless, we still cannot replace $(\i\tensor\pi^{\mathrm{can}})(\hat{\mathcal{W}}_{\Psi})=\hat{W}_{\Psi}$
by the naive equality $(\i\tensor\pi^{\Psi})(\hat{\mathcal{W}}_{\Psi})=\hat{W}_{\Psi}$
because the latter does not make sense as we do not know whether $\hat{\mathcal{W}}_{\Psi}\in\M\big(\Cz(\hat{\G})\tensormin\Cz(\G)^{\overline{\Psi}\boxtimes\dot{\Psi}}\big)$.
Also, note that in general $\Cz(\hat{\G}_{\Psi})$ may be different
from $\Cz(\hat{\G})$ \cite{DeCommer__cocycle_twist_CQG}. In the
unusual situation where $\Cz(\hat{\G}_{\Psi})$ does equal $\Cz(\hat{\G})$,
we could indeed deduce that $\hat{\mathcal{W}}_{\Psi}$ equals $(\i\tensor{(\pi^{\Psi})}^{-1})(\hat{W}_{\Psi})$,
which belongs to $\M\big(\Cz(\hat{\G})\tensormin\Cz(\G)^{\overline{\Psi}\boxtimes\dot{\Psi}}\big)$,
and that the `naive equality' holds.
\end{rem*}
We require the pull-back of the co-multiplication $\Delta_{\GPsi}$
of $\GPsi$ to $\Cz(\G)^{\overline{\Psi}\boxtimes\dot{\Psi}}$. The
morphism 
\[
(\iota\tensor\iota)\circ\Delta\in\Mor\big(\Cz(\G),(\Cz(\G)\rtimes_{\rho}\Gamma^{2})\tensormin(\Cz(\G)\rtimes_{\rho}\Gamma^{2})\big)
\]
together with the representation of $\Gamma^{2}$ inside $(\Cz(\G)\rtimes_{\rho}\Gamma^{2})\tensormin(\Cz(\G)\rtimes_{\rho}\Gamma^{2})$
given by $(\gamma_{1},\gamma_{2})\mapsto\lambda_{\gamma_{1}}^{L}\tensor\lambda_{\gamma_{2}}^{R}$,
$(\gamma_{1},\gamma_{2})\in\Gamma^{2}$, form a covariant representation
of $\Gamma^{2}\overset{\rho}{\curvearrowright}\Cz(\G)$ on account
of the intertwining relation 
\begin{equation}
\Delta\circ\rho_{\gamma_{1},\gamma_{2}}=(\rho_{\gamma_{1},0}\tensor\rho_{0,\gamma_{2}})\circ\Delta\qquad(\forall_{(\gamma_{1},\gamma_{2})\in\Gamma^{2}}),\label{eq:Delta_left_right_transl}
\end{equation}
which holds since 
\[
W^{*}\big(\one\tensor(\Ad(L_{\gamma_{1}}R_{\gamma_{2}}))(x)\big)W=W^{*}\big((\Ad(L_{\gamma_{1}}\tensor L_{\gamma_{1}}R_{\gamma_{2}}))(\one\tensor x)\big)W=(\Ad(L_{\gamma_{1}}\tensor R_{\gamma_{2}}))(W^{*}(\one\tensor x)W)
\]
for all $x\in B(\Ltwo(\G))$ by (\ref{eq:L_gamma_group_like}) (use
the flip) and the fact that $R_{\gamma}\in\Linfty(\hat{\G})'$. Consequently,
there exists (a unique) $\tilde{\Delta}\in\Mor\big(\Cz(\G)\rtimes_{\rho}\Gamma^{2},(\Cz(\G)\rtimes_{\rho}\Gamma^{2})\tensormin(\Cz(\G)\rtimes_{\rho}\Gamma^{2})\big)$
mapping $\iota(a)$ to $(\iota\tensor\iota)(\Delta(a))$ for $a\in\Cz(\G)$
and $\lambda_{\gamma_{1},\gamma_{2}}$ to $\lambda_{\gamma_{1}}^{L}\tensor\lambda_{\gamma_{2}}^{R}$
for $(\gamma_{1},\gamma_{2})\in\Gamma^{2}$. Finally, recall the function
$\Upsilon$ introduced in (\ref{eq:Upsilon_Psi_ring}) and consider
the unitary 
\begin{equation}
(\lambda^{R}\tensor\lambda^{L})(\Upsilon)\in\M\big((\Cz(\G)\rtimes_{\rho}\Gamma^{2})\tensormin(\Cz(\G)\rtimes_{\rho}\Gamma^{2})\big).\label{eq:R_L_Upsilon}
\end{equation}

\begin{thm}
\label{thm:Psi_comult_crossed_prod}The morphism 
\begin{equation}
\Ad((\lambda^{R}\tensor\lambda^{L})(\Upsilon))\circ\tilde{\Delta}\in\Mor\big(\Cz(\G)\rtimes_{\rho}\Gamma^{2},(\Cz(\G)\rtimes_{\rho}\Gamma^{2})\tensormin(\Cz(\G)\rtimes_{\rho}\Gamma^{2})\big)\label{eq:Psi_comult_crossed_prod}
\end{equation}
restricts to a morphism $\Delta^{\Psi}\in\Mor(\Cz(\G)^{\overline{\Psi}\boxtimes\dot{\Psi}},\Cz(\G)^{\overline{\Psi}\boxtimes\dot{\Psi}}\tensormin\Cz(\G)^{\overline{\Psi}\boxtimes\dot{\Psi}})$,
which the isomorphism $\pi^{\Psi}:\Cz(\G)^{\overline{\Psi}\boxtimes\dot{\Psi}}\to\Cz(\GPsi)$
intertwines with $\Delta_{\GPsi}$, i.e.~$(\pi^{\Psi}\tensor\pi^{\Psi})\circ\Delta^{\Psi}=\Delta_{\GPsi}\circ\pi^{\Psi}$. 
\end{thm}

The theorem is proved precisely as \cite[Theorem~4.11]{Kasprzak__Rieffel_deform_crossed_prod}
or \cite[Proposition~7]{Fima_Vainerman__twist_Rieffel_deform}. We
sketch the details for completeness.
\begin{proof}[Proof of Theorem~\ref{thm:Psi_comult_crossed_prod}]
By Theorem~\ref{thm:C0_G_Psi_crossed_prod} it suffices to show
that 
\begin{equation}
\big(\i\tensor(\Ad((\lambda^{R}\tensor\lambda^{L})(\Upsilon))\circ\tilde{\Delta})\big)(\hat{\mathcal{W}}_{\Psi})={(\hat{\mathcal{W}}_{\Psi})}_{13}{(\hat{\mathcal{W}}_{\Psi})}_{12}.\label{eq:comults_1}
\end{equation}
Indeed, suppose that the above formula has already been established.
Then for $\zeta,\eta\in\Ltwo(\G)$, we get from the equality $(\i\tensor\Delta_{\GPsi})(\hat{W}_{\Psi})={(\hat{W}_{\Psi})}_{13}{(\hat{W}_{\Psi})}_{12}$
and \cite[Lemma~A.5]{Kustermans_Vaes__LCQG_C_star} that 
\[
\Delta_{\GPsi}((\omega_{\zeta,\eta}\tensor\i)(\hat{W}_{\Psi}))=\sum_{i\in I}(\omega_{\zeta,e_{i}}\tensor\i)(\hat{W}_{\Psi})\tensor(\omega_{e_{i},\eta}\tensor\i)(\hat{W}_{\Psi}),
\]
where $\left(e_{i}\right)_{i\in I}$ is an orthonormal basis of $\Ltwo(\G)$
and the series, which has bounded partial sums, converges strictly
in $\M(\Cz(\GPsi)\tensormin\Cz(\GPsi))$. As a result, since $\pi^{\Psi}:\Cz(\G)^{\overline{\Psi}\boxtimes\dot{\Psi}}\to\Cz(\GPsi)$
is an isomorphism mapping $(\omega\tensor\i)(\hat{\mathcal{W}}_{\Psi})$
to $(\omega\tensor\i)(\hat{W}_{\Psi})$ for all $\omega\in\Lone(\hat{\G})$,
we have
\begin{equation}
\big(({(\pi^{\Psi})}^{-1}\tensor{(\pi^{\Psi})}^{-1})\circ\Delta_{\GPsi}\big)((\omega_{\zeta,\eta}\tensor\i)(\hat{W}_{\Psi}))=\sum_{i\in I}(\omega_{\zeta,e_{i}}\tensor\i)(\hat{\mathcal{W}}_{\Psi})\tensor(\omega_{e_{i},\eta}\tensor\i)(\hat{\mathcal{W}}_{\Psi}),\label{eq:comults_1.5}
\end{equation}
where the series has bounded partial sums and converges strictly in
$\M(\Cz(\G)^{\overline{\Psi}\boxtimes\dot{\Psi}}\tensormin\Cz(\G)^{\overline{\Psi}\boxtimes\dot{\Psi}})$.
On the other hand, by (\ref{eq:comults_1}), the series on the right-hand
side of (\ref{eq:comults_1.5}) converges strictly in $\M((\Cz(\G)\rtimes_{\rho}\Gamma^{2})\tensormin(\Cz(\G)\rtimes_{\rho}\Gamma^{2}))$
to the image of $(\omega_{\zeta,\eta}\tensor\i)(\hat{\mathcal{W}}_{\Psi})$
under $\Ad((\lambda^{R}\tensor\lambda^{L})(\Upsilon))\circ\tilde{\Delta}$.
Since on bounded subsets of $\M(\Cz(\G)^{\overline{\Psi}\boxtimes\dot{\Psi}}\tensormin\Cz(\G)^{\overline{\Psi}\boxtimes\dot{\Psi}})$
the strict topology is finer than the strict topology induced from
$\M((\Cz(\G)\rtimes_{\rho}\Gamma^{2})\tensormin(\Cz(\G)\rtimes_{\rho}\Gamma^{2}))$
(the `easy part' of Observation~\ref{obs:two_strict_top}), we
deduce that $\Ad((\lambda^{R}\tensor\lambda^{L})(\Upsilon))\circ\tilde{\Delta}$
maps $(\omega_{\zeta,\eta}\tensor\i)(\hat{\mathcal{W}}_{\Psi})$ to
the element of $\M(\Cz(\G)^{\overline{\Psi}\boxtimes\dot{\Psi}}\tensormin\Cz(\G)^{\overline{\Psi}\boxtimes\dot{\Psi}})\subseteq\M((\Cz(\G)\rtimes_{\rho}\Gamma^{2})\tensormin(\Cz(\G)\rtimes_{\rho}\Gamma^{2}))$
that is on the left-hand side of (\ref{eq:comults_1.5}). Therefore,
by density, $\Ad((\lambda^{R}\tensor\lambda^{L})(\Upsilon))\circ\tilde{\Delta}$
restricts to a map $\Delta^{\Psi}:\Cz(\G)^{\overline{\Psi}\boxtimes\dot{\Psi}}\to\M(\Cz(\G)^{\overline{\Psi}\boxtimes\dot{\Psi}}\tensormin\Cz(\G)^{\overline{\Psi}\boxtimes\dot{\Psi}})$
satisfying $(\pi^{\Psi}\tensor\pi^{\Psi})\circ\Delta^{\Psi}=\Delta_{\GPsi}\circ\pi^{\Psi}$.
Since $\Delta_{\GPsi}:\Cz(\GPsi)\to\M(\Cz(\GPsi)\tensormin\Cz(\GPsi))$
is non-degenerate (and, again, $\pi^{\Psi}:\Cz(\G)^{\overline{\Psi}\boxtimes\dot{\Psi}}\to\Cz(\GPsi)$
is an isomorphism), $\Delta^{\Psi}$ is also non-degenerate.

We proceed to prove (\ref{eq:comults_1}). Since $(\i\tensor\Delta)(\hat{W})=\hat{W}_{13}\hat{W}_{12}$,
(\ref{eq:comults_1}) is equivalent to 
\begin{multline*}
\Ad(\one\tensor(\lambda^{R}\tensor\lambda^{L})(\Upsilon))\big((L\tensor\lambda^{R})(\mathring{\Psi}))_{13}\cdot(\i\tensor\iota\tensor\iota)(\hat{W}_{13}\hat{W}_{12})\cdot(L\tensor\lambda^{L})(\overline{\Psi})_{12}\big)\\
=(L\tensor\lambda^{R})(\mathring{\Psi})_{13}\cdot(\i\tensor\iota)(\hat{W})_{13}\cdot(L\tensor\lambda^{R}\tensor\lambda^{L})(\overline{\Psi}_{13}\mathring{\Psi}_{12})\cdot(\i\tensor\iota)(\hat{W})_{12}\cdot(L\tensor\lambda^{L})(\overline{\Psi})_{12},
\end{multline*}
which in turn is the same as 
\begin{equation}
\Ad(\one\tensor(\lambda^{R}\tensor\lambda^{L})(\Upsilon))\big((\i\tensor\iota\tensor\iota)(\hat{W}_{13}\hat{W}_{12})\big)=(\i\tensor\iota)(\hat{W})_{13}\cdot(L\tensor\lambda^{R}\tensor\lambda^{L})(\overline{\Psi}_{13}\mathring{\Psi}_{12})\cdot(\i\tensor\iota)(\hat{W})_{12}.\label{eq:comults_2}
\end{equation}
 By (\ref{eq:lambda_L_lambda_R}) the left-hand side of (\ref{eq:comults_2})
is 
\[
(\i\tensor\iota)(\hat{W})_{13}\cdot(L\tensor\lambda^{R}\tensor\lambda^{L})\big((\sigma\tensor\id)((\i\tensor\Delta_{\hat{\Gamma}})(\Upsilon))\cdot(\Delta_{\hat{\Gamma}}\tensor\i)(\Upsilon^{*})\big)\cdot(\i\tensor\iota)(\hat{W})_{12}.
\]
One verifies by the $2$-cocycle condition that 
\[
(\sigma\tensor\id)((\i\tensor\Delta_{\hat{\Gamma}})(\Upsilon))\cdot(\Delta_{\hat{\Gamma}}\tensor\i)(\Upsilon^{*})=\overline{\Psi}_{13}\mathring{\Psi}_{12},
\]
proving (\ref{eq:comults_2}) as desired.
\end{proof}
In light of Theorems~\ref{thm:C0_G_Psi_crossed_prod} and \ref{thm:Psi_comult_crossed_prod},
we call the locally compact quantum group $\GPsi$ the \emph{Rieffel
deformation locally compact quantum group} constructed from $\G,\Gamma,\Psi$.

\subsection{Properties of the Rieffel deformation of locally compact quantum
groups\label{subsec:props_Rieffel_deform_LCQGs}}

Having described the Rieffel deformation locally compact quantum group
$\GPsi$ in terms of the Rieffel deformation \cst-algebra $\Cz(\G)^{\overline{\Psi}\boxtimes\dot{\Psi}}$,
we proceed to establish some further properties of $\GPsi$. We continue
to use Standing~Hypothesis~\ref{stand_hyp} and the notation of
the previous subsection.

Since the injective normal $*$-homomorphism $L:\Linfty(\hat{\Gamma})\to\Linfty(\hat{\G})=\Linfty(\hat{\G}_{\Psi})$
intertwines the co-multiplications of $\hat{\Gamma}$ and $\hat{\G}$,
it also intertwines the co-multiplications of $\hat{\Gamma}$ and
$\hat{\G}_{\Psi}=\widehat{\GPsi}$: 
\begin{equation}
\Delta_{\hat{\G}_{\Psi}}\circ L=(\Ad((L\tensor L)(\Psi)))\circ\Delta_{\hat{\G}}\circ L=(\Ad((L\tensor L)(\Psi)))\circ(L\tensor L)\circ\Delta_{\hat{\Gamma}}=(L\tensor L)\circ\Delta_{\hat{\Gamma}},\label{eq:Gamma_le_G_Psi}
\end{equation}
and thus makes $\Gamma$ into a closed quantum subgroup of the Rieffel
deformation $\GPsi$ (this is also mentioned in \cite[Section~4, p.~1028]{Fima_Vainerman__twist_Rieffel_deform}).
This, together with the observations made in the beginning of Subsection~\ref{subsec:cocycle_twisting},
shows that we can view $\G$ as the Rieffel deformation of $\GPsi$
using the conjugate cocycle $\overline{\Psi}$, i.e.~$(\GPsi)^{\overline{\Psi}}\cong\G$.
In Proposition~\ref{prop:from_G_Psi_to_G} below we discuss how this
isomorphism looks like from the perspective of the Rieffel deformation
of \cst-algebras.

\subsubsection{The co-unit and the antipode}

We say that \emph{$\Gamma\le\G$ factors through the centre $\mathscr{Z}(\G)$
of $\G$} \cite{Kasprzak_Skalski_Soltan__can_centr_exa_seq_LCQG}
if the image of the embedding $L:\Linfty(\hat{\Gamma})\to\Linfty(\hat{\G})$
is contained in $\Linfty(\widehat{\mathscr{Z}(\G)})$, the largest
central Baaj--Vaes subalgebra of $\Linfty(\hat{\G})$ \cite[Definition~2.2]{Kasprzak_Skalski_Soltan__can_centr_exa_seq_LCQG}.
In this case $R_{\gamma}=L_{-\gamma}$ for all $\gamma\in\Gamma$
and $R\left(g\right)=L(g(-\cdot))$ for all $g\in\Cz(\hat{\Gamma})$.
Note that if $\G$ is a locally compact group $G$, then $\Gamma\le G$
factors through the centre in the above sense if and only if $\Gamma$
is contained in the centre of $G$; this happens trivially if $G$
is commutative, and in particular if $G=\Gamma$.
\begin{prop}
\label{prop:Gamma_fact_thru_center}If $\Gamma\le\G$ factors through
the centre of $\G$ then $\hat{W}_{\Psi}=\hat{W}$ and hence $\GPsi=\G$,
and therefore $\pi^{\Psi}$ is an isomorphism from $\Cz(\G)^{\overline{\Psi}\boxtimes\dot{\Psi}}$
onto $\Cz(\G)$ and $(\pi^{\Psi}\tensor\pi^{\Psi})\circ\Delta^{\Psi}=\Delta_{\G}\circ\pi^{\Psi}$.
\end{prop}

\begin{proof}
Since $\Gamma\le\G$ factors through the centre of $\G$ we have $(L\tensor L)(\overline{\Psi})=(L\tensor R)(\overline{\Psi}(\cdot,-\cdot))$
(see above), so that by the definition (\ref{eq:W_hat_Psi}) of $\hat{W}_{\Psi}$,
\begin{equation}
\hat{W}_{\Psi}=(L\tensor R)(\mathring{\Psi})\cdot\hat{W}\cdot(L\tensor R)(\overline{\Psi}(\cdot,-\cdot)).\label{eq:Gamma_fact_thru_center__1}
\end{equation}
Denote by $\a$ the homeomorphism of $\hat{\Gamma}^{2}$ given by
$(\hat{\gamma}_{1},\hat{\gamma}_{2})\mapsto(\hat{\gamma}_{1},\hat{\gamma}_{1}+\hat{\gamma}_{2})$
for $\hat{\gamma}_{1},\hat{\gamma}_{2}\in\hat{\Gamma}$. We claim
that
\begin{equation}
\hat{W}\cdot(L\tensor R)(G)\cdot\hat{W}^{*}=(L\tensor R)(G\circ\a)\qquad(\forall_{G\in\Cz(\hat{\Gamma}^{2})}).\label{eq:Gamma_fact_thru_center__2}
\end{equation}
Indeed, (\ref{eq:Gamma_fact_thru_center__2}) holds for simple tensors
in $\Cz(\hat{\Gamma}^{2})\cong\Cz(\hat{\Gamma})\tensormin\Cz(\hat{\Gamma})$
by virtue of the assumption on $\Gamma$, namely that the image of
$L$ is contained in central subalgebra $\Linfty(\widehat{\mathscr{Z}(\G)})$
of $\Linfty(\hat{\G})$, and (\ref{eq:R_gamma_group_like}) (see also
(\ref{eq:lambda_L_lambda_R})). Combining (\ref{eq:Gamma_fact_thru_center__1})
and (\ref{eq:Gamma_fact_thru_center__2}) we get the desired equality
\[
\hat{W}_{\Psi}=(L\tensor R)(\mathring{\Psi})\cdot(L\tensor R)(\overline{\Psi}(\cdot,-\cdot)\circ\a)\cdot\hat{W}=\hat{W}
\]
because $\Psi(\cdot,-\cdot)\circ\a=\mathring{\Psi}$ (recall the definition
of $\mathring{\Psi}$ in (\ref{eq:Upsilon_Psi_ring})).

The rest of the assertions follow from Theorems~\ref{thm:C0_G_Psi_crossed_prod}
and \ref{thm:Psi_comult_crossed_prod}.
\end{proof}
Let us return to the general situation of Standing~Hypothesis~\ref{stand_hyp},
assuming simply that $\Gamma$ is a closed quantum subgroup of $\G$. 

We will require a concrete formula for the co-unit of $\GPsi$ that
fits our framework of the Rieffel deformation of \cst-algebras. To
this end we will assume in Proposition~\ref{prop:G_Psi_co_unit}
below that $\G$ is co-amenable. However, we emphasise that this requirement
could possibly be lifted by finding a description of $\CzU(\GPsi)$
as a Rieffel deformation of $\CzU(\G)$.

To prepare for Proposition~\ref{prop:G_Psi_co_unit} we note that
similarly to constructing the locally compact quantum group $\GPsi$
from $\Gamma\le\G$ and the $2$-cocycle $\Psi$ on $\hat{\Gamma}$,
one can construct the locally compact quantum group $\GammaPsi$ from
$\Gamma\le\Gamma$ and the same $\Psi$. By Proposition~\ref{prop:Gamma_fact_thru_center}
we have $\GammaPsi=\Gamma$. As usual, write $\rho$ for the left-right
translation action $\Gamma^{2}\curvearrowright\Cz(\G)$ and $\pi^{\Psi}$
for the isomorphism $\Cz(\G)^{\overline{\Psi}\boxtimes\dot{\Psi}}\to\Cz(\GPsi)$
introduced after Theorem~\ref{thm:C0_G_Psi_crossed_prod}, and now
also write $\Gamma^{2}\overset{\rho_{\Gamma}}{\curvearrowright}\Cz(\Gamma)$
and $\pi_{\Gamma}^{\Psi}:\Cz(\Gamma)^{\overline{\Psi}\boxtimes\dot{\Psi}}\to\Cz(\GammaPsi)=\Cz(\Gamma)$
for the analogous maps associated with $\Gamma\le\Gamma$ and $\Psi$.
\begin{prop}
\label{prop:G_Psi_co_unit}The locally compact quantum group $\G$
is co-amenable if and only if so is $\GPsi$. In this case the strong
quantum homomorphism $\Pi:\Cz(\G)\to\Cz(\Gamma)$ associated to $\Gamma$
being a closed quantum subgroup of $\G$ induces a canonical morphism
$\tilde{\Pi}\in\Mor(\Cz(\G)\rtimes_{\rho}\Gamma^{2},\Cz(\Gamma)\rtimes_{\rho_{\Gamma}}\Gamma^{2})$
(`acting on the $\Cz(\G)$ part') and this $\tilde{\Pi}$ restricts
to a morphism $\Pi^{\Psi}\in\Mor(\Cz(\G)^{\overline{\Psi}\boxtimes\dot{\Psi}},\Cz(\Gamma)^{\overline{\Psi}\boxtimes\dot{\Psi}})$.
Also $\pi_{\Gamma}^{\Psi}\circ\Pi^{\Psi}\circ{(\pi^{\Psi})}^{-1}\in\Mor(\Cz(\GPsi),\Cz(\GammaPsi=\Gamma))$
is the strong quantum homomorphism associated to $\Gamma$ being a
closed quantum subgroup of $\GPsi$, and $\epsilon_{\GPsi}:=\epsilon_{\Gamma}\circ\pi_{\Gamma}^{\Psi}\circ\Pi^{\Psi}\circ{(\pi^{\Psi})}^{-1}\in\Cz(\GPsi)^{*}$
is the co-unit of $\GPsi$ where $\epsilon_{\Gamma}\in\Cz(\Gamma)^{*}$
is the co-unit of $\Gamma$.
\end{prop}

\begin{proof}
Note that to show the first statement of the proposition it suffices
to prove the forward implication (the converse follows by symmetry).
So we assume that $\G$ is co-amenable.

The connection between $\Pi$ and $L$, which is slightly simplified
thanks to co-amenability, is
\begin{equation}
(\Pi\tensor\i)(W)=(\i\tensor L)(W_{\Gamma})\text{, equivalently }(\i\tensor\Pi)(\hat{W})=(L\tensor\i)(\hat{W}_{\Gamma}).\label{eq:Pi_L}
\end{equation}
This, the definition of $\rho^{L}$ and (\ref{eq:L_gamma_group_like})
imply that $\Pi$ intertwines $\rho^{L}$ with the left translation
action of $\Gamma$ on $\Cz(\Gamma)$. An analogous statement is true
for $\rho^{R}$ and thus also for $\rho$. Therefore, using the terminology
of \cite[Definition~3.7]{Kasprzak__Rieffel_deform_crossed_prod},
the pair $(\Pi,\id:\Gamma^{2}\to\Gamma^{2})$ is a morphism of deformation
data from $(\Cz(\G),\rho,\overline{\Psi}\boxtimes\dot{\Psi})$ to
$(\Cz(\Gamma),\rho_{\Gamma},\overline{\Psi}\boxtimes\dot{\Psi})$.
Hence \cite[Proposition~3.8]{Kasprzak__Rieffel_deform_crossed_prod}
applies, and yields that the induced morphism $\tilde{\Pi}\in\Mor(\Cz(\G)\rtimes_{\rho}\Gamma^{2},\Cz(\Gamma)\rtimes_{\rho_{\Gamma}}\Gamma^{2})$
(mapping $f\in\Cc(\Gamma^{2},\Cz(\G))$ to $\Pi\circ f\in\Cc(\Gamma^{2},\Cz(\Gamma))$)
 restricts to an element $\Pi^{\Psi}\in\Mor(\Cz(\G)^{\overline{\Psi}\boxtimes\dot{\Psi}},\Cz(\Gamma)^{\overline{\Psi}\boxtimes\dot{\Psi}})$,
and since $\Pi(\Cz(\G))=\Cz(\Gamma)$ we have $\Pi^{\Psi}(\Cz(\G)^{\overline{\Psi}\boxtimes\dot{\Psi}})=\Cz(\Gamma)^{\overline{\Psi}\boxtimes\dot{\Psi}}$. 

Consider the operators $\hat{W}_{\Psi,\Gamma},\hat{\mathcal{W}}_{\Psi,\Gamma}$
pertaining to the Rieffel deformation locally compact quantum group
$\GammaPsi$. By Proposition \ref{prop:Gamma_fact_thru_center} we
have $\hat{W}_{\Psi,\Gamma}=\hat{W}_{\Gamma}$ and thus $\GammaPsi=\Gamma$.
Keeping in mind Theorem~\ref{thm:C0_G_Psi_crossed_prod}, by (\ref{eq:Pi_L})
and the definition (\ref{eq:cal_W_hat_Psi}) of $\hat{\mathcal{W}}_{\Psi}$,
for all $\omega\in\Lone(\hat{\G})$ we have 
\[
\Pi^{\Psi}\big((\omega\tensor\i)(\hat{\mathcal{W}}_{\Psi})\big)=((\omega\circ L)\tensor\i)(\hat{\mathcal{W}}_{\Psi,\Gamma}),
\]
and consequently
\begin{equation}
\big(\pi_{\Gamma}^{\Psi}\circ\Pi^{\Psi}\circ{(\pi^{\Psi})}^{-1}\big)((\omega\tensor\i)(\hat{W}_{\Psi}))=((\omega\circ L)\tensor\i)(\hat{W}_{\Psi,\Gamma})=((\omega\circ L)\tensor\i)(\hat{W}_{\Gamma}).\label{eq:G_Psi_co_unit}
\end{equation}
On the other hand, the strong quantum homomorphism $\Pi_{\GPsi}:\CzU(\GPsi)\to\Cz(\Gamma)$
associated to $\Gamma$ being a closed quantum subgroup of $\GPsi$
maps $(\omega\tensor\i)(\hat{\wW}_{\Psi})$ to $((\omega\circ L)\tensor\i)(\hat{W}_{\Gamma})$
for each $\omega\in\Lone(\hat{\G})$, where $\hat{\wW}_{\Psi}$ is
the `right half-universal' left regular representation of $\widehat{\GPsi}=\hat{\G}_{\Psi}$.
Hence (\ref{eq:G_Psi_co_unit}) implies that $\Pi_{\GPsi}$ factors
through the reduced \cst-algebra $\Cz(\GPsi)$ via the $*$-homomorphism
$\pi_{\Gamma}^{\Psi}\circ\Pi^{\Psi}\circ{(\pi^{\Psi})}^{-1}:\Cz(\GPsi)\to\Cz(\Gamma)$.
Hence the co-unit $\epsilon_{\GPsi}$ of $\GPsi$, which equals $\epsilon_{\Gamma}\circ\Pi_{\GPsi}$,
factors through $\Cz(\GPsi)$. This entails that $\GPsi$ is co-amenable
and $\Pi_{\GPsi}=\pi_{\Gamma}^{\Psi}\circ\Pi^{\Psi}\circ{(\pi^{\Psi})}^{-1}$.
\end{proof}
In the next proposition we lift the unitary antipode $\Rant=\Rant_{\G}$
to the level of the crossed product $\Cz(\G)\rtimes_{\rho}\Gamma^{2}$
and use it to describe the unitary antipode $\Rant_{\GPsi}$ of the
Rieffel deformation locally compact quantum group $\GPsi$. We need
the following general description of the anti-representations of crossed
products.
\begin{lem}[{\cite[p.~1229]{Dadarlat_Hirshberg_Phillips__simp_nu_C_alg_not_eq_isom_op},
\cite{Buss_Sims__opp_alg_groupoid_C_alg}}]
\label{lem:opp_crossed_prod}Let $G\overset{\a}{\curvearrowright}A$
be an action of a locally compact group $G$ on a \cst-algebra $A$.
Consider the opposite \cst-algebra $A^{\op}$ and the action $G\overset{\a^{\op}}{\curvearrowright}A^{\op}$
given by $\a_{t}^{\op}(a^{\op}):=\a_{t}(a)^{\op}$ for $a\in A$,
$t\in G$. Then there is a canonical isomorphism $\left(A\rtimes_{\a}G\right)^{\op}\cong A^{\op}\rtimes_{\a^{\op}}G$
that for $f\in\Cc(G,A,\a)$ maps $f^{\op}$ to the element of $\Cc(G,A^{\op},\a^{\op})$
given by $t\mapsto f^{*}(t)^{\op,*}=\Delta(t^{-1})\a_{t}(f(t^{-1}))^{\op}$,
$t\in G$ (where $f^{*}$ is the involution of $f$ in $\Cc(G,A,\a)\subseteq A\rtimes_{\a}G$).
Consequently, the anti-representations of $A\rtimes_{\a}G$ are in
1-1 correspondence with the anti-covariant representations of $(G,A,\a)$,
namely the pairs $(\Xi,\mathscr{U})$ such that $\Xi$ is an anti-representation
of $A$ on a Hilbert space $\H$, $\mathscr{U}$ is an (unitary) anti-representation
of $G$ on $\H$, and $\Xi(\a_{t^{-1}}(a))=\mathscr{U}_{t}\Xi(a)\mathscr{U}_{t}^{*}$
for all $a\in A$ and $t\in G$. The anti-representation of $A\rtimes_{\a}G$
associated with such $(\Xi,\mathscr{U})$ maps $f\in\Cc(G,A,\a)$
to $\int_{G}\mathscr{U}_{t}\,\Xi(f(t))\d t$. A similar result holds
for reduced crossed products.
\end{lem}

\begin{prop}
\label{prop:R_G_Psi}There is a unique anti-automorphism $\tilde{\Rant}$
of $\Cz(\G)\rtimes_{\rho}\Gamma^{2}$ mapping the element $\int_{\Gamma^{2}}\iota(f(\gamma_{1},\gamma_{2}))\cdot\lambda_{\gamma_{1},\gamma_{2}}\d(\gamma_{1},\gamma_{2})$
to $\int_{\Gamma^{2}}\lambda_{-\gamma_{2},-\gamma_{1}}\cdot\iota(\Rant_{\G}(f(\gamma_{1},\gamma_{2})))\d(\gamma_{1},\gamma_{2})$
for $f\in\Cc(\Gamma^{2},\Cz(\G))$. Extending it to $\M(\Cz(\G)\rtimes_{\rho}\Gamma^{2})$
we have $\tilde{\Rant}(\Cz(\G)^{\overline{\Psi}\boxtimes\dot{\Psi}})=\Cz(\G)^{\overline{\Psi}\boxtimes\dot{\Psi}}$
and $\tilde{\Rant}|_{\Cz(\G)^{\overline{\Psi}\boxtimes\dot{\Psi}}}={(\pi^{\Psi})}^{-1}\circ\Rant_{\GPsi}\circ\pi^{\Psi}$.
\end{prop}

\begin{proof}
The map $\Rant$ is an anti-automorphism of $\Cz(\G)$ and the map
$\mathscr{U}:\Gamma^{2}\ni(\gamma_{1},\gamma_{2})\mapsto\lambda_{-\gamma_{2},-\gamma_{1}}\in\M(\Cz(\G)\rtimes_{\rho}\Gamma^{2})$
is a (unitary anti-) representation of $\Gamma^{2}$ inside $\Cz(\G)\rtimes_{\rho}\Gamma^{2}$.
The pair $\left(\iota\circ\Rant,\mathscr{U}\right)$ is an anti-covariant
representation of the action $\Gamma^{2}\overset{\rho}{\curvearrowright}\Cz(\G)$
in the sense that 
\[
\mathscr{U}_{\gamma_{1},\gamma_{2}}\cdot(\iota\circ\Rant)(x)\cdot\mathscr{U}_{\gamma_{1},\gamma_{2}}^{-1}=(\iota\circ\Rant)\left(\rho_{-\gamma_{1},-\gamma_{2}}(x)\right)\qquad(\forall_{x\in\Cz(\G)}\forall_{(\gamma_{1},\gamma_{2})\in\Gamma^{2}}).
\]
As a result, the anti-automorphism $\tilde{\Rant}$ of $\Cz(\G)\rtimes_{\rho}\Gamma^{2}$
described in the proposition's statement exists (uniquely) by Lemma~\ref{lem:opp_crossed_prod}.
Evidently $\tilde{\Rant}^{2}=\tilde{\Rant}$. Furthermore, a simple
calculation shows that, writing $\beta$ for the homeomorphism of
$\hat{\Gamma}^{2}$ mapping $(\hat{\gamma}_{1},\hat{\gamma}_{2})$
to $(-\hat{\gamma}_{2},-\hat{\gamma}_{1})$, we have $\tilde{\Rant}\circ\hat{\rho}_{\hat{\gamma}_{1},\hat{\gamma}_{2}}=\hat{\rho}_{\beta(\hat{\gamma}_{1},\hat{\gamma}_{2})}\circ\tilde{\Rant}$
for all $\hat{\gamma}_{1},\hat{\gamma}_{2}\in\hat{\Gamma}$ and $\tilde{\Rant}(\lambda(G))=\lambda(G\circ\beta)$
for all $G\in\Cz(\hat{\Gamma}^{2})$, from which one shows that $\tilde{\Rant}\circ\hat{\rho}_{\hat{\gamma}_{1},\hat{\gamma}_{2}}^{\overline{\Psi}\boxtimes\dot{\Psi}}=\hat{\rho}_{\beta(\hat{\gamma}_{1},\hat{\gamma}_{2})}^{\overline{\Psi}\boxtimes\dot{\Psi}}\circ\tilde{\Rant}$.
Consequently $\tilde{\Rant}(\Cz(\G)^{\overline{\Psi}\boxtimes\dot{\Psi}})\subseteq\Cz(\G)^{\overline{\Psi}\boxtimes\dot{\Psi}}$,
thus $\tilde{\Rant}(\Cz(\G)^{\overline{\Psi}\boxtimes\dot{\Psi}})=\Cz(\G)^{\overline{\Psi}\boxtimes\dot{\Psi}}$.

The modular operator $\hat{J}$ of $\hat{\G}$ implements the unitary
antipode $\Rant$. Thus, for all $x\in\Cz(\G)$ we have $\left(\pi^{\mathrm{can}}\circ\iota\circ\Rant\right)(x)=\Rant(x)=\hat{J}x^{*}\hat{J}=\hat{J}\pi^{\mathrm{can}}(\iota(x))^{*}\hat{J}$.
On the other hand, for every $\gamma\in\Gamma$ we have $\hat{J}L_{\gamma}\hat{J}=R_{\gamma}$
by definition, thus $\pi^{\mathrm{can}}(\lambda_{-\gamma_{2},-\gamma_{1}})=L_{-\gamma_{2}}R_{-\gamma_{1}}=\hat{J}R_{-\gamma_{2}}L_{-\gamma_{1}}\hat{J}=\hat{J}\left(L_{\gamma_{1}}R_{\gamma_{2}}\right)^{*}\hat{J}=\hat{J}\pi^{\mathrm{can}}(\lambda_{\gamma_{1},\gamma_{2}})^{*}\hat{J}$
for all $\gamma_{1},\gamma_{2}\in\Gamma$. All this amounts to 
\begin{equation}
(\pi^{\mathrm{can}}\circ\tilde{\Rant})(y)=\hat{J}\pi^{\mathrm{can}}(y)^{*}\hat{J}\qquad(\forall_{y\in\Cz(\G)\rtimes_{\rho}\Gamma^{2}}).\label{eq:R_G_Psi}
\end{equation}
Furthermore, recall from (\ref{eq:R_G_sup_Omega}) that the modular
operator $\hat{J}$ of $\hat{\G}$ also implements the unitary antipode
$\Rant_{\GPsi}$ (where as always, $\Linfty(\GPsi)$ is acting on
$\Ltwo(\G)$). Since $\tilde{\Rant}(\Cz(\G)^{\overline{\Psi}\boxtimes\dot{\Psi}})=\Cz(\G)^{\overline{\Psi}\boxtimes\dot{\Psi}}$,
we infer from (\ref{eq:R_G_Psi}) that $\pi^{\Psi}\circ\tilde{\Rant}|_{\Cz(\G)^{\overline{\Psi}\boxtimes\dot{\Psi}}}=\Rant_{\GPsi}\circ\pi^{\Psi}$,
i.e.~$\tilde{\Rant}|_{\Cz(\G)^{\overline{\Psi}\boxtimes\dot{\Psi}}}={(\pi^{\Psi})}^{-1}\circ\Rant_{\GPsi}\circ\pi^{\Psi}$.
\end{proof}

\subsubsection{Returning to the original locally compact quantum group}

The natural identification of $(\GPsi)^{\overline{\Psi}}$ with $\G$
as locally compact quantum groups (see the beginning of Subsection~\ref{subsec:props_Rieffel_deform_LCQGs})
also holds from the side of Rieffel deformations of \cst-algebras:
essentially, deforming the deformed \cst-algebra $\Cz(\G)^{\overline{\Psi}\boxtimes\dot{\Psi}}$
again with respect to $\overline{\Psi}$ rather than $\Psi$ gives
back $\Cz(\G)$ when making the appropriate identifications. This
is stated precisely in Proposition~\ref{prop:from_G_Psi_to_G} below.

Since $\Gamma$ is a closed quantum subgroup of $\GPsi$, there is
an associated left-right translation action $\Gamma^{2}\overset{}{\curvearrowright}\Cz(\GPsi)$,
which we denote by $\varrho^{\Psi}$ in Propositions~\ref{prop:left_right_transl_act_G_Psi}
and \ref{prop:from_G_Psi_to_G} below, analogous to the left-right
translation action $\Gamma^{2}\overset{\rho}{\curvearrowright}\Cz(\G)$.
In the next simple proposition we spell out the relationship between
this action and the canonical `deformed' action $\Gamma^{2}\overset{\rho^{\overline{\Psi}\boxtimes\dot{\Psi}}}{\curvearrowright}\Cz(\G)^{\overline{\Psi}\boxtimes\dot{\Psi}}$
given by $\rho_{\gamma_{1},\gamma_{2}}^{\overline{\Psi}\boxtimes\dot{\Psi}}=\Ad(\lambda_{\gamma_{1},\gamma_{2}})|_{\Cz(\G)^{\overline{\Psi}\boxtimes\dot{\Psi}}}$
(see (\ref{eq:deformed_action})).
\begin{prop}
\label{prop:left_right_transl_act_G_Psi}The isomorphism $\pi^{\Psi}:\Cz(\G)^{\overline{\Psi}\boxtimes\dot{\Psi}}\to\Cz(\GPsi)$
intertwines the left-right translation action $\Gamma^{2}\overset{\varrho^{\Psi}}{\curvearrowright}\Cz(\GPsi)$
and the `deformed' action $\Gamma^{2}\overset{\rho^{\overline{\Psi}\boxtimes\dot{\Psi}}}{\curvearrowright}\Cz(\G)^{\overline{\Psi}\boxtimes\dot{\Psi}}$,
i.e.~we have $\varrho_{\gamma_{1},\gamma_{2}}^{\Psi}=\pi^{\Psi}\circ\rho_{\gamma_{1},\gamma_{2}}^{\overline{\Psi}\boxtimes\dot{\Psi}}\circ{(\pi^{\Psi})}^{-1}$
for every $(\gamma_{1},\gamma_{2})\in\Gamma^{2}$.
\end{prop}

\begin{proof}
By the foregoing our identification of $\Gamma$ as a closed quantum
subgroup of $\GPsi$ is such that the $*$-homomorphism $L:\Linfty(\hat{\Gamma})\to\Linfty(\hat{\G})$
equals the analogous $*$-homomorphism $\Linfty(\hat{\Gamma})\to\Linfty(\hat{\G}_{\Psi}=\widehat{\GPsi})$.
Furthermore, since $\hat{J}=J_{\hat{\G}}$ equals $J_{\hat{\G}_{\Psi}=\widehat{\GPsi}}$
(see Subsection~\ref{subsec:cocycle_twisting}, where the latter
conjugation is denoted by $\hat{J}_{\Omega}$) and since $R_{\gamma}=\hat{J}L_{\gamma}\hat{J}$
for all $\gamma\in\Gamma$, also the $*$-homomorphism $R:\Linfty(\hat{\Gamma})\to\Linfty(\hat{\G})'$
equals the analogous $*$-homomorphism $\Linfty(\hat{\Gamma})\to\Linfty(\hat{\G}_{\Psi}=\widehat{\GPsi})'$.
Consequently, viewing $\Cz(\GPsi)$ as acting on $\Ltwo(\G)$, the
action $\Gamma^{2}\overset{\varrho^{\Psi}}{\curvearrowright}\Cz(\GPsi)$
is implemented, like the action $\Gamma^{2}\overset{\rho}{\curvearrowright}\Cz(\G)$,
by the unitaries $L_{\gamma_{1}}R_{\gamma_{2}}$, $(\gamma_{1},\gamma_{2})\in\Gamma^{2}$.
By the definition of $\pi^{\mathrm{can}}$ and its restriction $\pi^{\Psi}$
(just before and after Theorem~\ref{thm:C0_G_Psi_crossed_prod})
we deduce that $\varrho_{\gamma_{1},\gamma_{2}}^{\Psi}=\pi^{\Psi}\circ\rho_{\gamma_{1},\gamma_{2}}^{\overline{\Psi}\boxtimes\dot{\Psi}}\circ{(\pi^{\Psi})}^{-1}$
for all $(\gamma_{1},\gamma_{2})\in\Gamma^{2}$.
\end{proof}
As preparation for Proposition~\ref{prop:from_G_Psi_to_G}, note
that by Proposition~\ref{prop:left_right_transl_act_G_Psi} the $*$-isomorphism
$\pi^{\Psi}:\Cz(\G)^{\overline{\Psi}\boxtimes\dot{\Psi}}\to\Cz(\GPsi)$
induces a $*$-isomorphism 
\[
\widetilde{\pi^{\Psi}}:\Cz(\G)^{\overline{\Psi}\boxtimes\dot{\Psi}}\rtimes_{\rho^{\overline{\Psi}\boxtimes\dot{\Psi}}}\Gamma^{2}\to\Cz(\GPsi)\rtimes_{\varrho^{\Psi}}\Gamma^{2}.
\]
We identify $\Cz(\G)^{\overline{\Psi}\boxtimes\dot{\Psi}}\rtimes_{\rho^{\overline{\Psi}\boxtimes\dot{\Psi}}}\Gamma^{2}$
with $\Cz(\G)\rtimes_{\rho}\Gamma^{2}$, see the description of Landstad's
theorem in Subsection~\ref{subsec:Rieffel_deform_algebras}. We also
identify $(\GPsi)^{\overline{\Psi}}$ with $\G$ and thereby view
$\Cz((\GPsi)^{\overline{\Psi}})$ as equal to $\Cz(\G)$, both acting
on $\Ltwo(\G)$. Recall that $\iota:\Cz(\G)\to\M(\Cz(\G)\rtimes_{\rho}\Gamma^{2})$
is the canonical embedding.
\begin{prop}
\label{prop:from_G_Psi_to_G}Apply the constructions of Subsection~\ref{subsec:Rieffel_deform_LCQGs}
with $\GPsi$ and $\overline{\Psi}$ in place of $\G$ and $\Psi$,
viewing $\Gamma\le\GPsi$ as described above. Upon the above-mentioned
identification of $\Cz(\G)^{\overline{\Psi}\boxtimes\dot{\Psi}}\rtimes_{\rho^{\overline{\Psi}\boxtimes\dot{\Psi}}}\Gamma^{2}$
with $\Cz(\G)\rtimes_{\rho}\Gamma^{2}$, the \cst-algebra $\Cz(\GPsi)^{\Psi\boxtimes\dot{\overline{\Psi}}}\subseteq\M(\Cz(\GPsi)\rtimes_{\varrho^{\Psi}}\Gamma^{2})$
is mapped by ${(\widetilde{\pi^{\Psi}})}^{-1}$ to $\iota(\Cz(\G))$,
and the $*$-isomorphism $\Cz(\GPsi)^{\Psi\boxtimes\dot{\overline{\Psi}}}\to\Cz((\GPsi)^{\overline{\Psi}}\cong\G)$
analogous to $\pi^{\Psi}:\Cz(\G)^{\overline{\Psi}\boxtimes\dot{\Psi}}\to\Cz(\GPsi)$
is $\pi^{\mathrm{can}}\circ{(\widetilde{\pi^{\Psi}})}^{-1}|_{\Cz(\GPsi)^{\Psi\boxtimes\dot{\overline{\Psi}}}}=\iota^{-1}\circ{(\widetilde{\pi^{\Psi}})}^{-1}|_{\Cz(\GPsi)^{\Psi\boxtimes\dot{\overline{\Psi}}}}$.
\end{prop}

\begin{proof}
By \cite[Lemma~3.5 and its proof]{Kasprzak__Rieffel_deform_crossed_prod},
we have $\big(\Cz(\G)^{\overline{\Psi}\boxtimes\dot{\Psi}}\big)^{\Psi\boxtimes\dot{\overline{\Psi}}}=\iota(\Cz(\G))$
inside $\M(\Cz(\G)\rtimes_{\rho}\Gamma^{2})$ because the `doubly
deformed dual' action $(\hat{\rho}^{\overline{\Psi}\boxtimes\dot{\Psi}})^{\Psi\boxtimes\dot{\overline{\Psi}}}$
and the original dual action $\hat{\rho}$, both on $\Cz(\G)\rtimes_{\rho}\Gamma^{2}$,
are equal. This proves the first assertion.

The second assertion holds since the representation of $\Cz(\GPsi)\rtimes_{\varrho^{\Psi}}\Gamma^{2}$
on $\Ltwo(\G)$ analogous to the representation $\pi^{\mathrm{can}}$
of $\Cz(\G)\rtimes_{\rho}\Gamma^{2}$ on $\Ltwo(\G)$ equals $\pi^{\mathrm{can}}\circ{(\widetilde{\pi^{\Psi}})}^{-1}$.
This is true because of the definitions of these representations and
of $\widetilde{\pi^{\Psi}}$ and the fact that $\pi^{\Psi}=\pi^{\mathrm{can}}|_{\Cz(\G)^{\overline{\Psi}\boxtimes\dot{\Psi}}}$.
\end{proof}

\section{Convolution semigroups on Rieffel deformations\label{sec:convsemig}}

This final section applies the results established earlier in the
context of convolution semigroups of states on locally compact quantum
groups, leading to the main results of the paper.

\subsection{\label{subsec:conv_semig_basics}Convolution semigroup basics}

In this general subsection we introduce convolutions, convolution
operators and convolution semigroups and make several observations
on them.

Let $\G$ be a locally compact quantum group. The \emph{convolution}
of $\mu,\nu\in\CzU(\G)^{*}$ is $\mu\star\nu:=(\mu\tensor\nu)\circ\Delta^{\mathrm{u}}\in\CzU(\G)^{*}$.
The pair $\left(\CzU(\G)^{*},\star\right)$ is a (completely contractive
Banach) algebra with the co-unit $\epsilon$ as its unit. Using tacitly
the natural embeddings of Banach spaces $\Lone(\G)\hookrightarrow\Cz(\G)^{*}\hookrightarrow\CzU(\G)^{*}$
given by restriction and by composing with the quotient map, respectively,
each `smaller' algebra is an ideal in each `larger' algebra.

For $\mu\in\CzU(\G)^{*}$ the operators $\omega\mapsto\omega\star\mu$
and $\omega\mapsto\mu\star\omega$ on $\Lone(\G)$ are completely
bounded, and we denote by $\mathscr{L}_{\mu},\mathscr{R}_{\mu}\in B(\Linfty(\G))$,
respectively, their (completely bounded, normal) adjoints (we follow
the left/right convention of \cite{Runde__unif_cont_LCQG,Skalski_Viselter__convolution_semigroups}
extending the $L^{1}(\QG)$-bimodule actions on $L^{\infty}(\QG)$
rather than the opposite convention of, e.g., \cite{Junge_Neufang_Ruan__rep_thm_LCQG,Salmi_Skalski__idemp_states_LCQG}).
The linear maps $\mathscr{L}_{\mu},\mathscr{R}_{\mu}$ on $L^{\infty}(\QG)$
reduce to maps on $\Cz(\QG)$ \cite[Lemma~1.3]{Runde_Viselter_LCQGs_PosDef},
and for convenience, in this paper we will ordinarily write $\mathscr{L}_{\mu},\mathscr{R}_{\mu}$
for the maps on $\Cz(\QG)$. This should not lead to any confusion.
We call $\mathscr{L}_{\mu}$ and $\mathscr{R}_{\mu}$, $\mu\in\CzU(\G)^{*}$,
the left and right\emph{ convolution operators} on $\G$, respectively.

The operators $\mathscr{L}_{\mu},\mathscr{R}_{\mu}$ admit more concrete
formulas. For instance, if $\Ww\in\M(\CzU(\G)\tensormin K(\Ltwo(\G)))$
is the `left half-universal' left regular representation and
\[
\Delta^{\text{u-r}}\in\Mor(\Cz(\QG),\CzU(\G)\tensormin\Cz(\QG))
\]
is the `left half-universal' co-multiplication given by $\Delta^{\text{u-r}}(x):=\Ww^{*}(\one_{\CzU(\G)}\tensor x)\Ww$,
$x\in\Cz(\QG)$ \cite[Proposition~6.2~(2)]{Kustermans__LCQG_universal},
then $\mathscr{R}_{\mu}=(\mu\tensor\id)\circ\Delta^{\text{u-r}}$.
Remark that this formula extends to the original convolution operator
on $\Linfty(\G)$: for every $x\in\Linfty(\G)$ we have $\mathscr{R}_{\mu}x=(\mu\tensor\id)\left(\Ww^{*}(\one_{\CzU(\G)}\tensor x)\Ww\right)$,
where the right-hand side is \emph{a priori} in $\M(K(\Ltwo(\G)))$;
see \cite[Theorem~6]{Salmi_Skalski__idemp_states_LCQG_II} (based
on facts from \cite{Daws__CPM_LCQGs_2012,Junge_Neufang_Ruan__rep_thm_LCQG}
gathered earlier in \cite{Daws_Fima_Skalski_White_Haagerup_LCQG}).

\begin{defn}[\cite{Lindsay_Skalski__conv_semigrp_states,Skalski_Viselter__convolution_semigroups}]
A \emph{convolution semigroup} of states on $\G$ (or of $\CzU(\G)$)
is a one-parameter family $\left(\mu_{t}\right)_{t\ge0}$ of states
of $\CzU(\G)$ that forms a convolutive semigroup, i.e.
\[
\mu_{0}=\epsilon\text{ and }\mu_{s}\star\mu_{t}=\mu_{s+t}\text{ for all }s,t\ge0.
\]
Such $\left(\mu_{t}\right)_{t\ge0}$ is called $w^{*}$-continuous
if it is $w^{*}$-continuous in $t$ (at $0^{+}$ or, equivalently,
everywhere), and it is called symmetric if $\mu_{t}$ is invariant
under the universal unitary antipode $\Rant^{\mathrm{u}}$ for all
$t\ge0$.
\end{defn}

Convolution semigroups were studied in depth in \cite{Lindsay_Skalski__conv_semigrp_states,Skalski_Viselter__convolution_semigroups,Skalski_Viselter__generating_functionals}.
In particular, they were shown to be in $1$-$1$ correspondences
with other classes of semigroups \cite[Theorems~3.2 and 3.4]{Skalski_Viselter__convolution_semigroups}.
Since we will work with the `$\Cz(\QG)$-picture' of the convolution
semigroups, we would like to add the following complement of these
correspondences. Note that we only request the commutation relation
with `reduced' convolution operators.
\begin{thm}
\label{thm:SkV_Thm_3_2_C_0} Let $\QG$ be a locally compact quantum
group. 
\begin{enumerate}
\item \label{enu:SkV_Thm_3_2_C_0__1}There is a 1-1 correspondence between: 
\begin{itemize}
\item states $\mu$ of $\CzU(\QG)$;
\item completely positive maps $T$ on $\Cz(\QG)$ that have norm $1$ and
commute with the operators $\mathscr{L}_{\nu}$, $\nu\in\Cz(\QG)^{*}$;
equivalently, that are non-degenerate and satisfy the following commutation
relation: 
\begin{equation}
(T\ot\id)\circ\Delta_{\Cz(\QG)}=\Delta_{\Cz(\QG)}\circ T.\label{eq:right_conv_oper_comm_rel}
\end{equation}
\end{itemize}
The correspondence is given by $T=\mathscr{R}_{\mu}$.
\item \label{enu:SkV_Thm_3_2_C_0__2}There is a 1-1 correspondence between:
\begin{itemize}
\item $w^{*}$-continuous convolution semigroups $\left(\mu_{t}\right)_{t\ge0}$
of states of $\CzU(\QG)$;
\item $C_{0}$-semigroups $\left(T_{t}\right)_{t\ge0}$ of completely positive
maps on $\Cz(\QG)$ that have norm $1$ and commute with the operators
$\mathscr{L}_{\nu}$, $\nu\in\Cz(\QG)^{*}$; equivalently, that are
non-degenerate and satisfy the commutation relation (\ref{eq:right_conv_oper_comm_rel}).
\end{itemize}
The correspondence is given by $T_{t}=\mathscr{R}_{\mu_{t}}$ for
all $t\ge0$. 
\end{enumerate}
Similar results hold for left convolution operators.

\end{thm}

\begin{proof}
\ref{enu:SkV_Thm_3_2_C_0__1}. By \cite[Theorem~2.1~(e)]{Skalski_Viselter__convolution_semigroups},
the set of all completely positive maps $T$ on $\Cz(\QG)$ that commute
with the operators $\mathscr{L}_{\omega}$, $\omega\in L^{1}(\QG)$,
equals $\{\mathscr{R}_{\mu}:\mu\in\CzU(\QG)_{+}^{*}\}$, and the map
$\CzU(\QG)_{+}^{*}\ni\mu\mapsto\mathscr{R}_{\mu}$ is 1-1 and norm
preserving. Thus, for a completely positive $T$, the commutation
relation is \emph{a posteriori} equivalent to commuting with the operators
$\mathscr{L}_{\nu}$, $\nu\in\Cz(\QG)^{*}$, and also to $T$ being
strict and satisfying the commutation relation (\ref{eq:right_conv_oper_comm_rel}).
Evidently, such $T$ is non-degenerate if and only if the associated
$\mu$ is a state if and only if $\left\Vert T\right\Vert =1$. 

\ref{enu:SkV_Thm_3_2_C_0__2}. The map $\mu\mapsto\mathscr{R}_{\mu}$
is an anti-homomorphism, so the proof is complete using \cite[Theorem~4.6]{Runde_Viselter_LCQGs_PosDef}.
\end{proof}
The following simple lemma is merely a slight extension of \cite[Lemma~2.3, (a)$\iff$(c)]{Skalski_Viselter__convolution_semigroups}.
\begin{lem}
\label{lem:convolution_oper_converg}Let $\left(\mu_{i}\right)_{i\in\netdirset I}$
be a net in $\CzU(\G)^{*}$ and $\mu\in\CzU(\G)^{*}$. 
\begin{enumerate}
\item \label{enu:convolution_oper_converg__1}If $\left(\mu_{i}\right)_{i\in\netdirset I}$
converges to $\mu$ in the $w^{*}$-topology of $\CzU(\G)^{*}$, then
$\left(\mathscr{R}_{\mu_{i}}\right)_{i\in\netdirset I}$ converges
to $\mathscr{R}_{\mu}$ in the point--weak topology, i.e.~$\nu(\mathscr{R}_{\mu_{i}}(x))\xrightarrow[i\in\netdirset I]{}\nu(\mathscr{R}_{\mu}(x))$
for all $x\in\Cz(\G)$, $\nu\in\Cz(\G)^{*}$.
\item \label{enu:convolution_oper_converg__2}Suppose that $\left(\mu_{i}\right)_{i\in\netdirset I}$
is bounded and that $\nu(\mathscr{R}_{\mu_{i}}(x))\xrightarrow[i\in\netdirset I]{}\nu(\mathscr{R}_{\mu}(x))$
for all $x\in\Cz(\G)$ and all $\nu$ in a subspace $\mathbb{O}\subseteq\Cz(\G)^{*}$,
where \textup{$\mathbb{O}$} is such that if $\rho\in\CzU(\QG)^{*}$
and $\rho\star\nu=0$ for all $\nu\in\mathbb{O}$ then $\rho=0$.
Then $\mu_{i}\xrightarrow[i\in\netdirset I]{}\mu$ in the $w^{*}$-topology.
The condition on $\mathbb{O}$ is satisfied if $\mathbb{O}$ contains
$\left\{ \nu_{0}x:x\in\Cz(\G)\right\} \subseteq\Cz(\G)^{*}$ for some
non-zero $\nu_{0}\in\Cz(\G)^{*}$.
\end{enumerate}
\end{lem}

\begin{proof}
Recall the cancellation law $\clinspan\Delta^{\mathrm{u}}(\CzU(\G))(\one\tensor\CzU(\G))=\CzU(\G)\tensormin\CzU(\G)$,
which, by applying the canonical surjection $\CzU(\G)\to\Cz(\G)$
to the right tensor leg, implies the cancellation law $\clinspan\Delta^{\text{u-r}}(\Cz(\G))(\one\tensor\Cz(\G))=\CzU(\G)\tensormin\Cz(\G)$.

For all $\mu\in\CzU(\G)^{*}$, $\nu\in\Cz(\G)^{*}$, and $x\in\Cz(\G)$
we have $\nu(\mathscr{R}_{\mu}(x))=\mu\big((\id\tensor\nu)(\Delta^{\text{u-r}}(x))\big)$
and $(\id\tensor\nu)(\Delta^{\text{u-r}}(x))$ belongs to $\CzU(\G)$
(rather than just to $\M(\CzU(\G))$) by virtue of the cancellation
law. Thus \ref{enu:convolution_oper_converg__1} follows.

\ref{enu:convolution_oper_converg__2}. By the previous paragraph,
the limit assumption amounts to having $\mu_{i}(y)\xrightarrow[i\in\netdirset I]{}\mu(y)$
for all $y\in\left\{ (\id\tensor\nu)(\Delta^{\text{u-r}}(x)):x\in\Cz(\G),\nu\in\mathbb{O}\right\} \subseteq\CzU(\G)$.
The latter set is total in $\CzU(\G)$ by the Hahn--Banach theorem
and the assumption on $\mathbb{O}$ because a functional $\rho\in\CzU(\QG)^{*}$
annihilates that set if and only if $\rho\star\nu=0$ for all $\nu\in\mathbb{O}$.
So $\mu_{i}\xrightarrow[i\in\netdirset I]{}\mu$ in the $w^{*}$-topology
as $\left(\mu_{i}\right)_{i\in\netdirset I}$ is bounded. The last
statement of part~\ref{enu:convolution_oper_converg__2} follows
from the cancellation law.
\end{proof}
The conditions on $\mathbb{O}$ in part~\ref{enu:convolution_oper_converg__2}
hold if the maps $\left(\mathscr{R}_{\mu_{i}}\right)_{i\in\netdirset I},\mathscr{R}_{\mu}$
extend nicely to maps on a bigger \emph{von Neumann} algebra; this
is exploited in the next lemma.
\begin{lem}
\label{lem:convolution_oper_converg_by_ext}Let $\left(\mu_{i}\right)_{i\in\netdirset I}$
be a bounded net in $\CzU(\G)^{*}$ and $\mu\in\CzU(\G)^{*}$. Suppose
that there exist a von Neumann algebra $M$, a \cst-subalgebra $A\subseteq M$
with a $*$-isomorphism $\pi:A\to\Cz(\G)$, and maps $\left(\overline{\mathscr{R}_{\mu_{i}}}\right)_{i\in\netdirset I},\overline{\mathscr{R}_{\mu}}$
on $M$ that extend $\left(\mathscr{R}_{\mu_{i}}\right)_{i\in\netdirset I},\mathscr{R}_{\mu}$
via $\pi$, e.g.~$\overline{\mathscr{R}_{\mu}}|_{A}=\pi^{-1}\circ\mathscr{R}_{\mu}\circ\pi$,
such that $\overline{\mathscr{R}_{\mu_{i}}}\xrightarrow[i\in\netdirset I]{}\overline{\mathscr{R}_{\mu}}$
point--ultraweakly. Then $\mu_{i}\xrightarrow[i\in\netdirset I]{}\mu$
in the $w^{*}$-topology.
\end{lem}

\begin{proof}
Since the predual $M_{*}$ equals $\left\{ \omega y:\omega\in M_{*},y\in M\right\} $,
the (non-zero) subspace $\mathbb{O}:=\big\{\omega|_{A}\circ\pi^{-1}:\omega\in M_{*}\big\}$
of $\Cz(\G)^{*}$ contains $\left\{ \nu x:\nu\in\mathbb{O},x\in\Cz(\G)\right\} $.
Surely $\nu(\mathscr{R}_{\mu_{i}}(x))\xrightarrow[i\in\netdirset I]{}\nu(\mathscr{R}_{\mu}(x))$
for all $x\in\Cz(\G)$ and $\nu\in\mathbb{O}$. All conditions of
Lemma~\ref{lem:convolution_oper_converg}~\ref{enu:convolution_oper_converg__2}
are thus satisfied.
\end{proof}

\subsection{Completely positive maps on Rieffel deformations}

We state in this subsection a general result, Proposition~\ref{prop:single map},
on the functoriality of the Rieffel deformation involving non-degenerate
completely positive maps. For $*$-homomorphisms this was shown in
\cite[Subsection~3.2]{Kasprzak__Rieffel_deform_crossed_prod}. See
also \cite{Meyer_Roy_Woronowicz__quant_grp_tw_tens_prod_C_alg}. We
use the setting and notation of Subsection~\ref{subsec:Rieffel_deform_algebras},
assuming that we are given an action $\Gamma\overset{\rho}{\curvearrowright}A$
of a locally compact abelian group $\Gamma$ on a \cst-algebra $A$
and a continuous $2$-cocycle $\Psi$ on $\hat{\Gamma}$.
\begin{prop}
\label{prop:single map}Suppose that $S:A\to A$ is a non-degenerate
completely positive map commuting with the action $\Gamma\overset{\rho}{\curvearrowright}A$
in the obvious sense: $\rho_{\gamma}\circ S=S\circ\rho_{\gamma}$
for each $\gamma\in\Gamma$. Then the canonical completely positive
`extension' $\tilde{S}$ of $S$ to the crossed product $A\rtimes_{\rho}\Gamma$
is non-degenerate and preserves the algebra $A^{\Psi}$, and the respective
restriction $S^{\Psi}:A^{\Psi}\to A^{\Psi}$ is a  completely positive
map of norm $1$. The construction respects composition: if $T:A\to A$
is another map satisfying the assumptions above then we have $(S\circ T)^{\Psi}=S^{\Psi}\circ T^{\Psi}$.
\end{prop}

\begin{proof}
Note first that whenever $S$ is completely positive and equivariant
with respect to the $\Gamma$-action, it is well-known that it `extends'
to a completely positive map on the crossed product $A\rtimes_{\rho}\Gamma$,
which we denote by $\tilde{S}$, and which is determined by the formula
\begin{equation}
(\tilde{S}(f))(\gamma)=S(f(\gamma))\qquad(\forall_{f\in\Cc(\Gamma,A),\gamma\in\Gamma})\label{eq:Sformula}
\end{equation}
(this follows for example from \cite[Theorem~4.9 and Corollary~4.18]{Buss_Echterhoff_Willett__exotic_cr_prod_BC_conj}).
Note that the proofs in \cite{Buss_Echterhoff_Willett__exotic_cr_prod_BC_conj}
show that the norm of $\tilde{S}$ equals to that of $S$, essentially
as one can always work in the trivial unitisation of $A$ and suitably
rescale to work with the unital case.

We assume henceforth that the completely positive map $S:A\to A$
is non-degenerate and show that the completely positive map $\tilde{S}:A\rtimes_{\rho}\Gamma\to A\rtimes_{\rho}\Gamma$
is also non-degenerate.\emph{ }Consider the following canonical
contractive Banach algebraic approximate identity of $A\rtimes_{\rho}\Gamma$
\cite[footnote to Remark~A.8 (2)]{Echterhoff_Kaliszewski_Quigg_Raeburn__cat_appr_impr_thm_C_dyn_sys}.
Let $\left(a_{i}\right)_{i\in\netdirset I}$ and $\left(z_{j}\right)_{j\in\netdirset J}$
be approximate identities of $A$ and $\CStar(\Gamma)$, respectively.
Then $\left(\iota(a_{i})\lambda(z_{j})\right)_{(i,j)\in\netdirset I\times\netdirset J}$
is a contractive Banach algebraic approximate identity of $A\rtimes_{\rho}\Gamma$.
Since $S$ is non-degenerate, $Sa_{i}\xrightarrow[i\in\netdirset I]{}\one_{\M(A)}$
strictly in $\M(A)$. As a result, $\tilde{S}(\iota(a_{i})\lambda(z_{j}))=\iota(Sa_{i})\lambda(z_{j})\xrightarrow[(i,j)\in\netdirset I\times\netdirset J]{}\one_{\M(A\rtimes_{\rho}\Gamma)}$
strictly in $\M(A\rtimes_{\rho}\Gamma)$. Hence, $\tilde{S}$ is non-degenerate
by Lemma~\ref{lem:nondeg_CP_Banach_alg_approx_id}.

\emph{}Note that as $U_{\hat{\gamma}}=\lambda(\Psi_{\hat{\gamma}})$
for each $\hat{\gamma}\in\hat{\Gamma}$, we have $\tilde{S}(U_{\hat{\gamma}})=U_{\hat{\gamma}}$
by the defining formula (\ref{eq:Sformula}), so that $U_{\hat{\gamma}}$
is in the multiplicative domain of $\tilde{S}$. On the other hand
(\ref{eq:Sformula}) also shows that $\tilde{S}$ commutes with the
dual action $\hat{\rho}$. Thus $\tilde{S}$ commutes with the `deformed
dual' action $\hat{\rho}^{\Psi}$, and so it preserves its fixed-point
space. The multiplicative domain argument as in the first line of
this paragraph shows that 
\begin{gather*}
\tilde{S}(\lambda(f)x\lambda(g))=\lambda(f)\tilde{S}(x)\lambda(g)\qquad(\forall_{x\in\M(A\rtimes_{\rho}\Gamma),f,g\in\Cz(\hat{\Gamma})}),\\
\tilde{S}(\lambda_{\gamma}x\lambda_{\gamma}^{*})=\lambda_{\gamma}\tilde{S}(x)\lambda_{\gamma}^{*}\qquad(\forall_{x\in\M(A\rtimes_{\rho}\Gamma),\gamma\in\Gamma}).
\end{gather*}
This shows that $\tilde{S}$ preserves $A^{\Psi}$. Set $S^{\Psi}:=\tilde{S}|_{A^{\Psi}}:A^{\Psi}\to A^{\Psi}$.
Since $\tilde{S}$ is contractive, so is $S^{\Psi}$. Moreover, if
$\left(x_{i}\right)_{i\in\netdirset I}$ is an approximate identity
for $A^{\Psi}$, then since $A^{\Psi}$ is a non-degenerate subalgebra
of $\M(A\rtimes_{\rho}\Gamma)$, we have $x_{i}\xrightarrow[i\in\netdirset I]{}\one_{\M(A\rtimes_{\rho}\Gamma)}$
strictly in $\M(A\rtimes_{\rho}\Gamma)$, and the non-degeneracy of
$\tilde{S}$ implies that $S^{\Psi}x_{i}=\tilde{S}x_{i}\xrightarrow[i\in\netdirset I]{}\one_{\M(A\rtimes_{\rho}\Gamma)}$
strictly in $\M(A\rtimes_{\rho}\Gamma)$. Hence we must have $\left\Vert S^{\Psi}\right\Vert \ge1$.
All in all, $\left\Vert S^{\Psi}\right\Vert =1$.

The last assertion about respecting composition is obvious. 
\end{proof}
\begin{rem}
In Proposition~\ref{prop:single map} the existence of $\tilde{S}$
as well as its non-degeneracy assuming that of $S$ are true for actions
of arbitrary, not necessarily abelian, locally compact groups $\Gamma$
and for every crossed-product functor satisfying the equivalent conditions
of \cite[Theorem~4.9]{Buss_Echterhoff_Willett__exotic_cr_prod_BC_conj},
in particular for universal and reduced crossed products.
\end{rem}

We were not able to answer the following pertinent question:
\begin{question}
\label{ques:non_deg_deformed_CP_map}In Proposition~\ref{prop:single map},
is it always true that the restriction $S^{\Psi}:A^{\Psi}\to A^{\Psi}$
is non-degenerate?
\end{question}

Our main aim in this paper is to deform \emph{quantum convolution
semigroups}, which will be achieved in Theorem~\ref{thm:mainconv}.
To deduce the required continuity it is tempting to try to use directly
\cst-algebraic arguments; this would be possible if we could provide
a positive answer to the following question.
\begin{question}
\label{ques:deformed_C0_semig}Suppose that ${(S_{t})}_{t\geq0}$
is a point--norm continuous semigroup of non-degenerate completely
positive maps on $A$ with $S_{0}=\id_{A}$, commuting with the action
$\Gamma\overset{\rho}{\curvearrowright}A$, i.e.~such that $\rho_{\gamma}\circ S_{t}=S_{t}\circ\rho_{\gamma}$
for each $t\geq0$, $\gamma\in\Gamma$. Is it then true that the deformed
semigroup ${(S_{t}^{\Psi})}_{t\geq0}$ of  completely positive maps
on $A^{\Psi}$ is point--norm continuous?
\end{question}

In the specific context of Theorem~\ref{thm:mainconv} below we could
answer Question~\ref{ques:deformed_C0_semig} affirmatively with
the aid of \cite[Lemma~2.3]{Skalski_Viselter__convolution_semigroups},
without appealing to von Neumann algebras as in Lemma~\ref{lem:convolution_oper_converg_by_ext},
if we knew the following very general question had an affirmative
answer.
\begin{question}
\label{ques:point_norm_point_strict}Let ${(S_{t})}_{t\geq0}$ be
a $C_{0}$-semigroup of non-degenerate completely positive maps on
a \cst-algebra $A$. Is it then true that the semigroup on $\M(A)$
obtained by canonically extending each operator $S_{t}$ is point--strict
continuous?
\end{question}

We have been informed by Orr M.~Shalit and Michael Skeide that they
were recently able to solve Question~\ref{ques:point_norm_point_strict}
affirmatively in \cite{Shalit_Skeide__cont_strict_CP_semig}.

\subsection{Deforming convolution semigroups}

We are ready to present the main results of this paper. The setting
is that of Standing~Hypothesis~\ref{stand_hyp}, i.e.~we assume
that $\G$ is a locally compact quantum group, $\Gamma$ is a locally
compact abelian group that is a closed quantum subgroup of $\G$,
and $\Psi$ is a normalised continuous $2$-cocycle on $\hat{\Gamma}$.
The notation we use is that of Subsections~\ref{subsec:Rieffel_deform_LCQGs}
and \ref{subsec:props_Rieffel_deform_LCQGs}.

By (\ref{eq:L_gamma_group_like}) and (\ref{eq:R_gamma_group_like})
the (commuting) left and right translation actions $\Gamma\overset{\rho^{L},\rho^{R}}{\curvearrowright}\Cz(\G)$
given by $\gamma\mapsto\Ad(L_{\gamma})|_{\Cz(\G)},\Ad(R_{\gamma})|_{\Cz(\G)}$
satisfy  
\begin{align}
\begin{gathered}\Delta\circ\rho_{\gamma}^{L}=(\rho_{\gamma}^{L}\tensor\i)\circ\Delta\text{ and }\Delta\circ\rho_{\gamma}^{R}=(\i\tensor\rho_{\gamma}^{R})\circ\Delta,\\
(\rho_{\gamma}^{R}\tensor\rho_{\gamma}^{L})\circ\Delta=\Delta
\end{gathered}
 & \qquad(\forall_{\gamma\in\Gamma})\label{eq:left_right_translation_actions}
\end{align}
(the first two identities were already noted in (\ref{eq:Delta_left_right_transl}),
and they also follow from the fact that $\rho^{L},\rho^{R}$ are the
left and right, respectively, quantum homomorphisms associated with
$\Gamma\le\G$).

Let $\pi_{\mathrm{u}}:\CzU(\G)\to\Cz(\G)$ be the canonical surjection.
By (\ref{eq:left_right_translation_actions}) and \cite[Proposition~7.1]{Kustermans__LCQG_universal}
the actions $\Gamma\overset{\rho^{L},\rho^{R}}{\curvearrowright}\Cz(\G)$
lift uniquely to (commuting) actions $\Gamma\overset{\rho^{L,\mathrm{u}},\rho^{R,\mathrm{u}}}{\curvearrowright}\CzU(\G)$
such that 
\begin{align}
\begin{gathered}\rho_{\gamma}^{L}\circ\pi_{\mathrm{u}}=\pi_{\mathrm{u}}\circ\rho_{\gamma}^{L,\mathrm{u}}\text{ and }\rho_{\gamma}^{R}\circ\pi_{\mathrm{u}}=\pi_{\mathrm{u}}\circ\rho_{\gamma}^{R,\mathrm{u}},\\
\Delta^{\mathrm{u}}\circ\rho_{\gamma}^{L,\mathrm{u}}=(\rho_{\gamma}^{L,\mathrm{u}}\tensor\i)\circ\Delta^{\mathrm{u}}\text{ and }\Delta^{\mathrm{u}}\circ\rho_{\gamma}^{R,\mathrm{u}}=(\i\tensor\rho_{\gamma}^{R,\mathrm{u}})\circ\Delta^{\mathrm{u}}
\end{gathered}
 & \qquad(\forall_{\gamma\in\Gamma}),\label{eq:univ_left_right_translation_actions}
\end{align}
and these actions satisfy 
\begin{equation}
(\rho_{\gamma}^{R,\mathrm{u}}\tensor\rho_{\gamma}^{L,\mathrm{u}})\circ\Delta^{\mathrm{u}}=\Delta^{\mathrm{u}}\qquad(\forall_{\gamma\in\Gamma});\label{eq:univ_left_right_translation_actions_2}
\end{equation}
we leave filling the details to the reader. The resulting left-right
translation action $\Gamma^{2}\overset{\rho^{\mathrm{u}}}{\curvearrowright}\CzU(\G)$
is given by $(\gamma_{1},\gamma_{2})\mapsto\rho_{\gamma_{1}}^{L,\mathrm{u}}\circ\rho_{\gamma_{2}}^{R,\mathrm{u}}$,
and its `diagonal' is called the \emph{adjoint action} $\Gamma\overset{\mathrm{ad}^{\mathrm{u}}}{\curvearrowright}\CzU(\G)$;
namely, $\mathrm{ad}_{\gamma}^{\mathrm{u}}:=\rho_{\gamma}^{L,\mathrm{u}}\circ\rho_{\gamma}^{R,\mathrm{u}}$
for each $\gamma\in\Gamma$.

Our intention is to deform convolution semigroups on $\G$, so the
next proposition plays a key role. Recall from Theorem~\ref{thm:C0_G_Psi_crossed_prod}
that the natural representation $\pi^{\mathrm{can}}$ of $\Cz(\G)\rtimes_{\rho}\Gamma^{2}$
on $\Ltwo(\G)$ restricts to a bijection $\pi^{\Psi}$ of $\Cz(\G)^{\overline{\Psi}\boxtimes\dot{\Psi}}$
onto $\Cz(\GPsi)\subseteq B(\Ltwo(\G))$. Recall also from Theorem~\ref{thm:Psi_comult_crossed_prod}
the morphism $\Delta^{\Psi}\in\Mor(\Cz(\G)^{\overline{\Psi}\boxtimes\dot{\Psi}},\Cz(\G)^{\overline{\Psi}\boxtimes\dot{\Psi}}\tensormin\Cz(\G)^{\overline{\Psi}\boxtimes\dot{\Psi}})$
that satisfies $(\pi^{\Psi}\tensor\pi^{\Psi})\circ\Delta^{\Psi}=\Delta_{\GPsi}\circ\pi^{\Psi}$. 
\begin{prop}
\label{prop:single_functional} Let $\mu\in\CzU(\G)^{*}$. 
\begin{enumerate}
\item \label{enu:single_functional__1}The right convolution operator $\mathscr{R}_{\mu}:\Cz(\G)\to\Cz(\G)$
commutes with the left-right translation action $\Gamma^{2}\overset{\rho}{\curvearrowright}\Cz(\G)$
if and only if $\mu$ is invariant under the adjoint action $\Gamma\overset{\mathrm{ad}^{\mathrm{u}}}{\curvearrowright}\CzU(\G)$. 
\item \label{enu:single_functional__2}If the equivalent conditions above
hold and $\mu$ is in addition a state, then:
\begin{enumerate}
\item \label{enu:single_functional__2_a}Applying the construction of Proposition~\ref{prop:single map}
to $\mathscr{R}_{\mu}$ yields an operator $\mathscr{R}_{\mu}^{\Psi}$
on $\Cz(\G)^{\overline{\Psi}\boxtimes\dot{\Psi}}$, and the operator
$\pi^{\Psi}\circ\mathscr{R}_{\mu}^{\Psi}\circ{(\pi^{\Psi})}^{-1}$
on $\Cz(\GPsi)$ is the right convolution operator of some state $\mu^{\Psi}$
of $\CzU(\GPsi)$.
\item \label{enu:single_functional__2_b}The state $\mu^{\Psi}$ is invariant
under the adjoint action $\Gamma\overset{}{\curvearrowright}\CzU(\G^{\Psi})$. 
\item \label{enu:single_functional__2_c}Repeating the construction of part~\ref{enu:single_functional__2_a}
with $\GPsi,\overline{\Psi},\mu^{\Psi}$ instead of $\G,\Psi,\mu$
and keeping in mind the natural identification of $(\GPsi)^{\overline{\Psi}}$
with $\G$, the new state ${(\mu^{\Psi})}^{\overline{\Psi}}$ equals
$\mu$.
\end{enumerate}
\end{enumerate}
Similar results hold for left convolution operators.

\end{prop}

\begin{proof}
\ref{enu:single_functional__1}. The map $\mathscr{R}_{\mu}$ commutes
with the right action $\Gamma\overset{\rho^{R}}{\curvearrowright}\Cz(\G)$
since $R_{\gamma}\in\Linfty(\hat{\G})'$ for all $\gamma\in\Gamma$
(note again our convention for the right convolution operator $\mathscr{R}_{\mu}$),
so the non-trivial condition comes from the commutation with the left
action $\Gamma\overset{\rho^{L}}{\curvearrowright}\Cz(\G)$. By the
formulas $\mathscr{R}_{\mu}=(\mu\tensor\id)\circ\Delta^{\text{u-r}}$,
(\ref{eq:univ_left_right_translation_actions}), and (\ref{eq:univ_left_right_translation_actions_2}),
for every $\gamma\in\Gamma$ we have $\mathscr{R}_{\mu}\circ\rho_{\gamma}^{L}=\mathscr{R}_{\mu\circ\rho_{\gamma}^{L,\mathrm{u}}}$
and $\rho_{\gamma}^{L}\circ\mathscr{R}_{\mu}=\mathscr{R}_{\mu\circ\rho_{-\gamma}^{R,\mathrm{u}}}$.
Therefore, by the injectivity of the map $\CzU(\G)^{*}\ni\nu\mapsto\mathscr{R}_{\nu}$
\cite[Proposition~8.3 and its proof]{Daws__mult_self_ind_dual_B_alg}
(see also \cite[Theorem~2.1~(a)]{Skalski_Viselter__convolution_semigroups}),
$\mathscr{R}_{\mu}$ commutes with $\rho^{L}$ if and only if $\mu\circ\rho_{\gamma}^{L,\mathrm{u}}=\mu\circ\rho_{-\gamma}^{R,\mathrm{u}}$,
i.e.~$\mu\circ\mathrm{ad}_{\gamma}^{\mathrm{u}}=\mu$, for all $\gamma\in\Gamma$.

\ref{enu:single_functional__2}~\ref{enu:single_functional__2_a}.
The foregoing implies that we are in the situation where we can apply
Proposition~\ref{prop:single map} to the non-degenerate completely
positive map $S:=\mathscr{R}_{\mu}$, yielding the canonical non-degenerate
completely positive map $\tilde{S}:\Cz(\G)\rtimes_{\rho}\Gamma^{2}\to\Cz(\G)\rtimes_{\rho}\Gamma^{2}$
and its restriction $S^{\Psi}=\mathscr{R}_{\mu}^{\Psi}:\Cz(\G)^{\overline{\Psi}\boxtimes\dot{\Psi}}\to\Cz(\G)^{\overline{\Psi}\boxtimes\dot{\Psi}}$
having norm $1$. We verify that the  completely positive map $T:=\pi^{\Psi}\circ S^{\Psi}\circ{(\pi^{\Psi})}^{-1}:\Cz(\GPsi)\to\Cz(\GPsi)$
of norm $1$ satisfies the commutation relation  of Theorem~\ref{thm:SkV_Thm_3_2_C_0}~\ref{enu:SkV_Thm_3_2_C_0__1}
for the locally compact quantum group $\GPsi$, namely that it commutes
with the left convolution operators $\mathscr{L}_{\nu}^{\GPsi}=(\i\tensor\nu)\circ\Delta_{\Cz(\GPsi)}$,
$\nu\in\Cz(\GPsi)^{*}$, of $\GPsi$. By Theorem~\ref{thm:Psi_comult_crossed_prod}
this is equivalent to $S^{\Psi}$ commuting with the operators $(\i\tensor\theta)\circ\Delta^{\Psi}$,
$\theta\in(\Cz(\G)^{\overline{\Psi}\boxtimes\dot{\Psi}})^{*}$, which
\emph{a posteriori} map $\Cz(\G)^{\overline{\Psi}\boxtimes\dot{\Psi}}$
into itself. 

The subspace $\mathbb{O}:=\{\vartheta\circ{(\pi^{\Psi})}^{-1}:\vartheta\in(\Cz(\G)\rtimes_{\rho}\Gamma^{2})^{*}\}$
of $\Cz(\GPsi)^{*}$ is $w^{*}$-dense by the Hahn--Banach theorem
because an element of $\Cz(\G)^{\overline{\Psi}\boxtimes\dot{\Psi}}\subseteq\M(\Cz(\G)\rtimes_{\rho}\Gamma^{2})$
annihilated by all of $(\Cz(\G)\rtimes_{\rho}\Gamma^{2})^{*}$ must
be zero. Thus, by the left version of Lemma~\ref{lem:convolution_oper_converg}~\ref{enu:convolution_oper_converg__1},
it suffices to show that $T$ commutes with the operators $\mathscr{L}_{\nu}^{\GPsi}$
for $\nu\in\mathbb{O}$. In turn, this is equivalent to $S^{\Psi}$
commuting with the operators $(\i\tensor\vartheta|_{\Cz(\G)^{\overline{\Psi}\boxtimes\dot{\Psi}}})\circ\Delta^{\Psi}$,
$\vartheta\in(\Cz(\G)\rtimes_{\rho}\Gamma^{2})^{*}$. By the definition
of $\Delta^{\Psi}$ in Theorem~\ref{thm:Psi_comult_crossed_prod},
it suffices to show that $\tilde{S}$ commutes with the operators
$(\i\tensor\vartheta)\circ\Ad((\lambda^{R}\tensor\lambda^{L})(\Upsilon))\circ\tilde{\Delta}$,
$\vartheta\in(\Cz(\G)\rtimes_{\rho}\Gamma^{2})^{*}$ (see Observation~\ref{obs:slice_maps_nondeg_incl}).
This amounts to the equality 
\begin{equation}
(\tilde{S}\tensor\i)\circ\Ad((\lambda^{R}\tensor\lambda^{L})(\Upsilon))\circ\tilde{\Delta}=\Ad((\lambda^{R}\tensor\lambda^{L})(\Upsilon))\circ\tilde{\Delta}\circ\tilde{S}\label{eq:S_tilde_Delta_tilde_Upsilon}
\end{equation}
(which makes sense because $\tilde{S}$, unlike $S^{\Psi}$, is already
known to be non-degenerate).

The unitary $(\lambda^{R}\tensor\lambda^{L})(\Upsilon)$ discussed
in (\ref{eq:R_L_Upsilon}) satisfies $(\tilde{S}\ot\id)\big((\lambda^{R}\tensor\lambda^{L})(\Upsilon)\big)=(\lambda^{R}\tensor\lambda^{L})(\Upsilon)$;
thus it is in the multiplicative domain of $\tilde{S}\ot\id$, and
it suffices in fact to show that 
\[
(\tilde{S}\ot\id)\circ\tilde{\Delta}=\tilde{\Delta}\circ\tilde{S}.
\]
Let then $f\in\Cc(\Gamma^{2},\Cz(\G))\subseteq\Cz(\G)\rtimes_{\rho}\Gamma^{2}$
and compute on one hand 
\begin{align*}
(\tilde{\Delta}\circ\tilde{S})(f) & =\int_{(\gamma_{1},\gamma_{2})\in\Gamma^{2}}(\iota\ot\iota)(\Delta(S(f(\gamma_{1},\gamma_{2}))))\cdot(\lambda_{\gamma_{1}}^{L}\tensor\lambda_{\gamma_{2}}^{R})\d(\gamma_{1},\gamma_{2})\\
 & =\int_{(\gamma_{1},\gamma_{2})\in\Gamma^{2}}((\iota\circ S)\ot\iota)(\Delta(f(\gamma_{1},\gamma_{2})))\cdot(\lambda_{\gamma_{1}}^{L}\tensor\lambda_{\gamma_{2}}^{R})\d(\gamma_{1},\gamma_{2})
\end{align*}
and on the other 
\[
(\tilde{S}\ot\id)\circ\tilde{\Delta}(f)=\int_{(\gamma_{1},\gamma_{2})\in\Gamma^{2}}(\tilde{S}\ot\id)\big((\iota\ot\iota)(\Delta(f(\gamma_{1},\gamma_{2})))\cdot(\lambda_{\gamma_{1}}^{L}\tensor\lambda_{\gamma_{2}}^{R})\big)\d(\gamma_{1},\gamma_{2}),
\]
which are equal.

\ref{enu:single_functional__2}~\ref{enu:single_functional__2_b}.
By part \ref{enu:single_functional__1} applied to $\GPsi$ in place
of $\G$, we should show that the operator $\pi^{\Psi}\circ\mathscr{R}_{\mu}^{\Psi}\circ{(\pi^{\Psi})}^{-1}$
on $\Cz(\GPsi)$ commutes with the left-right translation action $\Gamma^{2}\overset{}{\curvearrowright}\Cz(\GPsi)$,
which by Proposition~\ref{prop:left_right_transl_act_G_Psi} is equal
to $\pi^{\Psi}\circ\rho_{\gamma_{1},\gamma_{2}}^{\overline{\Psi}\boxtimes\dot{\Psi}}\circ{(\pi^{\Psi})}^{-1}$
at $(\gamma_{1},\gamma_{2})\in\Gamma^{2}$. But $\rho_{\gamma_{1},\gamma_{2}}^{\overline{\Psi}\boxtimes\dot{\Psi}}$
is just the restriction to $\Cz(\G)^{\overline{\Psi}\boxtimes\dot{\Psi}}$
of the automorphism of $\Cz(\G)\rtimes_{\rho}\Gamma^{2}$ given by
$\Ad(\lambda_{\gamma_{1},\gamma_{2}})$, i.e.~induced by the automorphism
$\rho_{\gamma_{1},\gamma_{2}}$ of $\Cz(\G)$. The required commutation
is now clear because of the definition of $\mathscr{R}_{\mu}^{\Psi}$
and the assumption that $\mathscr{R}_{\mu}$ commutes with $\rho_{\gamma_{1},\gamma_{2}}$.

\ref{enu:single_functional__2}~\ref{enu:single_functional__2_c}.
This follows from the second (last) assertion in Proposition~\ref{prop:from_G_Psi_to_G}.
We leave the technical verification to the reader.
\end{proof}
\begin{rem}
The proof of Proposition~\ref{prop:single_functional}~\ref{enu:single_functional__2}~\ref{enu:single_functional__2_a}.~would
become technically simpler if we knew in advance that $S^{\Psi}=\mathscr{R}_{\mu}^{\Psi}$
is non-degenerate (Question~\ref{ques:non_deg_deformed_CP_map}),
for then we could say immediately that our goal is to show that $(T\ot\id)\circ\Delta_{\Cz(\GPsi)}=\Delta_{\Cz(\GPsi)}\circ T$
(see (\ref{eq:right_conv_oper_comm_rel})), which is equivalent to
$(S^{\Psi}\ot\id)\circ\Delta^{\Psi}=\Delta^{\Psi}\circ S^{\Psi}$,
and for this last equality it suffices to verify (\ref{eq:S_tilde_Delta_tilde_Upsilon}).
\end{rem}

\begin{rem}
\label{rem:single_functional_trivial}The condition of Proposition~\ref{prop:single_functional}~\ref{enu:single_functional__1}
holds trivially in the following particular cases:
\begin{enumerate}
\item \label{enu:rem:single_functional_trivial__1}When $\G$ is a locally
compact group $G$ and the support of the measure $\mu$ is contained
in the centraliser $Z_{G}(\Gamma)$ of $\Gamma$ inside $G$. (See
more on this case in Subsection~\ref{subsec:commutative} below.)
\item When $\Gamma\le\G$ factors through the centre $\mathscr{Z}(\G)$,
hence $\GPsi=\G$ by Proposition~\ref{prop:Gamma_fact_thru_center}.
In this case $R_{\gamma}=L_{-\gamma}$, thus $\rho_{\gamma}^{R}=\rho_{-\gamma}^{L}$,
for all $\gamma\in\Gamma$, hence $\mathrm{ad}^{\mathrm{u}}$ is trivial.
\end{enumerate}
\end{rem}

Another trivial case is when $\mu$ is `supported by $\Gamma$'.
In this situation $\mu^{\Psi}$ too is `supported by $\Gamma$'
as the next corollary shows.
\begin{cor}
\label{cor:ext_of_prob_meas_on_Gamma}Suppose that $\G$ is co-amenable
and write $\Pi:\Cz(\G)\to\Cz(\Gamma)$ for the strong quantum homomorphism
associated to $\Gamma$ being a closed quantum subgroup of $\G$.
Let $\nu$ be a probability measure on $\Gamma$ and consider the
state $\mu:=\nu\circ\Pi$ of $\Cz(\G)$. Then $\mu$ satisfies the
condition of Proposition~\ref{prop:single_functional}~\ref{enu:single_functional__1}
and the `deformed' state $\mu^{\Psi}$ of $\Cz(\GPsi)$ is also
`supported by $\Gamma$', i.e., it factors through the strong quantum
homomorphism associated to $\Gamma$ being a closed quantum subgroup
of $\GPsi$.
\end{cor}

Remark that as before, we believe that the co-amenability condition
is not required.
\begin{proof}[Proof of Corollary~\ref{cor:ext_of_prob_meas_on_Gamma}]
The first assertion is obvious.

Using Proposition~\ref{prop:single_functional} as well as Proposition~\ref{prop:G_Psi_co_unit}
and its notation we have 
\begin{equation}
\begin{split}\mu^{\Psi} & =\epsilon_{\GPsi}\circ\pi^{\Psi}\circ\mathscr{R}_{\mu}^{\Psi}\circ{(\pi^{\Psi})}^{-1}=\epsilon_{\Gamma}\circ\pi_{\Gamma}^{\Psi}\circ\Pi^{\Psi}\circ{(\pi^{\Psi})}^{-1}\circ\pi^{\Psi}\circ\mathscr{R}_{\mu}^{\Psi}\circ{(\pi^{\Psi})}^{-1}\\
 & =\epsilon_{\Gamma}\circ\pi_{\Gamma}^{\Psi}\circ\Pi^{\Psi}\circ\mathscr{R}_{\mu}^{\Psi}\circ{(\pi^{\Psi})}^{-1}=\epsilon_{\Gamma}\circ\pi_{\Gamma}^{\Psi}\circ(\Pi\circ\mathscr{R}_{\mu})^{\Psi}\circ{(\pi^{\Psi})}^{-1}.
\end{split}
\label{eq:mu_Psi}
\end{equation}
On the other hand,
\[
\Pi\circ\mathscr{R}_{\mu}=(\mu\tensor\Pi)\circ\Delta_{\G}=(\nu\tensor\i)\circ(\Pi\tensor\Pi)\circ\Delta_{\G}=(\nu\tensor\i)\circ\Delta_{\Gamma}\circ\Pi=\mathscr{R}_{\Gamma,\nu}\circ\Pi,
\]
where $\mathscr{R}_{\Gamma,\nu}$ is the right convolution operator
on $\Cz(\Gamma)$ associated with $\nu$. In conclusion, the state
\[
\mu^{\Psi}=\epsilon_{\Gamma}\circ\pi_{\Gamma}^{\Psi}\circ(\mathscr{R}_{\Gamma,\nu}\circ\Pi)^{\Psi}\circ{(\pi^{\Psi})}^{-1}=\epsilon_{\Gamma}\circ\pi_{\Gamma}^{\Psi}\circ\mathscr{R}_{\Gamma,\nu}^{\Psi}\circ\Pi^{\Psi}\circ{(\pi^{\Psi})}^{-1}
\]
indeed factors through $\pi_{\Gamma}^{\Psi}\circ\Pi^{\Psi}\circ{(\pi^{\Psi})}^{-1}$,
as the second assertion claims.
\end{proof}
\begin{rem}
\label{rem:inv_by_averag}If $\Gamma$ is compact, we can construct
states of $\CzU(\G)$ satisfying the condition of Proposition~\ref{prop:single_functional}~\ref{enu:single_functional__1}
by averaging: if $\nu$ is an arbitrary state of $\CzU(\G)$, then
$\mu(x):=\int_{\Gamma}(\nu\circ\mathrm{ad}_{\gamma}^{\mathrm{u}})(x)\d\gamma$,
$x\in\CzU(\G)$, defines a state $\mu$ of $\CzU(\G)$ that is invariant
under the adjoint action. Furthermore, if $\nu$ is faithful, then
so is $\mu$.
\end{rem}

We now arrive at the first main result of this paper, asserting that
under a mild condition, a convolution semigroup of states on $\G$
can be canonically deformed into a convolution semigroup of states
on the Rieffel deformation $\GPsi$.
\begin{thm}
\label{thm:mainconv}Suppose that ${(\mu_{t})}_{t\geq0}$ is a $w^{*}$-continuous
convolution semigroup of states of $\CzU(\G)$ invariant under the
adjoint action $\Gamma\overset{\mathrm{ad}^{\mathrm{u}}}{\curvearrowright}\CzU(\G)$.
Applying the construction of Proposition~\ref{prop:single map} to
the operators in the $C_{0}$-semigroup ${(\mathscr{R}_{\mu_{t}})}_{t\geq0}$
of right convolution operators on $\Cz(\G)$ yields a $C_{0}$-semigroup
${(\mathscr{R}_{\mu_{t}}^{\Psi})}_{t\geq0}$ on $\Cz(\G)^{\overline{\Psi}\boxtimes\dot{\Psi}}$.
The $C_{0}$-semigroup ${(\pi^{\Psi}\circ\mathscr{R}_{\mu_{t}}^{\Psi}\circ{(\pi^{\Psi})}^{-1})}_{t\geq0}$
on $\Cz(\GPsi)$ is induced by a $w^{*}$-continuous convolution semigroup
of states ${(\mu_{t}^{\Psi})}_{t\geq0}$ of $\CzU(\GPsi)$ (via the
correspondence in Theorem~\ref{thm:SkV_Thm_3_2_C_0}~\ref{enu:SkV_Thm_3_2_C_0__2}).
The states ${(\mu_{t}^{\Psi})}_{t\geq0}$ are invariant under the
adjoint action $\Gamma\overset{}{\curvearrowright}\CzU(\G^{\Psi})$.

Repeating the above construction with $\GPsi,\overline{\Psi},{(\mu_{t}^{\Psi})}_{t\geq0}$
instead of $\G,\Psi,{(\mu_{t})}_{t\geq0}$, the new convolution semigroup
of states of $\CzU((\GPsi)^{\overline{\Psi}}\cong\G)$ equals ${(\mu_{t})}_{t\geq0}$.
\end{thm}

\begin{proof}
Proposition~\ref{prop:single_functional} implies that we are in
the situation where we can apply Proposition~\ref{prop:single map}
to each of the operators ${(\mathscr{R}_{\mu_{t}})}_{t\geq0}$, yielding
a family ${(\mathscr{R}_{\mu_{t}}^{\Psi})}_{t\geq0}$ of operators
on $\Cz(\G)^{\overline{\Psi}\boxtimes\dot{\Psi}}$ such that each
of the operators in ${(\pi^{\Psi}\circ\mathscr{R}_{\mu_{t}}^{\Psi}\circ{(\pi^{\Psi})}^{-1})}_{t\geq0}$
is the right convolution operator of some state $\mu_{t}^{\Psi}$
of $\CzU(\GPsi)$, which is invariant under the adjoint action $\Gamma\overset{\mathrm{ad}^{\mathrm{u}}}{\curvearrowright}\CzU(\G^{\Psi})$.
By the defining formula (\ref{eq:Sformula}), ${(\mathscr{R}_{\mu_{t}}^{\Psi})}_{t\geq0}$
is a semigroup: $\mathscr{R}_{\mu_{0}}^{\Psi}=\id_{\Cz(\G)^{\overline{\Psi}\boxtimes\dot{\Psi}}}$
and $\mathscr{R}_{\mu_{s}}^{\Psi}\circ\mathscr{R}_{\mu_{t}}^{\Psi}=\mathscr{R}_{\mu_{s+t}}^{\Psi}$
for all $s,t\ge0$.

The semigroup ${(\pi^{\Psi}\circ\mathscr{R}_{\mu_{t}}^{\Psi}\circ{(\pi^{\Psi})}^{-1})}_{t\geq0}$
on $\Cz(\GPsi)$ is induced by a convolution semigroup ${(\mu_{t}^{\Psi})}_{t\geq0}$
on $\GPsi$, and we have to show that ${(\mu_{t}^{\Psi})}_{t\geq0}$
is $w^{*}$-continuous. This will be done via a short detour to the
von Neumann algebraic setting, where the topological considerations
are simpler. Recall from Subsection~\ref{subsec:conv_semig_basics}
that each map $\mathscr{R}_{\mu}:\Cz(\G)\to\Cz(\G)$, $\mu\in\CzU(\G)^{*}$,
is actually a restriction of a normal map on $\Linfty(\G)$, and from
Subsection~\ref{subsec:Rieffel_deform_LCQGs} that the left and right
translation actions $\Gamma\curvearrowright\Cz(\G)$ are restrictions
of (von Neumann algebraic) actions on $\Linfty(\G)$ and hence so
is the left-right translation action $\Gamma^{2}\overset{\rho}{\curvearrowright}\Cz(\G)$.
We denote these extensions in the same way. The (extended) operators
${(\mathscr{R}_{\mu_{t}})}_{t\geq0}$ on $\Linfty(\G)$ form a $C_{0}^{*}$-semigroup
and commute with the left-right translation action $\Gamma^{2}\overset{\rho}{\curvearrowright}\Linfty(\G)$.
This $C_{0}^{*}$-semigroup induces -- by elementary von Neumann
algebra theory -- a $C_{0}^{*}$-semigroup ${(\overline{\mathscr{R}}_{\mu_{t}})}_{t\geq0}$
on the von Neumann algebraic crossed product $\Linfty(\G)\rtimes_{\rho}\Gamma^{2}$,
which canonically contains $\M(\Cz(\G)\rtimes_{\rho}\Gamma^{2})$
and thus also $\Cz(\G)^{\overline{\Psi}\boxtimes\dot{\Psi}}$. Evidently,
the maps ${(\overline{\mathscr{R}}_{\mu_{t}})}_{t\geq0}$ on $\Linfty(\G)\rtimes_{\rho}\Gamma^{2}$
extend the maps ${(\mathscr{R}_{\mu_{t}}^{\Psi})}_{t\geq0}$ on $\Cz(\G)^{\overline{\Psi}\boxtimes\dot{\Psi}}$.
By Lemma~\ref{lem:convolution_oper_converg_by_ext} (with $\G$ being
$\GPsi$, $A$ being $\Cz(\G)^{\overline{\Psi}\boxtimes\dot{\Psi}}$
and $M$ being $\Linfty(\G)\rtimes_{\rho}\Gamma^{2}$), the convolution
semigroup ${(\mu_{t}^{\Psi})}_{t\geq0}$ is $w^{*}$-continuous.

Finally the last statement follows immediately from Proposition~\ref{prop:single_functional}~\ref{enu:single_functional__2}~\ref{enu:single_functional__2_c}. 
\end{proof}
\begin{rem}
A more transparent way to show the $C_{0}$-continuity of ${(\mathscr{R}_{\mu_{t}}^{\Psi})}_{t\geq0}$
is to move everything to the von Neumann algebraic setting, especially
identifying $\Linfty(\GPsi)$ with a subalgebra of $\Linfty(\G)\rtimes_{\rho}\Gamma^{2}$,
and then use \cite[Lemma~2.3]{Skalski_Viselter__convolution_semigroups}
implication (c)$\implies$(a) in lieu of Lemma~\ref{lem:convolution_oper_converg_by_ext}.
The realisation of $\Linfty(\GPsi)$ inside $\Linfty(\G)\rtimes_{\rho}\Gamma^{2}$
is given in \cite[Proposition~5]{Fima_Vainerman__twist_Rieffel_deform},
which is the von Neumann analogue of Theorem~\ref{thm:C0_G_Psi_crossed_prod}.
We have not followed this line of proof in order to make the text,
which takes the \cst-algebraic perspective, as self-contained as
possible.
\end{rem}

Our next goal is to discuss the \emph{symmetry} of the convolution
semigroups of states on $\GPsi$ constructed in Theorem~\ref{thm:mainconv}.
We call a state of the universal \cst-algebra of a locally compact
quantum group \emph{symmetric} if it is invariant under the unitary
antipode. 

For a state $\mu\in\CzU(\G)^{*}$ that is invariant under the adjoint
action of $\Gamma$, we know from Proposition~\ref{prop:single_functional}
that the convolution operators $\mathscr{R}_{\mu},\mathscr{L}_{\mu}$
on $\Cz(\G)$ induce maps $\mathscr{R}_{\mu}^{\Psi},\mathscr{L}_{\mu}^{\Psi}$
on $\Cz(\G)^{\overline{\Psi}\boxtimes\dot{\Psi}}$, and the maps $\pi^{\Psi}\circ\mathscr{R}_{\mu}^{\Psi}\circ{(\pi^{\Psi})}^{-1}$
and $\pi^{\Psi}\circ\mathscr{L}_{\mu}^{\Psi}\circ{(\pi^{\Psi})}^{-1}$
on $\Cz(\GPsi)$ are the right and left convolution operators, respectively,
associated with two, \emph{a priori} possibly different, states on
$\CzU(\GPsi)$. The next proposition shows that, as expected, these
two states coincide, and they are symmetric if $\mu$ is.
\begin{prop}
\label{prop:G_Psi_induced_measures}Assume that $\G$ is co-amenable.
Let $\mu$ be a state of $\CzU(\G)$ that is invariant under the adjoint
action $\Gamma\overset{\mathrm{ad}^{\mathrm{u}}}{\curvearrowright}\CzU(\G)$.
Consider the non-degenerate completely positive maps $\pi^{\Psi}\circ\mathscr{R}_{\mu}^{\Psi}\circ{(\pi^{\Psi})}^{-1}$
and $\pi^{\Psi}\circ\mathscr{L}_{\mu}^{\Psi}\circ{(\pi^{\Psi})}^{-1}$
on $\Cz(\GPsi)$ constructed in Proposition~\ref{prop:single_functional}.
Then 
\begin{equation}
\epsilon_{\GPsi}\circ\pi^{\Psi}\circ\mathscr{L}_{\mu}^{\Psi}\circ{(\pi^{\Psi})}^{-1}=\epsilon_{\GPsi}\circ\pi^{\Psi}\circ\mathscr{R}_{\mu}^{\Psi}\circ{(\pi^{\Psi})}^{-1}\label{eq:G_Psi_induced_measures__1}
\end{equation}
 and 
\begin{equation}
\epsilon_{\GPsi}\circ\pi^{\Psi}\circ\mathscr{R}_{\mu}^{\Psi}\circ{(\pi^{\Psi})}^{-1}\circ\Rant_{\GPsi}=\epsilon_{\GPsi}\circ\pi^{\Psi}\circ\mathscr{L}_{\mu\circ\Rant_{\G}}^{\Psi}\circ{(\pi^{\Psi})}^{-1}=\epsilon_{\GPsi}\circ\pi^{\Psi}\circ\mathscr{R}_{\mu\circ\Rant_{\G}}^{\Psi}\circ{(\pi^{\Psi})}^{-1}.\label{eq:G_Psi_induced_measures__2}
\end{equation}
In particular, if $\mu$ is symmetric, then 
\begin{equation}
\begin{split}\epsilon_{\GPsi}\circ\pi^{\Psi}\circ\mathscr{R}_{\mu}^{\Psi}\circ{(\pi^{\Psi})}^{-1}\circ\Rant_{\GPsi} & =\epsilon_{\GPsi}\circ\pi^{\Psi}\circ\mathscr{R}_{\mu}^{\Psi}\circ{(\pi^{\Psi})}^{-1}\\
 & =\epsilon_{\GPsi}\circ\pi^{\Psi}\circ\mathscr{L}_{\mu}^{\Psi}\circ{(\pi^{\Psi})}^{-1}=\epsilon_{\GPsi}\circ\pi^{\Psi}\circ\mathscr{L}_{\mu}^{\Psi}\circ{(\pi^{\Psi})}^{-1}\circ\Rant_{\GPsi}.
\end{split}
\label{eq:G_Psi_induced_measures__3}
\end{equation}
\end{prop}

\begin{proof}
We consider the following maps:
\begin{itemize}
\item the non-degenerate completely positive maps $\widetilde{\mathscr{R}_{\mu}},\widetilde{\mathscr{L}_{\mu}}$
on $\Cz(\G)\rtimes_{\rho}\Gamma^{2}$ induced by $\mathscr{R}_{\mu},\mathscr{L}_{\mu}$,
which restrict (after extension) to the non-degenerate completely
positive maps $\mathscr{R}_{\mu}^{\Psi},\mathscr{L}_{\mu}^{\Psi}$
on $\Cz(\G)^{\overline{\Psi}\boxtimes\dot{\Psi}}$ and yield the convolution
operators $\pi^{\Psi}\circ\mathscr{R}_{\mu}^{\Psi}\circ{(\pi^{\Psi})}^{-1}$
and $\pi^{\Psi}\circ\mathscr{L}_{\mu}^{\Psi}\circ{(\pi^{\Psi})}^{-1}$
on $\Cz(\GPsi)$, respectively (Proposition~\ref{prop:single_functional});
\item the morphism $\tilde{\Pi}\in\Mor(\Cz(\G)\rtimes_{\rho}\Gamma^{2},\Cz(\Gamma)\rtimes_{\rho_{\Gamma}}\Gamma^{2})$
induced by the strong quantum homomorphism $\Pi:\Cz(\G)\to\Cz(\Gamma)$,
which restricts to a morphism $\Pi^{\Psi}\in\Mor(\Cz(\G)^{\overline{\Psi}\boxtimes\dot{\Psi}},\Cz(\Gamma)^{\overline{\Psi}\boxtimes\dot{\Psi}})$
such that $\epsilon_{\GPsi}=\epsilon_{\Gamma}\circ\pi_{\Gamma}^{\Psi}\circ\Pi^{\Psi}\circ{(\pi^{\Psi})}^{-1}$
(Proposition~\ref{prop:G_Psi_co_unit});
\item the anti-automorphism $\tilde{\Rant}$ of $\Cz(\G)\rtimes_{\rho}\Gamma^{2}$,
which satisfies $\tilde{\Rant}(\Cz(\G)^{\overline{\Psi}\boxtimes\dot{\Psi}})=\Cz(\G)^{\overline{\Psi}\boxtimes\dot{\Psi}}$
and $\tilde{\Rant}|_{\Cz(\G)^{\overline{\Psi}\boxtimes\dot{\Psi}}}={(\pi^{\Psi})}^{-1}\circ\Rant_{\GPsi}\circ\pi^{\Psi}$
(Proposition~\ref{prop:R_G_Psi}).
\end{itemize}
As mentioned in Subsection~\ref{subsec:Rieffel_deform_LCQGs}, $\rho^{L},\rho^{R}$
are the left and right, respectively, quantum homomorphisms associated
with $\Gamma$ being a closed quantum subgroup of $\G$. The formulas
that relate these quantum homomorphisms to $\Pi$ \cite[(2.1) and (2.2)]{Brannan_Chirvasitu_Viselter__act_quo_lat}
and the assumption that $\mu$ is invariant under the adjoint action
of $\Gamma$ imply that $\Pi\circ\mathscr{R}_{\mu}=\Pi\circ\mathscr{L}_{\mu}$.
Therefore, $\Pi^{\Psi}\circ\mathscr{R}_{\mu}^{\Psi}=(\Pi\circ\mathscr{R}_{\mu})^{\Psi}=(\Pi\circ\mathscr{L}_{\mu})^{\Psi}=\Pi^{\Psi}\circ\mathscr{L}_{\mu}^{\Psi}$
. Just as in (\ref{eq:mu_Psi}), by the formula $\epsilon_{\GPsi}=\epsilon_{\Gamma}\circ\pi_{\Gamma}^{\Psi}\circ\Pi^{\Psi}\circ{(\pi^{\Psi})}^{-1}$
we deduce formula (\ref{eq:G_Psi_induced_measures__1}) because 
\[
\begin{split}\epsilon_{\GPsi}\circ\pi^{\Psi}\circ\mathscr{L}_{\mu}^{\Psi}\circ{(\pi^{\Psi})}^{-1} & =\epsilon_{\Gamma}\circ\pi_{\Gamma}^{\Psi}\circ\Pi^{\Psi}\circ\mathscr{L}_{\mu}^{\Psi}\circ{(\pi^{\Psi})}^{-1}\\
 & =\epsilon_{\Gamma}\circ\pi_{\Gamma}^{\Psi}\circ\Pi^{\Psi}\circ\mathscr{R}_{\mu}^{\Psi}\circ{(\pi^{\Psi})}^{-1}=\epsilon_{\GPsi}\circ\pi^{\Psi}\circ\mathscr{R}_{\mu}^{\Psi}\circ{(\pi^{\Psi})}^{-1}.
\end{split}
\]
Also $\mu\circ\Rant_{\G}$ is invariant under the adjoint action of
$\Gamma$ because $\Rant_{\G}$ intertwines the actions $\rho^{L}$
and $\rho^{R}$. Applying (\ref{eq:G_Psi_induced_measures__1}) to
$\mu\circ\Rant_{\G}$ proves the second equality in (\ref{eq:G_Psi_induced_measures__2}).

Observe that $\mathscr{R}_{\mu}\circ\Rant_{\G}=\Rant_{\G}\circ\mathscr{L}_{\mu\circ\Rant_{\G}}$.
This implies that $\widetilde{\mathscr{R}_{\mu}}\circ\tilde{\Rant}=\tilde{\Rant}\circ\widetilde{\mathscr{L}_{\mu\circ\Rant_{\G}}}$,
because for each $f\in\Cc(\Gamma^{2},\Cz(\G))$ we have 
\[
\begin{split}(\widetilde{\mathscr{R}_{\mu}}\circ\tilde{\Rant})\left(\int_{\Gamma^{2}}\iota(f(\gamma_{1},\gamma_{2}))\lambda_{\gamma_{1},\gamma_{2}}\d(\gamma_{1},\gamma_{2})\right) & =\widetilde{\mathscr{R}_{\mu}}\left(\int_{\Gamma^{2}}\lambda_{-\gamma_{2},-\gamma_{1}}\iota(\Rant_{\G}(f(\gamma_{1},\gamma_{2})))\d(\gamma_{1},\gamma_{2})\right)\\
 & =\int_{\Gamma^{2}}\lambda_{-\gamma_{2},-\gamma_{1}}\iota((\mathscr{R}_{\mu}\circ\Rant_{\G})(f(\gamma_{1},\gamma_{2})))\d(\gamma_{1},\gamma_{2})\\
 & =\int_{\Gamma^{2}}\lambda_{-\gamma_{2},-\gamma_{1}}\iota((\Rant_{\G}\circ\mathscr{L}_{\mu\circ\Rant_{\G}})(f(\gamma_{1},\gamma_{2})))\d(\gamma_{1},\gamma_{2})\\
 & =\tilde{\Rant}\left(\int_{\Gamma^{2}}\iota(\mathscr{L}_{\mu\circ\Rant_{\G}}(f(\gamma_{1},\gamma_{2})))\lambda_{\gamma_{1},\gamma_{2}}\d(\gamma_{1},\gamma_{2})\right)\\
 & =(\tilde{\Rant}\circ\widetilde{\mathscr{L}_{\mu\circ\Rant_{\G}}})\left(\int_{\Gamma^{2}}\iota(f(\gamma_{1},\gamma_{2}))\lambda_{\gamma_{1},\gamma_{2}}\d(\gamma_{1},\gamma_{2})\right).
\end{split}
\]
By restricting to $\Cz(\G)^{\overline{\Psi}\boxtimes\dot{\Psi}}$
we obtain $\mathscr{R}_{\mu}^{\Psi}\circ\tilde{\Rant}|_{\Cz(\G)^{\overline{\Psi}\boxtimes\dot{\Psi}}}=\tilde{\Rant}|_{\Cz(\G)^{\overline{\Psi}\boxtimes\dot{\Psi}}}\circ\mathscr{L}_{\mu\circ\Rant_{G}}^{\Psi}$.
Consequently, using the fact that $\epsilon_{\GPsi}$ is invariant
under $\Rant_{\GPsi}$, we get 
\[
\begin{split}\epsilon_{\GPsi}\circ\pi^{\Psi}\circ\mathscr{R}_{\mu}^{\Psi}\circ{(\pi^{\Psi})}^{-1}\circ\Rant_{\GPsi} & =\epsilon_{\GPsi}\circ\pi^{\Psi}\circ\mathscr{R}_{\mu}^{\Psi}\circ\tilde{\Rant}|_{\Cz(\G)^{\overline{\Psi}\boxtimes\dot{\Psi}}}\circ{(\pi^{\Psi})}^{-1}\\
 & =\epsilon_{\GPsi}\circ\pi^{\Psi}\circ\tilde{\Rant}|_{\Cz(\G)^{\overline{\Psi}\boxtimes\dot{\Psi}}}\circ\mathscr{L}_{\mu\circ\Rant_{G}}^{\Psi}\circ{(\pi^{\Psi})}^{-1}\\
 & =\epsilon_{\GPsi}\circ\Rant_{\GPsi}\circ\pi^{\Psi}\circ\mathscr{L}_{\mu\circ\Rant_{G}}^{\Psi}\circ{(\pi^{\Psi})}^{-1}\\
 & =\epsilon_{\GPsi}\circ\pi^{\Psi}\circ\mathscr{L}_{\mu\circ\Rant_{\G}}^{\Psi}\circ{(\pi^{\Psi})}^{-1}.
\end{split}
\]
This completes the proof of (\ref{eq:G_Psi_induced_measures__2}).
Now (\ref{eq:G_Psi_induced_measures__3}) follows easily if $\mu$
is symmetric, that is, $\mu=\mu\circ\Rant_{\G}$.
\end{proof}
The next theorem, which is the second main result of this paper, follows
as a direct corollary of Proposition~\ref{prop:G_Psi_induced_measures}.
We stress again that the co-amenability assumption might be unnecessary.
\begin{thm}
\label{thm:symmetry}Assume that $\G$ is co-amenable. Under the assumptions
of Theorem~\ref{thm:mainconv}, if each state in the convolution
semigroup ${(\mu_{t})}_{t\geq0}$ on $\G$ is symmetric (i.e., $\Rant$-invariant),
then each state in the induced convolution semigroup ${(\mu_{t}^{\Psi})}_{t\geq0}$
on $\GPsi$ is symmetric (i.e., $\Rant_{\GPsi}$-invariant).
\end{thm}

\subsection{The commutative case\label{subsec:commutative}}

In this subsection we focus on the particular case that $\G=G$ is
a locally compact \emph{group}; as before, $\Gamma$ is an abelian
closed subgroup of $G$ and $\Psi$ is a normalised continuous $2$-cocycle
on $\hat{\Gamma}$. We write $\Rieffel G{\Psi}$ for $\GPsi$ and
begin with a few observations.
\begin{observation}
\label{obs:G_Psi_not_comm}In order for the locally compact quantum
group $\Rieffel G{\Psi}$ to be of interest it must not be commutative,
equivalently, its dual $\widehat{\Rieffel G{\Psi}}=\hat{G}_{\Psi}$
should not be co-commutative. This means that, with $\Omega:=(L\tensor L)(\Psi)\in\VN(G)\tensorn\VN(G)$
as usual, we should not have $\hat{\Delta}_{\Omega}=\sigma\circ\hat{\Delta}_{\Omega}$,
where $\sigma$ is the flip operator and $\hat{\Delta}_{\Omega}=(\Ad\Omega)\circ\hat{\Delta}$
as in Subsection~\ref{subsec:cocycle_twisting}. Equivalently, the
unitary $\Omega\sigma(\Omega)^{*}\in\VN(G)\tensorn\VN(G)$ should
not belong to the commutant of $\left\{ \lambda_{g}^{G}\tensor\lambda_{g}^{G}\in\VN(G)\tensorn\VN(G):g\in G\right\} $.
\end{observation}

\begin{observation}
Denote again by $Z_{G}(\Gamma)$ the centraliser of $\Gamma$ inside
$G$. A computation similar to (\ref{eq:Gamma_le_G_Psi}) shows that
the canonical embedding $L_{Z_{G}(\Gamma)}$ of $\Linfty(\widehat{Z_{G}(\Gamma)})=\VN(Z_{G}(\Gamma))$
inside $\Linfty(\hat{G})=\VN(G)$ turns $Z_{G}(\Gamma)$ into a closed
quantum subgroup of $\Rieffel G{\Psi}$. Furthermore, one can repeat
the proof of Proposition~\ref{prop:G_Psi_co_unit} to show that the
restriction map $\Pi_{Z_{G}(\Gamma)}:\Cz(G)\to\Cz(Z_{G}(\Gamma))$
induces a morphism $\Pi_{Z_{G}(\Gamma)}^{\Psi}\in\Mor(\Cz(G)^{\overline{\Psi}\boxtimes\dot{\Psi}},\Cz(Z_{G}(\Gamma))^{\overline{\Psi}\boxtimes\dot{\Psi}})$
and the morphism $\pi_{Z_{G}(\Gamma)}^{\Psi}\circ\Pi_{Z_{G}(\Gamma)}^{\Psi}\circ{(\pi^{\Psi})}^{-1}\in\Mor(\Cz(\Rieffel G{\Psi}),\Cz(\Rieffel{Z_{G}(\Gamma)}{\Psi}=Z_{G}(\Gamma)))$
is the strong quantum homomorphism associated to $Z_{G}(\Gamma)$
being a closed quantum subgroup of $\Rieffel G{\Psi}$, where the
equality $\Rieffel{Z_{G}(\Gamma)}{\Psi}=Z_{G}(\Gamma)$ follows from
Proposition~\ref{prop:Gamma_fact_thru_center}. In addition, the
co-unit $\epsilon_{\Rieffel G{\Psi}}$ of $\Rieffel G{\Psi}$ equals
$\epsilon_{Z_{G}(\Gamma)}\circ\pi_{Z_{G}(\Gamma)}^{\Psi}\circ\Pi_{Z_{G}(\Gamma)}^{\Psi}\circ{(\pi^{\Psi})}^{-1}$.

Consequently, repeating the proof of Corollary~\ref{cor:ext_of_prob_meas_on_Gamma}
shows that for every probability measure $\mu$ on $G$ that is supported
by $Z_{G}(\Gamma)$, the `deformed' state $\mu^{\Psi}$ of $\Cz(\Rieffel G{\Psi})$
given by Proposition~\ref{prop:single_functional}~\ref{enu:single_functional__2}
is also `supported by $Z_{G}(\Gamma)$', i.e., it factors through
the strong quantum homomorphism associated to $Z_{G}(\Gamma)$ being
a closed quantum subgroup of $\Rieffel G{\Psi}$. This strengthens
Remark~\ref{rem:single_functional_trivial}~\ref{enu:rem:single_functional_trivial__1},
and means that when seeking in the commutative case examples of the
results of the previous subsection in which the convolution semigroup
of measures cannot be `transported directly' to the Rieffel deformation,
we should avoid measures supported by the centraliser $Z_{G}(\Gamma)$.
\end{observation}

Remark that if $\Gamma$ is compact, then as a particular case of
Remark~\ref{rem:inv_by_averag} we can take an arbitrary probability
measure $\nu$ on $G$ and integrate $\Gamma\ni\gamma\mapsto\nu\circ\mathrm{ad}_{\gamma}$,
where $\Gamma\overset{\mathrm{ad}}{\curvearrowright}\Cz(G)$ is the
adjoint action, to get an $\mathrm{ad}$-invariant probability measure
$\mu$ on $G$, which can be later used to generate a convolution
semigroup of measures on $G$ associated with the relevant compound
Poisson process and satisfying the assumptions of Corollary~\ref{corx}.
In addition, if $\nu$ has full support, then so has $\mu$. If $\Gamma$
contains the centre $Z=Z(G)$ of $G$, we can replace the condition
that $\Gamma$ is compact by the slightly more general condition that
the quotient $\Gamma/Z$ is compact, in which case the fact that $\Gamma\overset{\mathrm{ad}}{\curvearrowright}\Cz(G)$
factors through the quotient map $\Gamma\to\Gamma/Z$ allows us to
consider the induced map $\Gamma/Z\ni\left[\gamma\right]\mapsto\nu\circ\mathrm{ad}_{\gamma}$
and integrate it (with respect to the Haar measure on $\Gamma/Z$)
to obtain an $\mathrm{ad}$-invariant probability measure on $G$.
\begin{example}
\label{ex:E(2)}Take $G:=\mathrm{\tilde{E}}(2)=\left\{ \left(\begin{smallmatrix}w & z\\
0 & w^{-1}
\end{smallmatrix}\right):z\in\C,w\in\mathbb{T}\right\} $ as a topological subgroup of $\mathrm{GL}(2,\C)$ and $\Gamma:=\left\{ \left(\begin{smallmatrix}w & 0\\
0 & w^{-1}
\end{smallmatrix}\right):w\in\mathbb{T}\right\} \cong\mathbb{T}.$ Then $Z_{G}(\Gamma)=\Gamma$. As explained in the previous paragraph,
averaging allows constructing probability measures on $G$ that are
invariant under the adjoint action $\Gamma\overset{\mathrm{ad}}{\curvearrowright}\Cz(G)$
and have full support and so, in particular, are not supported by
$Z_{G}(\Gamma)$. 

Fixing $w_{1},w_{2}\in\mathbb{T}$, $w_{1}\neq\pm w_{2}$, consider
the bicharacter $\Psi:=\Psi_{w_{1},w_{2}}$ on $\hat{\mathbb{T}}\cong\Z$
given by $\Psi_{w_{1},w_{2}}(n,m):=w_{1}^{n}w_{2}^{m}$, $(n,m)\in\Z^{2}$,
and form the Rieffel deformation locally compact quantum group $\Rieffel G{\Psi}$.
To prove that $\Rieffel G{\Psi}$ is not commutative we should show
that $\Omega\sigma(\Omega)^{*}$ does not belong to the commutant
$\left\{ \lambda_{g}^{G}\tensor\lambda_{g}^{G}\in\VN(G)\tensorn\VN(G):g\in G\right\} '$,
see Observation~\ref{obs:G_Psi_not_comm}. But $\Omega=\lambda_{j(w_{1})}^{G}\tensor\lambda_{j(w_{2})}^{G}$
where $j(w):=\left(\begin{smallmatrix}w & 0\\
0 & w^{-1}
\end{smallmatrix}\right)$ for $w\in\mathbb{T}$, and so $\Omega\sigma(\Omega)^{*}=\lambda_{j(w_{1}w_{2}^{-1})}^{G}\tensor\lambda_{j(w_{2}w_{1}^{-1})}^{G}$,
which is indeed not in the above commutant because $w_{1}w_{2}^{-1}\notin\left\{ 1,-1\right\} $.
\end{example}

At the moment we do not have an example of a locally compact group
$G$ with a non-compact (or more generally, non-compact modulo the
centre of $G$) abelian subgroup $\Gamma$ admitting a probability
measure on $G$ which is invariant under the adjoint action of $\Gamma$
and not supported by the centraliser $Z_{G}(\Gamma)$. The question
of whether or not such an example exists seems to be related to deep
topological group considerations. Uri Bader has informed us that results
of \cite{Bader_Duchesne_Lecureux} imply that when $G$ is an algebraic
group and $\Gamma$ satisfies mild non-compactness assumptions, such
an invariant measure cannot exist. 

Finding non-trivial examples where $\G$ is a `genuine' locally
compact quantum group is quite tricky and requires further work; the
main motivation behind this paper is to look for ways of producing
such examples. One can always take a single state, average it and
then construct the associated compound Poisson process as described
above for classical $\G$, but this is rather straightforward and
less interesting from the analytic point of view.

\section*{Acknowledgments}

We are grateful to Kenny De~Commer and to Sergey Neshveyev for helpful
conversations and correspondences about the subject of this paper;
to Stefaan Vaes for asking a question that prompted us to further
extend the results we had at the time; to Orr M.~Shalit and to Michael
Skeide for answering Question~\ref{ques:point_norm_point_strict}
affirmatively; to Uri Bader for his enlightening remarks concerning
Subsection~\ref{subsec:commutative}; to Mariusz Tobolski and to
Jacek Krajczok for their assistance on Observation~\ref{obs:two_strict_top};
to Alcides Buss for an interesting correspondence on crossed products
and in particular for providing the references in Lemma~\ref{lem:opp_crossed_prod};
and to L{\'e}onard Cadilhac and {\'E}ric Ricard for their help with
a problem of convergence in crossed products. We also thank the referees
for their useful comments. Finally we thank Pawe{\ldash} Kasprzak,
our Rieffel deformation guru, for the ongoing inspiration.

\section*{Funding}

The first author was partially supported by the National Science Centre
(NCN) grant no. 2020/39/I/ST1/01566. 

\bibliographystyle{amsalpha}
\bibliography{RieffelDeformation}

\providecommand{\bysame}{\leavevmode\hbox to3em{\hrulefill}\thinspace}
\providecommand{\MR}{\relax\ifhmode\unskip\space\fi MR }
\providecommand{\MRhref}[2]{%
  \href{http://www.ams.org/mathscinet-getitem?mr=#1}{#2}
}
\providecommand{\href}[2]{#2}
\begin{thebibliography}{DFSW16}

\bibitem[BDL]{Bader_Duchesne_Lecureux}
U.~Bader, B.~Duchesne, and J.~L\'{e}cureux, \emph{Almost algebraic actions of
  algebraic groups and applications to algebraic representations}, Groups Geom.
  Dyn. \textbf{11} (2017), no.~2, 705--738.

\bibitem[BCV]{Brannan_Chirvasitu_Viselter__act_quo_lat}
M.~Brannan, A.~Chirvasitu, and A.~Viselter, \emph{Actions, quotients and
  lattices of locally compact quantum groups}, Doc. Math. \textbf{25} (2020),
  2553--2582.

\bibitem[BEW]{Buss_Echterhoff_Willett__exotic_cr_prod_BC_conj}
A.~Buss, S.~Echterhoff, and R.~Willett, \emph{Exotic crossed products and the
  {B}aum--{C}onnes conjecture}, J. Reine Angew. Math. \textbf{740} (2018),
  111--159.

\bibitem[BuS]{Buss_Sims__opp_alg_groupoid_C_alg}
A.~Buss and A.~Sims, \emph{Opposite algebras of groupoid {C$^*$}-algebras},
  Israel J. Math. \textbf{244} (2021), no.~2, 759--774.

\bibitem[DHP]{Dadarlat_Hirshberg_Phillips__simp_nu_C_alg_not_eq_isom_op}
M.~Dadarlat, I.~Hirshberg, and N.~C. Phillips, \emph{Simple nuclear {C$^*$}-algebras not equivariantly isomorphic to their opposites}, J.
  Noncommut. Geom. \textbf{12} (2018), no.~4, 1227--1253.


\bibitem[Daw$_1$]{Daws__mult_self_ind_dual_B_alg}
M.~Daws, \emph{Multipliers, self-induced and dual {B}anach algebras},
  Dissertationes Math. (Rozprawy Mat.) \textbf{470} (2010), 62 pp.

\bibitem[Daw$_2$]{Daws__CPM_LCQGs_2012}
\bysame, \emph{Completely positive multipliers of quantum groups}, Internat. J.
  Math. \textbf{23} (2012), no.~12, 1250132, 23 pp.

\bibitem[Daw$_3$]{Daws__Categorical_asp_QG_mult_intr_grp}
\bysame, \emph{Categorical aspects of quantum groups: multipliers and intrinsic
  groups}, Canad. J. Math. \textbf{68} (2016), no.~2, 309--333.

\bibitem[DFSW]{Daws_Fima_Skalski_White_Haagerup_LCQG}
M.~Daws, P.~Fima, A.~Skalski, and S.~White, \emph{The {H}aagerup property for
  locally compact quantum groups}, J. Reine Angew. Math. \textbf{711} (2016),
  189--229.

\bibitem[DKSS]{Daws_Kasprzak_Skalski_Soltan__closed_q_subgroups_LCQGs}
M.~Daws, P.~Kasprzak, A.~Skalski, and P.~M. So{\ldash}tan, \emph{Closed quantum
  subgroups of locally compact quantum groups}, Adv. Math. \textbf{231} (2012),
  no.~6, 3473--3501.

\bibitem[DeC$_1$]{DeCommer__PhD}
K.~De~Commer, \emph{Galois coactions for algebraic and locally compact quantum
  groups}, Ph.D. thesis, KU Leuven, 2009, available at
  \url{http://homepages.vub.ac.be/~kdecomme/PhD.html}.

\bibitem[DeC$_2$]{DeCommer__cocycle_twist_CQG}
\bysame, \emph{On cocycle twisting of compact quantum groups}, J. Funct. Anal.
  \textbf{258} (2010), no.~10, 3362--3375.

\bibitem[DeC$_3$]{DeCommer__Galois_obj_cocycle_twist_LCQG}
\bysame, \emph{Galois objects and cocycle twisting for locally compact quantum
  groups}, J. Operator Theory \textbf{66} (2011), no.~1, 59--106.

\bibitem[EKQR]{Echterhoff_Kaliszewski_Quigg_Raeburn__cat_appr_impr_thm_C_dyn_sys}
S.~Echterhoff, S.~Kaliszewski, J.~Quigg, and I.~Raeburn, \emph{A categorical
  approach to imprimitivity theorems for {$C^*$}-dynamical systems}, Mem. Amer.
  Math. Soc. \textbf{180} (2006), no.~850.

\bibitem[EnV]{Enock_Vainerman__deform_Kac_alg_abel_grp}
M.~Enock and L.~Va{\u\shorti}nerman, \emph{Deformation of a {K}ac algebra by an
  abelian subgroup}, Comm. Math. Phys. \textbf{178} (1996), no.~3, 571--596.

\bibitem[FiV]{Fima_Vainerman__twist_Rieffel_deform}
P.~Fima and L.~Vainerman, \emph{Twisting and {R}ieffel's deformation of locally
  compact quantum groups: deformation of the {H}aar measure}, Comm. Math. Phys.
  \textbf{286} (2009), no.~3, 1011--1050.

\bibitem[JNR]{Junge_Neufang_Ruan__rep_thm_LCQG}
M.~Junge, M.~Neufang, and Z.-J. Ruan, \emph{A representation theorem for
  locally compact quantum groups}, Internat. J. Math. \textbf{20} (2009),
  no.~3, 377--400.

\bibitem[KaN]{KalantarNeufang_groups}
M.~Kalantar and M.~Neufang, \emph{From quantum groups to groups}, Canad. J.
  Math. \textbf{65} (2013), no.~5, 1073--1094.

\bibitem[Kas]{Kasprzak__Rieffel_deform_crossed_prod}
P.~Kasprzak, \emph{Rieffel deformation via crossed products}, J. Funct. Anal.
  \textbf{257} (2009), no.~5, 1288--1332.

\bibitem[KSS]{Kasprzak_Skalski_Soltan__can_centr_exa_seq_LCQG}
P.~Kasprzak, A.~Skalski, and P.~M. So{\ldash}tan, \emph{The canonical central
  exact sequence for locally compact quantum groups}, Math. Nachr. \textbf{290}
  (2017), no.~8-9, 1303--1316.

\bibitem[Kus]{Kustermans__LCQG_universal}
J.~Kustermans, \emph{Locally compact quantum groups in the universal setting},
  Internat. J. Math. \textbf{12} (2001), no.~3, 289--338.

\bibitem[KV$_1$]{Kustermans_Vaes__LCQG_C_star}
J.~Kustermans and S.~Vaes, \emph{Locally compact quantum groups}, Ann. Sci.
  \'Ecole Norm. Sup. (4) \textbf{33} (2000), no.~6, 837--934.

\bibitem[KV$_2$]{Kustermans_Vaes__LCQG_von_Neumann}
\bysame, \emph{Locally compact quantum groups in the von {N}eumann algebraic
  setting}, Math. Scand. \textbf{92} (2003), no.~1, 68--92.

\bibitem[Lanc]{Lance}
E.~C. Lance, \emph{Hilbert {$C^{*}$}-modules. {A} toolkit for operator
  algebraists}, London Mathematical Society Lecture Note Series, vol. 210,
  Cambridge University Press, Cambridge, 1995.

\bibitem[Land]{Landstad__dual_thy_cov_sys}
M.~B. Landstad, \emph{Duality theory for covariant systems}, Trans. Amer. Math.
  Soc. \textbf{248} (1979), no.~2, 223--267.

\bibitem[LS]{Lindsay_Skalski__conv_semigrp_states}
J.~M. Lindsay and A.~G. Skalski, \emph{Convolution semigroups of states}, Math.
  Z. \textbf{267} (2011), no.~1-2, 325--339.

\bibitem[MRW$_1$]{Meyer_Roy_Woronowicz__hom_quant_grps}
R.~Meyer, S.~Roy, and S.~L. Woronowicz, \emph{Homomorphisms of quantum groups},
  M\"unster J. Math. \textbf{5} (2012), 1--24.

\bibitem[MRW$_2$]{Meyer_Roy_Woronowicz__quant_grp_tw_tens_prod_C_alg}
\bysame, \emph{Quantum group-twisted tensor products of {C{$^*$}}-algebras},
  Internat. J. Math. \textbf{25} (2014), no.~2, 1450019, 37 pp.

\bibitem[NeT]{Neshveyev_Tuset__deform_C_alg_cocycle_LCQG}
S.~Neshveyev and L.~Tuset, \emph{Deformation of {C$^\ast$}-algebras by
  cocycles on locally compact quantum groups}, Adv. Math. \textbf{254} (2014),
  454--496.

\bibitem[Ped]{Pedersen__book}
G.~K. Pedersen, \emph{{C$^{\ast} $}-algebras and their automorphism groups},
  London Mathematical Society Monographs, vol.~14, Academic Press, Inc.,
  London-New York, 1979.

\bibitem[Pop]{Popa__some_rig_Bernoulli}
S.~Popa, \emph{Some rigidity results for non-commutative {B}ernoulli shifts},
  J. Funct. Anal. \textbf{230} (2006), no.~2, 273--328.

\bibitem[Rie]{Rieffel_deform}
M.~A. Rieffel, \emph{Non-compact quantum groups associated with abelian
  subgroups}, Comm. Math. Phys. \textbf{171} (1995), no.~1, 181--201.

\bibitem[Run]{Runde__unif_cont_LCQG}
V.~Runde, \emph{Uniform continuity over locally compact quantum groups}, J.
  Lond. Math. Soc. (2) \textbf{80} (2009), no.~1, 55--71.

\bibitem[RuV]{Runde_Viselter_LCQGs_PosDef}
V.~Runde and A.~Viselter, \emph{On positive definiteness over locally compact
  quantum groups}, Canad. J. Math. \textbf{68} (2016), no.~5, 1067--1095.

\bibitem[SaS$_1$]{Salmi_Skalski__idemp_states_LCQG}
P.~Salmi and A.~Skalski, \emph{Idempotent states on locally compact quantum
  groups}, Q. J. Math. \textbf{63} (2012), no.~4, 1009--1032.

\bibitem[SaS$_2$]{Salmi_Skalski__idemp_states_LCQG_II}
\bysame, \emph{Idempotent states on locally compact quantum groups {II}}, Q. J.
  Math. \textbf{68} (2017), no.~2, 421--431.

\bibitem[ShS]{Shalit_Skeide__cont_strict_CP_semig}
O.~M. Shalit and M.~Skeide, \emph{Continuity of strict {CP}-semigroups},
  preprint, 2023.

\bibitem[SkV$_1$]{Skalski_Viselter__convolution_semigroups}
A.~Skalski and A.~Viselter, \emph{Convolution semigroups on locally compact
  quantum groups and noncommutative {D}irichlet forms}, J. Math. Pures Appl.
  (9) \textbf{124} (2019), 59--105.

\bibitem[SkV$_2$]{Skalski_Viselter__generating_functionals}
\bysame, \emph{Generating functionals for locally compact quantum groups}, Int.
  Math. Res. Not. IMRN \textbf{2021} (2021), no.~14, 10981--11009.

\bibitem[StZ]{Stratila_Zsido__lectures_vN}
{\cedilla{S}}.~Str{\u{a}}til{\u{a}} and L.~Zsid{\'o}, \emph{Lectures on von
  {N}eumann algebras}, Abacus Press, Tunbridge Wells, England, 1979.

\bibitem[Tak]{Takesaki__book_vol_2}
M.~Takesaki, \emph{Theory of operator algebras. {II}}, Encyclopaedia of
  Mathematical Sciences, vol. 125, Springer-Verlag, Berlin, 2003.

\bibitem[Vae]{Vaes__new_appr_induct_impriv}
S.~Vaes, \emph{A new approach to induction and imprimitivity results}, J.
  Funct. Anal. \textbf{229} (2005), no.~2, 317--374.

\bibitem[VV]{Vaes_Vainerman__extensions_LCQGs}
S.~Vaes and L.~Vainerman, \emph{Extensions of locally compact quantum groups
  and the bicrossed product construction}, Adv. Math. \textbf{175} (2003),
  no.~1, 1--101.

\bibitem[VD]{Van_Daele__LCQGs}
A.~{Van Daele}, \emph{Locally compact quantum groups. {A} von {N}eumann algebra
  approach}, SIGMA Symmetry Integrability Geom. Methods Appl. \textbf{10}
  (2014), paper 082, 41 pp.

\bibitem[Wan]{Wang__deformation}
S.~Wang, \emph{Deformations of compact quantum groups via {R}ieffel's
  quantization}, Comm. Math. Phys. \textbf{178} (1996), no.~3, 747--764.

\bibitem[Wil]{Williams__cros_prod_book}
D.~P. Williams, \emph{Crossed products of {C${^\ast}$}-algebras}, Mathematical
  Surveys and Monographs, vol. 134, American Mathematical Society, Providence,
  RI, 2007.

\end{thebibliography}

\end{document}